\documentclass[twoside,11pt]{article}
\usepackage{blindtext}
\usepackage[abbrvbib]{jmlr2e}

\usepackage{amsfonts} 
\usepackage{mathrsfs}
\usepackage{amsmath} 
\usepackage[utf8]{inputenc}
\usepackage[english]{babel} 
\usepackage{mathtools} 
\usepackage{latexsym}
\usepackage{hyperref} 
\usepackage{color}
\usepackage{enumitem}
\usepackage[symbol]{footmisc}

\newcommand{\effinf}{\tilde{\psi}_{f_0}}
\newcommand{\ap}{$\an-$posterior }
\newcommand{\an}{{\alpha_n}}
\newtheorem{innercustomthm}{Model}
\newtheorem{assumption}[theorem]{Assumption}
\newenvironment{custommodel}[1]
  {\renewcommand\theinnercustomthm{#1}\innercustomthm}
  {\endinnercustomthm}

\newcommand{\ninf}[1]{\| #1 \|_\infty}

\newcommand{\e}{\mathbb{E}}
\newcommand{\N}{\mathbb{N}}
\newcommand{\R}{\mathbb{R}}

\newcommand{\cL}{\mathcal{L}}

\newcommand{\eps}{\varepsilon}

\newcommand{\al}{\alpha}

\newcommand{\cN}{\mathcal{N}}

\newcommand{\given}{\,|\,}
\definecolor{blendedblue}{rgb}{0.2,0.2,0.7}

\renewenvironment{proof}[1]{\par\noindent{\bf #1 \ }}{\hfill\BlackBox\\[2mm]}

\usepackage{lastpage}
\jmlrheading{24}{2023}{1-\pageref{LastPage}}{1/23; Revised
11/23}{12/23}{23-0089}{Alice L'Huillier, Luke Travis, Isma\"el Castillo and Kolyan Ray}

\ShortHeadings{Semiparametric Inference Using Fractional Posteriors}{L'Huillier, Travis, Castillo and Ray}
\firstpageno{1}

\numberwithin{theorem}{section}
\begin{document}
\title{Semiparametric Inference Using Fractional Posteriors}
\author{$^*$, Luke Travis$^\dagger$, Isma\"el Castillo$^*$ and Kolyan Ray$^\dagger$\\
\\
LPSM, Sorbonne Universit\'e$^*$ and Imperial College London$^\dagger$}
\date{}

\author{\name Alice L'Huillier$^*$ \email alice.lhuillier@sorbonne-universite.fr \\
       \addr LPSM, Sorbonne Universit\'e\\
       4, place Jussieu \\
75005, Paris, France
       \AND
       \name Luke Travis$^*$ \email luke.travis15@imperial.ac.uk \\
       \addr Department of Mathematics\\
       Imperial College London\\
       London SW7 2AZ, United Kingdom 
       \AND
       \name Isma\"el Castillo \email ismael.castillo@upmc.fr \\
       \addr LPSM, Sorbonne Universit\'e\\
       4, place Jussieu\\
75005, Paris, France
\AND
       \name Kolyan Ray \email kolyan.ray@imperial.ac.uk \\
       \addr Department of Mathematics\\
       Imperial College London\\
       London SW7 2AZ, United Kingdom 
       }

\editor{}

\maketitle

\begin{abstract}%
We establish a general Bernstein--von Mises theorem for approximately linear semiparametric functionals of fractional posterior distributions based on nonparametric priors. This is illustrated in a number of nonparametric settings and for different classes of prior distributions, including Gaussian process priors. We show that fractional posterior credible sets can provide reliable semiparametric uncertainty quantification, but have inflated size. To remedy this, we further propose a \textit{shifted-and-rescaled} fractional posterior set that is an efficient confidence set having optimal size under regularity conditions. As part of our proofs, we also refine existing contraction rate results for fractional posteriors by sharpening the dependence of the rate on the fractional exponent.
\end{abstract}


%\tableofcontents

\section{Introduction}
\footnotetext[1]{Equal contribution.}
In this work, we establish theoretical guarantees for the {\it fractional} or {\it tempered} or {\it $\alpha_n$-posterior}, which is obtained in a similar way to the usual Bayesian posterior distribution, but with the likelihood raised to a power $\an\in(0,1]$. Suppose that we model data $Y=Y^n$ with a log-likelihood $\ell_n(\eta; Y^n) = \ell_n(\eta)$, and that we assign a prior distribution $\Pi= \Pi_n$ to the parameter $\eta\in S$. 
%The index $n$ can be the number of observations or more generally correspond to the amount of available information. 
The fractional posterior is then defined as
\begin{equation}
\Pi_\an(B | Y^n) = \frac{\int_B e^{\an\ell_n(\eta)} d\Pi(\eta)}{\int_S e^{\an\ell_n(\eta)} d\Pi(\eta)}, \qquad \qquad B \text{ measurable}.   \label{eq:alpha_posterior_definition}
\end{equation}
One interpretation is that $\al_n$ induces a tempering effect: for $\al_n<1$ the contribution of the data in Bayes' formula is downweighted, thus lowering the importance of the data relative to the prior. When $\al_n=1$, this reduces to the usual posterior distribution. We study here the frequentist behaviour of the fractional posterior for \textit{semiparametric inference}, that is when estimating a low-dimensional functional $\psi(\eta)$ of the parameter $\eta$ when the latter is assigned a high- or infinite-dimensional prior. As reflected in our notation, we will allow the power $\al_n$ to possibly depend on $n$.

Fractional posteriors have been used in a wide variety of settings, including Bayesian model selection \citep{O95}, marginal likelihood approximation \citep{FP08}, empirical Bayes methods \citep{MT20} and more recently variational inference \citep{ARC16,HTLC18,BHP18,AR20,MORV22}. One motivation for their use in statistical inference is their greater robustness to possible model misspecification compared to the usual Bayesian posterior. \citet{GvO17} empirically demonstrate that in a misspecified linear regression setting, fractional posteriors can outperform traditional posteriors, motivating their \textit{safe Bayesian} approach \citep{Grunwald_2012,SafeProbability}, which consists of a data-driven choice of $\alpha_n$. The C-posterior \citep{Cposterior} is another special case of the fractional posterior, which has empirically been shown to be more robust to model misspecification than the full posterior in specific examples. \citet{BHW16} argue that within a decision-theoretic framework, fractional posteriors can be viewed as principled ways to update prior beliefs. In particular, under model misspecification, they show that a choice $\alpha_n\neq 1$ may be necessary for good performance. Computationally, fractionally downweighting parallel distributions can also improve sampling convergence and yield faster mixing times \citep{GT95}.

In all cases, the choice of the fractional power $\alpha_n$, often termed the \textit{learning rate}, plays a key role. There are many proposals for picking this (see, e.g., \citealp{Grunwald_2012,GvO17,HW17,LHW19,SM19}), each aiming to achieve a different target. However, one common and major motivation for using generalized Bayesian methods is to provide uncertainty quantification via the use of generalized posterior credible sets, whose performance are sensitive to the choice of $\alpha_n$ in practice \citep{WM20}. This motivates our work, whose main contribution is to obtain a precise theoretical characterization of the role of $\alpha_n$ for some widely-used Bayesian nonparametric priors, in particular Gaussian processes and histograms. More precisely, for these common high- and infinite-dimensional priors, we obtain nonparametric convergence rates and semiparametric Bernstein-von Mises theorems having the correct dependence on both $n$ and $\alpha_n$. We further use these insights to construct rescaled credible sets from the $\alpha_n$-posterior that are optimal from an information-theoretic perspective for uncertainty quantification.

To both gain some intuition for the results ahead and relate these to the existing literature, consider the simple parametric example where we observe $Y_1,\dots,Y_n \sim^{iid} \cN(\theta,1)$ with a conjugate prior $\Pi=\cN(\mu,\sigma^2)$ for $\theta$. A direct calculation yields the fractional posterior
\begin{equation}\label{eq:conjugate}
\Pi_{\al_n}[\cdot\given Y_1,\dots,Y_n]=\cN \left( \frac{n\alpha_n \bar{Y}_n + \mu \sigma^{-2}}{n\al_n + \sigma^{-2}} ,\frac{1}{n\al_n + \sigma^{-2}} \right) \approx \cN \left( \hat{\theta}_{\text{MLE}} , \frac{1}{n\alpha_n}I_0^{-1} \right),
\end{equation} 
where in this model the MLE equals the sample mean $\hat{\theta}_{\text{MLE}} = \bar{Y}_n$, the Fisher information $I_0 = 1$ and the last (Bernstein--von Mises) approximation holds as $n\alpha_n \to \infty$. Observe that (i) the $\alpha_n$-posterior above can be obtained from the original posterior by replacing $n$ by the \textit{effective sample size} $n\alpha_n$, so that the tempering effect means one effectively only uses $n'=n\al_n$ of the data -- with the exception that $\bar{Y}_n$ in the centering remains identical. Second, (ii) the posterior variance scales as $(n\al_n)^{-1}$ for large $n$ and hence the diameter of a credible set constructed from two-sided $\al_n$--posterior quantiles is enlarged by a multiplicative factor of order $1/\alpha_n$ compared to the traditional posterior. Third, (iii) the choice of $\alpha_n$ does not asymptotically affect the location of the $\alpha_n$-posterior mean. Combined with (ii), this implies that credible sets from the $\alpha_n$-posterior do not have the correct frequentist coverage asymptotically, being conservative (too large). In view of these observations, our main results can be heuristically summarized as implying that for $\alpha_n$-posteriors based on common Bayesian \textit{nonparametric} priors in the well-specified setting:
\begin{enumerate}
\item The $\alpha_n$-posterior contraction rate is of the same form as the full posterior contraction rate, but with the sample size $n$ replaced by the effective sample size $n' = n\alpha_n$.
\item For semiparametric Bayesian inference involving sufficiently regular low-dimensional functionals, a Bernstein--von Mises distributional approximation holds as in \eqref{eq:conjugate}.
\item Under regularity conditions, suitably \textit{rescaled} credible sets from the $\alpha_n$-posterior have asymptotically correct frequentist coverage and information-theoretic optimal diameter, and are thus efficient confidence sets (unlike the standard credible sets).
\end{enumerate}

The Bernstein--von Mises (BvM) distributional approximation in \eqref{eq:conjugate} has been extended for the $\alpha_n$-posterior to general regular low-dimensional parametric models \citep{miller21,MORV22}. However, such proof techniques do not extend to the present semiparametric setting, where one wishes to estimate a finite-dimensional functional in the presence of a high- or infinite-dimensional prior, such as a Gaussian process. In Section \ref{sect:sp-bvm}, we derive analogous semiparametric BvM results to \eqref{eq:conjugate} for the $\alpha_n$-posterior by building on the ideas of \citet{CR2015}. We apply these results to the concrete examples of density estimation and the nonparametric Gaussian white noise model, illustrating our results using histogram and Gaussian process priors, including for the standard Mat\'ern and squared exponential covariance kernels.
Since the $\alpha_n$-posterior variance inflates the usual posterior variance, the resulting credible sets can be much larger than needed leading to conservative uncertainty quantification in the well-specified setting. In Section \ref{sect:corrected_credible_sets}, we further show that suitably rescaled credible sets can correct for this, yielding optimal (efficient) uncertainty quantification and potentially mitigating one of the downsides of fractional posterior inference.

Unlike semiparametric BvM results, nonparametric contraction rates for $\alpha_n$-posteriors have previously been studied in the literature. When the model is well-specified, these often require weaker conditions for convergence and sometimes lead to simpler proofs compared to the usual posterior. A remarkable result is that when $\al_n<1$, testing or metric entropy conditions which are typically needed for deriving posterior convergence rates as in \citet{GGV,vdv_FNBI} are not needed for the fractional posterior, at least when convergence is expressed in terms of certain information-theoretic distances such as R\'eyni--divergences. This was established  by \citet{tongz06} (and was earlier obtained for consistency by \citealp{walkerhjort01}), see also \citet{kruijervdv13,BPY,GM20} for related results and examples (we refer to \citealp{vdv_FNBI}, Chapters 6 and 8, for further results and historical notes). This means that using fractional posteriors often allows one to broaden the set of priors or models for which desirable properties are obtained compared to usual posteriors, as only a prior mass condition is needed, avoiding sometimes delicate constructions with sieve sets in order to keep entropies under control.
However, the works \citet{tongz06, kruijervdv13, BPY,GM20} do not seek to obtain a sharp dependence of the rate on $\al_n$, and  do not yield sharp results in the norms we are interested in, see Section \ref{sec:contraction} for more discussion. In particular, we show that %using the testing-approach of \cite{GGV,vdv_FNBI}, 
one can recover the heuristic idea that the fractional posterior uses $n' = n\alpha_n$ fraction of the data. Such sharp nonparametric contraction rates for the $\alpha_n$-posterior in terms of both $n$ and $\alpha_n$ are needed to obtain precise semiparametric BvM results.

In this paper, we restrict to well-specified nonparametric models. Compared to parametric models, nonparametric models attempt to be sufficiently broad that model misspecification is unlikely, so that the well-specified case covers a far larger set of situations. There are nonetheless important notions of nonparametric model misspecification (e.g. \citealp{vdv_FNBI}, Chapter 8.5) that will be dealt with in future work. Note that the choice $\alpha_n >1$, which is not covered by our results, is also used in the literature, for instance in variational inference \citep{ARC16,BHP18} and distributed Bayesian computation \citep{SzvZ19}.
Finally, the fractional posterior is a special case of a Gibbs posterior \citep{GibbsPosterior}, where one replaces the log-likelihood with (the negative of) a risk function, and with a multiplicative constant $\lambda$, also called inverse-temperature parameter, playing the role of $\al_n$. Gibbs posteriors appear naturally in the study of PAC-Bayesian bounds, see \citet{Catoni_2004,Catoni_2007} and the recent overview  by \citet{alquier_2021}. Although we focus here on the special case of the log-likelihood, it would be interesting to also investigate similar questions as in the present paper for $\lambda$.

\textbf{Outline.}
In this paper, we investigate the behaviour of fractional posteriors both for functionals of infinite-dimensional models (the semiparametric problem, see Sections \ref{sect:sp-bvm} and \ref{sect:corrected_credible_sets}) and for contraction rates of the overall unknown parameter (the nonparametric problem, see Section \ref{sec:contraction}). We start each main section by a general result valid under fairly generic conditions, which we then apply to specific models and priors. In particular, we will consider three main example--cases: the nonparametric Gaussian white noise model, density estimation with random histogram priors, and density estimation with exponentiated Gaussian process priors.

 In Section \ref{sect:sp-bvm}, we study the semiparametric problem and investigate the distribution induced from the  fractional posterior on a functional $\psi(\eta)$, where $\eta$ is an infinite-dimensional parameter. We show that under certain conditions, the fractional posterior distribution of $\sqrt{n\an}(\psi(\eta) - \hat {\psi})$, with $\hat{\psi}$ an efficient estimator of $\psi$, converges to a normal distribution with variance equal to the efficient information bound for estimating the functional. In some cases, the conditions for this to hold differ slightly from those needed for the classical posterior with $\al_n=1$ studied in \citet{CR2015}. Although this posterior asymptotic normality (which we shall call the $\al_n$--BvM result) is of interest in itself, it also implies that credible sets from the $\al_n$--posterior are length--inflated by a factor $1/\sqrt{\an}$ compared to the case $\al_n=1$, giving them large (conservative) coverage but making them inefficient. In Section \ref{sect:corrected_credible_sets}, we study the frequentist coverage properties of a shifted--and--dilated version of the $\al_n$--credible sets. Under an appropriate condition on the centering of the $\al_n$--BvM result, which can always be verified if $\al_n$ is bounded from below, the transformed credible set is shown to be an asymptotically optimal credible set, thereby remedying this issue. We show that when $\al_n$ may go to zero, this is no longer necessarily the case, and assessing coverage becomes more delicate.

Nonparametric contraction rates are studied in Section \ref{sec:contraction}. We first obtain a generic result for the contraction rate of the $\alpha_n$-posterior in terms of a R\'enyi divergence and under a prior mass condition only, slightly sharpening the recent result by \citet{BPY}. We then show that under further entropy conditions \citep{vdv_FNBI}, one can improve this rate in certain regimes of $\alpha_n$, in particular deriving the expected nonparametric rate with $n$ replaced by the effective sample size $n'=n\alpha_n$, thereby generalising the very specific one--dimensional Gaussian example above to the infinite--dimensional setting. We also briefly discuss supremum--norm contraction rates, and show that the above message still holds.
 
Our results are investigated in the three concrete example settings mentioned above. Note that we restrict to these settings for simplicity of exposition, but that our results can be applied much more broadly to settings where the semiparametric BvM tools discussed in the next sections can be deployed, which includes  contexts as different as inverse problems \citep{rn22}, survival analysis  \citep{cv21}, inference for diffusions \citep{nr20}, causal inference \citep{RVV20}, etc.  We also perform simulations which confirm that the derived asymptotic theoretical properties are empirically relevant and observable at reasonable finite sample sizes: in particular, we illustrate that the modified credible sets have close to optimal coverage already at moderate sample size.

\textbf{Framework and notation.} Throughout the paper, we consider the following general setting. Let $(\mathcal{Y}^n, \mathcal{A}^n, P^n_{\eta} : \eta \in S)$ be a sequence of statistical experiments indexed by a parameter $\eta$, where $Y=Y^n$ are the observations, $S$ is a metric measure space, and $n$ is an indexing parameter quantifying the available amount of information. For each $n \in \N$ and $\eta \in S$, we assume that $P^n_{\eta}$ admits a density $p^n_{\eta}$ relative to a $\sigma$-finite measure $\mu^n$ defined on the measurable space $(\mathcal{Y}^n, \mathcal{A}^n)$.

Throughout the following, we make a number of notational simplifications, enumerated here. We write
$P^n_{\eta_0}=:P_0$ for the probability under the true parameter $\eta_0$,
 $E^n_{\eta_0}=:E_0$ for the corresponding expectation under $P_0$,
 $o_{P^n_{\eta_0}}(1) =: o_P(1)$ for a term which is $o(1)$ in $P_0-$probability,  
 $\Pi_n =: \Pi$ for a prior which may depend on $n$,
 $\Pi_{\alpha_n}(\cdot |Y^n)$ for the \ap distribution,
 and $E_{\alpha_n}(\cdot |Y^n)$ for the expectation with respect to the \ap.

We study frequentist properties of the $\alpha_n$--posterior distribution as $n\rightarrow \infty$, that is assuming the observation $Y$ is distributed according to $P^n_{\eta_0}$ for some true value of the parameter $\eta_0$. We consider the regime $n\to\infty$ with $\alpha_n \in(0,1]$ such that $n'=n\alpha_n \to\infty$, with further conditions on $\alpha_n$ required for some results. The condition $n' \to\infty$ is minimal for asymptotic results given the interpretation of $n'$ as the effective sample size used by the fractional posterior, see \eqref{eq:conjugate}. Of particular interest is the regime $\alpha_n \to 0$, since several existing results in the literature hold for ``$\alpha$ small enough'', for instance robustness to misspecification of both fractional posteriors \citep{GvO17} and their variational approximations \citep{MORV22}.

\section{Semiparametric Bernstein-von Mises Theorems}\label{sect:sp-bvm}

Using a nonparametric statistical model provides generality and flexibility, and global nonparametric rates for fractional posteriors will be discussed in Section \ref{sec:contraction}. Even in this general setting, it is often the case statisticians are interested in estimating a finite-dimensional parameter or aspect of the model, the so-called {\em semiparametric problem}. Perhaps the simplest example is, say in density estimation to fix ideas, the problem of estimating a linear functional $\int_0^1 a f$ of the unknown density $f$, where $a$ is a given square-integrable function (e.g. the indicator of an interval). We have seen that in the simple one-dimensional example in the introduction, the $\al_n$--posterior gives a distribution that is inflated by a factor of size roughly $1/\sqrt{\al_n}$ compared to the classical posterior. In this section, we will show that this in fact corresponds to a general phenomenon which carries over to estimation of many semiparametric functionals. As mentioned earlier, we allow $\al_n\in(0,1]$ to depend on $n$ and assume $n\al_n\to\infty$ as $n\to\infty$. 

More precisely, given a functional $\psi : S \rightarrow \R$ of interest,  we wish to study the properties of the marginal $\alpha_n$--posterior distribution of $\psi(\eta)$, i.e the push-forward measure of the $\alpha_n$--posterior defined by \eqref{eq:alpha_posterior_definition} through the map $\psi$. We first consider a fairly general setting and introduce sufficient conditions for the posterior distribution to be asymptotically Gaussian (in a sense given in the next paragraph) with an optimal (efficient) variance. Afterwards, we apply this general result to the Gaussian white noise model and density estimation.

We say that a distribution $Q_Y$ on $\mathbb{R}$, depending on the data $Y$, \textit{converges weakly in $P_0$-probability to a Gaussian distribution $\cN(0,V)$}, {\color{black} denoted $Q_Y  \leadsto \cN(0,V)$} if, as $n\to\infty$,
\begin{equation}\label{cvl}
 d_{BL}\left(Q_Y, \cN(0,V) \right) \to^{P_0} 0, 
\end{equation} 
where $d_{BL}$ is the bounded Lipschitz distance between probability distributions on $\mathbb{R}$ (the latter distance metrises weak convergence, see Chapter 11 of \citealp{D02}). In the sequel, we take $Q_Y$ to be a re-centered and re-scaled version of the $\al_n$--posterior distribution induced on the functional $\psi(\eta)$.   More precisely, given a {\em rate}  $v_n$ and a {\em centering} $\mu=\mu(Y)$, consider the map $\tau_\psi:\eta \to v_n(\psi(\eta)-\mu)$. Below we will say that the $\al_n$--posterior distribution of $v_n(\psi(\eta)-\mu)$ converges weakly in $P_0$--probability to a $\cN(0,V)$ distribution if \eqref{cvl} holds for 
\[ Q_Y = \Pi_{\al_n}[\cdot\given Y] \circ \tau_\psi^{-1},\]
that is, for the push-forward measure of the $\al_n$--posterior through $\tau_\psi$. To establish \eqref{cvl}, one can, for instance, verify that Laplace transforms converge in $P_0$--probability, see \citet{CR2015} for details. 

When $v_n = \sqrt{n\alpha_n}$ and $\mu = \hat\psi$ is an efficient estimator of $\psi(\eta)$, writing $\mathcal{L}_\an( \sqrt{n\alpha_n}(\psi(\eta)-\hat\psi) |Y)$ for the marginal $\alpha_n$-posterior distribution of $\sqrt{n\alpha_n}(\psi(\eta)-\hat\psi)$, the above says that
$$\mathcal{L}_\an( \sqrt{n\alpha_n}(\psi(\eta)-\hat\psi) |Y) \approx \cN(0,V)$$
as $n\to\infty$. Such a result, known as a \textit{semiparametric BvM theorem}, says that the above marginal $\alpha_n$-posterior distribution asymptotically converges to a Gaussian distribution, with the precise form of convergence defined via \eqref{cvl}. It is perhaps more intuitive to express this distributional approximation as $\mathcal{L}_\an( \psi(\eta) |Y) \approx \cN(\hat\psi, V/(n\alpha_n))$, mirroring the conjugate example \eqref{eq:conjugate}. Recall that we assume there is a true $P_0 = P_{\eta_0}^n$ generating the data and we are taking the large-sample frequentist limit $n\to\infty$.

\subsection{A generic LAN setting}

Recall the log-likelihood is denoted by $\ell_n(\eta) = \log p^n_\eta(Y^n)$ and we write $o_P(1)$ as a shorthand for $o_{P_0}(1)=o_{P_{\eta_0}}(1)$. The following setting formalises a generic semiparametric framework as in \citet{CR2015} (see also \citealp{c12} and \citealp{vdv_FNBI}, where similar settings are considered in order to derive BvM theorems). A main difference is in the control of remainder terms, which here depend on $\al_n$ (one recovers the conditions of \citealp{CR2015} when $\al_n=1$).
%\sbl{We comment on this framework below.}

\begin{assumption}\label{ass:expansion_assumption}
Let $(\mathcal{H}, \langle\cdot, \cdot\rangle_L)$ be a Hilbert space with associated norm $\|\cdot \|_L$. In the following, $R_n$ and $r$ are remainder terms which are controlled through the last part of the assumption.

{\em LAN expansion.} Suppose  the log-likelihood around $\eta_0$ can be written, for suitable $\eta$'s to be specified below, as
$$
\ell_n(\eta) = \ell_n(\eta_0) -\frac{n}{2}\|\eta - \eta_0\|_L^2  + \sqrt{n}W_n(\eta - \eta_0) + R_n(\eta, \eta_0),
$$
where $W_n : h \mapsto W_n(h)$ is $P_0^n-$almost surely a linear map and $W_n(h)$ converges weakly to $\cN(0, \| h\|_L^2 )$ as $n \rightarrow \infty$.% and $R_n$ is a reminder term whose behaviour will be controlled below. 

{\em Functional expansion.} Suppose that the functional $\psi$ around $\eta_0$ can be written, for some $\psi_0\in\mathcal{H}$, as %and $r(\eta, \eta_0)$ to be controlled below,
$$
\psi(\eta) - \psi(\eta_0) =\langle \psi_0, \eta - \eta_0 \rangle_L + r(\eta, \eta_0).
$$
Define, for any fixed $t\in\R$, a {\em path} through $\eta$ as
\begin{equation}
\eta_t = \eta - \frac{t\psi_0}{\sqrt{n\an}}.\label{def:eta_perturbed}% \hspace{5mm} \textrm{for any } t \in \mathbb{R}. 
\end{equation}

{\em Remainder terms control.} 
Suppose that there exists a sequence of measurable sets $A_n$ satisfying
\[ \Pi_\an[A_n | Y^n] = 1 + o_P(1), \]
	such that $\eta - \eta_0 \in \mathcal{H}$ for all $\eta \in A_n$ and $n$ sufficiently large, and for any fixed $t \in \mathbb{R}$,
$$
\sup_{\eta \in A_n}|t\sqrt{n\an} r(\eta, \eta_0) + \an(R_n(\eta, \eta_0) - R_n(\eta_t, \eta_0 ))| = o_P(1).
$$
\end{assumption}

For $\psi_0$ and $W_n$ as in Assumption \ref{ass:expansion_assumption}, further define,
\begin{align}
\hat{\psi} &= \psi(\eta_0) + \frac{W_n(\psi_0)}{\sqrt{n}}, \hspace{5mm} V_{0} = \left|\left| \psi_0 \right| \right|^2_L.	\label{def:psi_hat}
\end{align}
The term $V_0$ is the {\em efficiency bound} for estimating $\psi(\eta_0)$; an estimator $\tilde{\psi}=\tilde{\psi}(Y)$ is said to be {\em linear efficient} for estimating $\psi(\eta_0)$ if it can be expanded as $\tilde{\psi}=\psi(\eta_0) + W_n(\psi_0)/\sqrt{n}+o_P(1/\sqrt{n})$ or equivalently if $\sqrt{n}(\tilde{\psi}-\hat\psi)=o_P(1)$. For such an estimator,  $\sqrt{n}(\tilde{\psi}-\psi(\eta_0))$ converges in distribution to a $\cN(0,V_0)$ variable. Note that $\hat\psi$ is itself not an estimator as it depends on unknown quantities. But in all the following limiting results at rate $1/\sqrt{n}$ or $1/\sqrt{n\al_n}$, this quantity can be replaced by any linear efficient estimator $\tilde\psi$ since $\tilde{\psi}=\hat\psi+o_P(1/\sqrt{n})$.

{\color{black}
{\em Interpretation of Assumption \ref{ass:expansion_assumption}.} The first condition requires that the log-likelihood expands around $\eta_0$ as the sum of a negative quadratic term, a stochastic term and a remainder term. This type of Local Asymptotic Normality assumption is reminiscent of the classical LAN expansion in parametric models (see e.g. \citealp{aad98}, Chapter 7); the main difference is that here in the (more general) nonparametric setting, we require a control of remainder terms on typically larger neighborhoods. While in smooth parametric models the LAN expansion is formulated in a $1/\sqrt{n}$--neighborhood of the truth, $A_n$ in  Assumption \ref{ass:expansion_assumption} will generally be chosen as a set on which the posterior for $\eta$ concentrates; since the present setting is nonparametric, the diameter of this set is typically a nonparametric convergence rate that is slower than $1/\sqrt{n}$. Finally, Assumption \ref{ass:expansion_assumption} involves the functional $\psi(\eta)$ and requires that it can be expanded around the true value $\psi(\eta_0)$ in a way that is `compatible' with the LAN--inner product. These assumptions are later verified for several classes of priors in white noise regression and density estimation for a broad range of $\al_n$ values. More generally, we expect Assumption \ref{ass:expansion_assumption} to hold in a wide variety of setting. For instance, in the case $\al_n=1$, since they were introduced in \cite{CR2015}, these assumptions have been verified in diffusion models \citep{nr20}; inverse problems \citep{NS_19, rn20, rn22}; survival models \citep{cv21}; the Cox model \citep{c12, nc23}; and causal inference \citep{RVV20} amongst others.
}

\subsection{General BvM Theorems}

With Assumption \ref{ass:expansion_assumption}, we can prove a general BvM type result for the \ap distribution of $\psi(\eta)$. {\color{black} For the statement below, the conditional expectation in the display is $E[G(\eta)\given A_n]=\int_{A_n} G(\eta)dP(\eta)/P(A_n)$, applied here with $P=\Pi_{\al_n}[\cdot\given Y^n]$ the $\al_n$--posterior distribution and $G$ the specific exponential function of $\eta$ appearing in the display.} 
%, this is Theorem \ref{thm:general_bvm_ap} below. %In the sections that follow we will provide results which help one to verify Assumption \ref{ass:expansion_assumption} in particular settings and with particular priors. 
\begin{theorem}[Semiparametric BvM for the \ap ]\label{thm:general_bvm_ap}
Let $\Pi = \Pi_n$ be a prior distribution on $\eta$ and suppose that Assumption \ref{ass:expansion_assumption} holds with sets $A_n$.
Then for any $t\in\R$,
$$E_{\alpha_n}(e^{t\sqrt{n\an}(\psi(\eta) - \hat{\psi})} | Y^n, A_n) = e^{o_P(1) + t^2V_0/2} \cdot \frac{\int_{A_n} e^{{\alpha_n}\ell_n(\eta_t)}d\Pi(\eta)}{\int e^{{\alpha_n}\ell_n(\eta)}d\Pi(\eta)},$$
where $E_{\alpha_n}$ denotes expectation with respect to the ${\alpha_n}-$posterior. 
%{\color{green}[K: what does $|A_n$ above mean? Should there be $A_n$ in the denominator of the Bayes formula?]}
Furthermore, if for any $t\in\R$,
$$\frac{\int_{A_n} e^{{\alpha_n}\ell_n(\eta_t)}d\Pi(\eta)}{\int e^{{\alpha_n}\ell_n(\eta)}d\Pi(\eta)} = 1+o_P(1),$$
then the \ap distribution of $\sqrt{n{\alpha_n}}(\psi(\eta) - \hat{\psi})$ converges weakly in $P_0-$probability to a Gaussian distribution with mean 0 and variance $V_{0}$.

%{\color{green}[I think you only need that the above convergence holds for $t$ in a neighbourhood of 0, not for all $t\in\R$. Is this worth mentioning?]}
\end{theorem}

%%%%%%

%{\color{red}Do we want to discuss the change of measure condition?}
The last display of Theorem \ref{thm:general_bvm_ap} is a ``change-of-measure"--type condition. It is satisfied if a small additive perturbation of the prior (replacing $\eta$ by $\eta_t$ or vice-versa) has little effect on computing the integrals in the display. It can often be checked by doing a change of measure in the prior, see e.g. \citet{c12} and \citet{CR2015}.

We now apply this general result to the following two prototypical nonparametric models, which will serve as concrete examples for our main results here and in Section \ref{sec:contraction}.

\begin{custommodel}{(GWN)}[Gaussian white noise]\label{def:GWN}
For $f \in L^2[0,1]$, one observes the trajectory $Y^n = (Y^n(t): t \in [0,1])$  
$$
dY^n(t) = f(t)dt + \frac{1}{\sqrt{n}}dB(t), \hspace{5mm} t \in [0,1],
$$
where $B(t)$ is a standard Brownian motion.
For $(\phi_k)_{k \geq 1}$ any orthonormal basis of $L^2[0,1]$, it is statistically equivalent to observe the subprocess $(Y_{k}^n = \int_0^1 \phi_k(t) dY^n(t): k \geq 1)$ acting on this basis. In particular, the problem can be rewritten as observing $Y^n={(Y^n_k)}_k$ with
$$
Y^n_k = f_k + \frac{1}{\sqrt{n}} \eps_k, \hspace{5mm}k \geq 1,
$$
where $f_k = \int_0^1 f(t) \phi_k(t) dt$ and $\eps_k \sim^{iid} \cN(0,1)$.
\end{custommodel}

The Gaussian white noise model is the continuous analogue of nonparametric regression with fixed or uniform random design \citep{R08}. It is a standard approach in statistical theory to instead consider this model \citep{J19}, which behaves asymptotically identically to nonparametric regression, but simplifies certain technical arguments due to the discretization. Commonly used priors for $f$ are series priors and Gaussian process priors, see below for specific examples.

\begin{custommodel}{(D)}[Density estimation]\label{def:density_estimation}
For $f$ a probability density with respect to Lebesgue measure on the interval $[0, 1]$, one observes $Y=Y^n=(Y_1, \dots, Y_n)$ with $Y_1,\dots,Y_n\sim^{iid} f$.
\end{custommodel}

Many different priors have been used for density functions; for example histograms, P\'olya trees, mixture models and logistically transformed priors, see the monograph by \citet{vdv_FNBI}. Here we will focus on two large classes: random histograms and exponentiated Gaussian processes.

Although for clarity of exposition we focus on these prototypical models, our techniques extend to others. The results from this section require the form of local asymptotic normality (LAN) described in Assumption \ref{ass:expansion_assumption}, which is expected in order to derive asymptotic normality results, while the nonparametric results from Section \ref{sec:contraction} only require a prior mass condition in the minimal case.

%%%%%%
{\color{black} {\it Gaussian White Noise.} In Model \ref{def:GWN}, the likelihood admits a LAN expansion, with $\eta=f$, $\|\cdot\|_L = \|\cdot\|_2$ and $R_n=0$:
	$$
	\ell_n(f) - \ell_n(f_0) = -\frac{n}{2}\|f - f_0\|_2^2 + \sqrt{n} W_n(f-f_0),
	$$ 
	where, for $g = \sum_{k=1}^\infty g_k \phi_k$, we set $W_n(g) = \sum_{k=1}^\infty g_k\eps_k$. For the functional, we assume that it admits the following expansion
	\begin{align}\label{expansion_psi_GWN}
		\psi(f) - \psi(f_0) = \langle\psi_0, f - f_0 \rangle_2 + r(f,f_0)
	\end{align}
	for some $\psi_0\in L^2([0,1])$. This gives $\hat{\psi} = \psi(f_0) + \frac{W_n(\psi_0)}{\sqrt{n}} = \psi(f_0) + \frac{\sum_{k=1}^{\infty} \psi_{0,k}\eps_k }{\sqrt{n}}$, where $ \psi_{0,k} = \int_0^1 \psi_0(t)\phi_k(t)dt$,  and $V_0 = \|\psi_0\|_2^2$. Theorem \ref{thm:general_bvm_ap} immediately implies the following result.

\begin{theorem}[Semiparametric BvM in Gaussian white noise]\label{thm:bvm_gwn_ap}
Let $\psi : L^2[0,1] \rightarrow \mathbb{R}$ be a functional of $f$ satisfying \eqref{expansion_psi_GWN}. Suppose that $A_n \subset L^2[0,1]$ and the remainder term $r$ in \eqref{expansion_psi_GWN} satisfy Assumption \ref{ass:expansion_assumption}, and that for $f_t = f - \frac{t\psi_0}{\sqrt{n\an}}$, it holds that
\begin{align}
	\frac{\int_{A_n}e^{\alpha_n\ell_n(f_t)}d\Pi(f)}{\int e^{\alpha_n\ell_n(f)}d\Pi(f)} = 1 + o_P(1). \label{GTGWN_3}
\end{align}
Then for $\hat{\psi} = \psi(f_0) + \frac{\sum_{k=1}^{\infty} \psi_{0,k}\eps_k }{\sqrt{n}}$, the $\alpha_n-$posterior distribution of $\sqrt{n\alpha_n}(\psi(f) - \hat{\psi})$ converges weakly in $P_0-$probability to a Gaussian distribution with mean 0 and variance $\|\psi_0\|_2^2$.
\end{theorem}}

% {\it Gaussian White Noise.} In Model \ref{def:GWN}, we have the exact LAN expansion,
% 	$$
% 	\ell_n(f) - \ell_n(f_0) = -\frac{n}{2}\|f - f_0\|_2^2 + W_n(f-f_0),
% 	$$ 
% 	where, for $g = \sum_{k=1}^\infty g_k \phi_k$, we set $W_n(g) = \sum_{k=1}^\infty g_k\epsilon_k$. In particular, the LAN norm $\|\cdot\|_L = \|\cdot\|_2$, the remainder $R_n(f, f_0) = 0$ and
% 	$$\hat{\psi} = \psi(\eta_0) + \frac{W_n(\psi_0)}{\sqrt{n}} = \sum_{k=1}^\infty \psi_{0,k}Y_k,$$
% where $\psi_{0,k} = \int_0^1 \psi_0(t) \phi_k(t) dt$ and with corresponding efficiency bound $V_0 = \|\psi_0\|_L^2 = \|\psi_0\|_2^2$. Theorem \ref{thm:general_bvm_ap} immediately implies the following result.

% \begin{theorem}[Semiparametric BvM in Gaussian white noise]\label{thm:bvm_gwn_ap}
% Let $\psi : L^2[0,1] \rightarrow \mathbb{R}$ be a functional of $f$ with expansion $\psi(f) - \psi(f_0) = \langle\psi_0, f - f_0 \rangle_2 + r(f,f_0)$.
% Suppose that $A_n \subset L^2[0,1]$ satisfies Assumption \ref{ass:expansion_assumption}, and that for $f_t = f - \frac{t\psi_0}{\sqrt{n\an}}$, it holds that
% \begin{align}
% 	\frac{\int_{A_n}e^{\alpha_n\ell_n(f_t)}d\Pi(f)}{\int e^{\alpha_n\ell_n(f)}d\Pi(f)} = 1 + o_P(1). \label{GTGWN_3}
% \end{align}
% Then for $\hat{\psi} = \sum_{k=1}^\infty \psi_{0,k} Y_k$, the $\alpha_n-$posterior distribution of $\sqrt{n\alpha_n}(\psi(f) - \hat{\psi})$ converges weakly in $P_0-$probability to a Gaussian distribution with mean 0 and variance $\|\psi_0\|_2^2$.
% \end{theorem}

We emphasise that the form of $\hat{\psi}$ and the limiting variance come from simply considering the expansion of the log-likelihood and the functional as defined in Assumption \ref{ass:expansion_assumption}.

{\it Density Estimation.}
%We now provide sufficient conditions under which Theorem \ref{thm:general_bvm_ap} holds in the density estimation context. 
For $f,g\in L^2[0,1]$, let $F(g) = \int g(t) f(t) dt$. For $\eta = \log f$, we have the LAN expansion:
\begin{align*}
\ell_n(\eta) - \ell_n(\eta_0) &=\sum_{i=1}^n\{\eta(Y_i) - \eta_0(Y_i)\}
%&= \frac{1}{\sqrt{n}}\sum_{i=1}^n h(Y_i) \\
%&= \sqrt{n}F_0(h) + \frac{1}{\sqrt{n}}\sum_{i=1}^n [h(Y_i) - F_0(h)] \\
= -\frac{n}{2}\|\eta - \eta_0\|_L^2 + \sqrt{n}W_n(\eta - \eta_0) + R_n(\eta, \eta_0),
\end{align*}
where, for $g \in L^2(f_0)$, 
$
	\|g\|_L^2 = \int (g - F_0(g))^2f_0$, $W_n(g) = \frac{1}{\sqrt{n}}\sum_{i=1}^n [g(Y_i) - F_0(g)],
$
and
$
R_n(\eta, \eta_0) = \sqrt{n} F_0(h) + \frac{1}{2}\|h\|_L^2
$
for $h = \sqrt{n}(\eta - \eta_0)$.
For the functional expansion, we assume there exists a bounded measurable function $\effinf : [0,1]\to\mathbb{R}$ such that
\begin{equation}\label{expansion_psi}
\psi(f) - \psi(f_0) = \int \effinf f + \tilde{r}(f,f_0)\qquad\text{and}\qquad \int\effinf f_0 = 0.
\end{equation}
 In this case,
\begin{align*}
\psi(f) - \psi(f_0) = \int (f - f_0)\effinf + \tilde{r}(f,f_0) &= \left\langle\frac{f-f_0}{f_0},\effinf\right\rangle_L + \tilde{r}(f,f_0) \\	
&=\langle \eta - \eta_0, \effinf \rangle_L + r(f,f_0),
\end{align*}
with $r(f,f_0)=\mathcal{B}(f,f_0) + \tilde{r}(f,f_0)$ and
$$
\mathcal{B}(f,f_0) = -\int \left[\eta - \eta_0 - \frac{f-f_0}{f_0}\right]\effinf f_0.
$$
Note that the last steps are required since the functional expansion should hold in terms of the parameter $\eta = \log f$ rather than the density $f$ itself.
This gives  $\hat{\psi} = \psi(f_0) + W_n(\effinf)/\sqrt{n} = \psi(f_0) + \sum_{i=1}^n \effinf(Y_i)/n$, and limiting variance $\|\effinf\|_L^2 = \int \effinf^2 f_0$. With this in mind, we obtain the following result.

\begin{theorem}[Semiparametric BvM in density estimation]\label{thm:bvm_density_ap}
	Let $f\to \psi(f)$ be a functional on probability densities on $[0,1]$ and assume there exists a bounded measurable function $\effinf : [0,1]\to\mathbb{R}$ such that \eqref{expansion_psi} holds. 
%	, and satisfying $\int\effinf f_0 = 0$ such that,
%\begin{align} \label{expansion_psi}
%\psi(f) - \psi(f_0) = \int \effinf f + \tilde{r}(f,f_0).
%\end{align}
Suppose that for some sequence $\eps_n \rightarrow 0$ and sets $A_n \subset \{f: \|f-f_0\|_1 \leq \eps_n\}$, for  $\tilde{r}$ as in \eqref{expansion_psi},
\begin{align}
	& \Pi_{\alpha_n}(A_n | Y^n) = 1 + o_P(1), \label{GTDE_1}\\
	&\sup_{f \in A_n}\tilde{r}(f, f_0) = o\left(\frac{1}{\sqrt{n\alpha_n}}\right) \label{GTDE_2}. 
\end{align}
Denote
$f_t = fe^{-t\effinf/\sqrt{n\alpha_n}}/F(e^{-t\effinf/\sqrt{n\alpha_n}})$ and for $A_n$ as above,  assume that
\begin{equation} \label{GTDE_3}
\frac{\int_{A_n}e^{\alpha_n\ell_n(f_t)}d\Pi(f)}{\int e^{\alpha_n\ell_n(f)}d\Pi(f)} = 1 + o_P(1). 
\end{equation}
Then for $\hat{\psi} = \psi(f_0) + \frac{1}{n}\sum_{i=1}^n \effinf(Y_i)$, the $\alpha_n-$posterior distribution of $\sqrt{n\alpha_n}(\psi(f) - \hat{\psi})$ converges weakly in $P_0-$probability to a Gaussian distribution with mean 0 and variance $\int  \effinf^2 f_0$.
\end{theorem}

We now proceed to apply these results to concrete priors.

%%%%%%%%%%%%%%%%%%%%%%%%%%%%%%%%%%%%%%%%%%%%%%%%%%%%

\subsection{Random Histogram Priors}

We first illustrate our main theorem for density estimation using a class of histogram priors. We will see that the $\al_n$--posterior, although leading to an enlarged variance in estimating functionals, can sometimes lead to weaker conditions in terms of regularities. In particular, although the $\al_n$--posterior rate may then be slower, it provides more robustness against possible semiparametric bias that may occur for certain functionals. We provide an example where uncertainty quantification is unreliable for the true posterior, because credible sets will suffer from bias, whereas credible sets from the $\al_n$--posterior still cover the true unknown function.

{\em Random histogram prior. } For any integer $k$, we define a distribution on $H^1_k$, the subset of regular histograms with $k$ equally spaced bins which are densities on $[0, 1]$. Let $S^1_k = \{\omega \in [0, 1]^k, \: \sum_{i=1}^{k}\omega_i =1 \}$ be the unit simplex in $\R^k$. Denote by $\mathcal{D}(\delta_{1},\dots, \delta_{k})$ the Dirichlet distribution with real positive weights $(\delta_{1},\dots,\delta_{k})$ on $S^1_k$ and consider the induced measure $\mathcal{H}(k, \delta_{1},\dots, \delta_{k})$ on $H^1_k$ defined as 
	\begin{align}\label{def::RHP}
		f(x)=k\sum_{j=1}^{k}\omega_j1_{I_j}(x), \qquad \omega=(\omega_1, \dots, \omega_k) \sim \mathcal{D}(\delta_{1},\dots, \delta_{k}),
	\end{align}
	where $I_j=[(j - 1)/k, j/k]$ for $j = 1, \dots , k$.
	We now define the random histogram prior $\Pi=\Pi_n$ that we will use throughout this section. Let $K_n\to\infty$ be a diverging sequence to be chosen below and $(\delta_{1,n},\dots ,\delta_{K_n,n})$ a sequence of positive weights and set $\Pi=\Pi_n=\mathcal{H}(K_n, \delta_{1,n},\dots, \delta_{K_n, n})$. We  assume the weights satisfy the technical condition
	\begin{align}\label{condition_weights}
	    \sum_{i=1}^{K_n}\delta_{i,n}=o(\sqrt{n\alpha_n})
	\end{align}
	as $n \rightarrow \infty$, which ensures that the prior is not too concentrated around its mean.

{\em Linear functionals.} Let us apply  Theorem \ref{thm:bvm_density_ap} to the case of linear functionals, i.e. those of the form $\psi(f) = \int \tilde{\psi}_{f_0} f$. For $k\geq1$ and $h$ in $L^2[0,1]$, consider the $L^2$-projection $h_{[k]}$ of $h$ onto the set of histograms with $k$ bins:
	\begin{align*}
	    h_{[k]}= k\sum_{j=1}^{k} \left(\int_{I_j}h\right) 1_{I_j}.
	\end{align*}
Writing $\tilde{\psi}=\tilde{\psi}_{f_0}$ for short, define $\hat{\psi}_{[k]}$ and the sequence $V_k$ from the projection $\tilde{\psi}_{[k]}$ of $\tilde{\psi}$ as 
	\begin{align*}
		\hat{\psi}_{[k]}&= \psi(f_{0})+\frac{1}{n}\sum_{i=1}^{n} \tilde{\psi}_{[k]}(Y_i), \qquad V_k=\int f_0\tilde{\psi}_{[k]}^2 - \left(\int f_0\tilde{\psi}_{[k]}\right)^2.
	\end{align*}
Recall that here, $\hat{\psi}=\psi(f_0)+\sum_{i=1}^{n} \tilde{\psi}(Y_i)/n$ and  $V_0=\int f_0\tilde{\psi} ^2$ (not to be confused with setting $k=0$ in the last display).
	\begin{proposition}\label{thm_random_histo_prior} Let $\Pi$ be the random histogram prior \eqref{def::RHP} with $k=K_n$ and weights satisfying \eqref{condition_weights}. 
	Suppose $f_0$ is bounded and 
		\begin{align}\label{assum::prop_RHP_1}
			\Pi_{\alpha_n}(\|f-f_{0,[K_n]}\|_1 \leq \eps_n|Y^n)= 1+o_P(1),
		\end{align}
for a sequence $\eps_n \to 0$. Suppose additionally that   
	\begin{align}\label{assum::prop_RHP_2}
		\sqrt{n\alpha_n}(\hat{\psi}_{[K_n]} - \hat{\psi}) = o_P(1).
	\end{align} 
	Then the $\alpha_n$--posterior distribution of $\sqrt{n\alpha_n}(\psi(f)-\hat{\psi})$ converges weakly in $P_0-$probability to a Gaussian distribution with mean 0 and variance $V_0$.
	\end{proposition}
	% We could allow a remainder term $r$ in this theorem as in the previous general theorem in the context of density estimation, but then the formulation of the theorem is a bit more complicated.	
Assumption \eqref{assum::prop_RHP_2} ensures that, asymptotically, the posterior distribution is centered at an efficient estimator. Arguments in Lemma \ref{lem:bias} give the  expansion $\hat{\psi}_{[K_n]}  - \hat{\psi} = F_0(\tilde{\psi}_{[K_n]})+o_P(1/\sqrt{n})$, so that \eqref{assum::prop_RHP_2} can also be formulated as $\sqrt{n\alpha_n}F_0(\tilde{\psi}_{[K_n]})= o(1)$. Let us also note that, without assuming \eqref{assum::prop_RHP_2}, the proof of Proposition \ref{thm_random_histo_prior} still gives that the $\alpha_n$--posterior distribution of $\sqrt{n\alpha_n}(\psi(f)-\hat{\psi}_{[K_n]})$ converges weakly to a $\cN(0,V_0)$ variable. The marginal posterior is thus centered at $\hat{\psi}_{[K_n]}$, whether this is an efficient estimator or not.
 %\label{rem:BvM_PRH}
	
To gain a quantitative understanding of the minimal smoothness assumptions required by the BvM and understand how these relate to those for the full posterior \citep{CR2015}, we next consider H\"older smoothness scales.

%	\paragraph{Application in the context of Hölder regularity assumptions.} Consider estimating the functional $\psi(f)=\int a f $. This functional can be expanded as in \eqref{expansion_psi} with $\tilde{r}=0$ and $\tilde{\psi} = a - \int a f_0 $. 
	
	\begin{corollary}\label{cor::BvM_RHP} Consider estimating $\psi(f_0)=\int_0^1 a f_0$ with $a \in \mathcal{C}^{\gamma}([0,1])$, $ f_0 \in \mathcal{C}^{\beta}([0,1])$ bounded
away from zero and $\beta,\gamma\in (0,1]$. Let $\Pi$ be the random histogram prior \eqref{def::RHP}  with weights satisfying \eqref{condition_weights} and $(n\alpha_n)^{-b}\leq \delta_{i,n} \leq 1$ for some $b>0$, and with $K_n=o(n\alpha_n/\log(n\alpha_n))$. If
	\begin{align}\label{assum::cor_RHP_1}
		\sqrt{n\alpha_n} K_n^{-\gamma -\beta} =o(1),
	\end{align}
	then the $\alpha_n$--posterior distribution of $\sqrt{n\alpha_n}(\psi(f)-\hat{\psi})$ converges weakly in $P_0-$probability to a Gaussian distribution with mean 0 and variance $V_0$.
%	 the result	of Proposition \ref{thm_random_histo_prior} holds.
	\end{corollary}
	
%	\begin{remark}\label{rem::BvM_RHP} 		Assumption \ref{assum::cor_RHP_1} ensures that the no-bias condition is satisfied.
%		As in Proposition \ref{thm_random_histo_prior}, if Assumption \ref{assum::cor_RHP_1} is not satisfied, then we still have that the $\alpha_n$-posterior distribution of $\sqrt{n\alpha_n}(\psi(f)-\hat{\psi}_{[K_n]})$
%		converges weakly to a Gaussian distribution with mean 0 and
%		variance $V=\int f_0 \tilde{\psi}_{f_0}^2$ in $P_0$-probability.
%	\end{remark}
	
The assumption that $K_n$ is of smaller order than $n\al_n$ ensures that the $\alpha_n$-posterior for $f$ at least concentrates around $f_0$ at a rate going to $0$. Condition \eqref{assum::cor_RHP_1} is sufficient for \eqref{assum::prop_RHP_2} under the assumed regularity conditions and, as in Proposition \ref{thm_random_histo_prior}, without assuming \eqref{assum::cor_RHP_1}, the proof of Corollary \ref{cor::BvM_RHP} still gives that the $\alpha_n$--posterior distribution of $\sqrt{n\alpha_n}(\psi(f)-\hat{\psi}_{[K_n]})$ converges weakly to a $\cN(0,V_0)$ variable. Many choices of $\al_n, K_n$ fulfill these conditions. Note that the larger $K_n$, the weaker the regularity conditions, e.g. taking $K_n$ slightly smaller that $n\alpha_n$ gives that a BvM type result, at rate $\sqrt{n\alpha_n}$, holds if the regularities satisfy $\gamma+\beta > 1/2$.

It is interesting to compare the above result with ones for the standard posterior ($\al_n=1$, which was considered in \citealp{CR2015}, Theorem 4.2, albeit under a special choice of $K_n$ only).  Since the conditions on $K_n$ depend on $\al_n$, we underline that we compare both {\em under the same prior} (i.e. with the same choice of $K_n$).
\begin{enumerate}
			\item Case $\alpha_n=\alpha \in (0,1)$: consider a sequence $K_n = o(n/\log{n})$, weights such that $n^{-b}\leq \delta_{i,n} \leq 1 $ and $\sum_{i=1}^{K_n}\delta_{i,n} = o(\sqrt{n})$, and the corresponding random histogram prior. Given this prior, the result obtained for the $\alpha$--posterior is very similar to the one obtained for the full posterior: the larger $K_n$, the smaller the regularities of the representers of the functional $a$ and $f_0$ may be, and the choice $K_n \approx n$ leads to the condition $\gamma+\beta > 1/2$ also for the $\alpha$--posterior. However, a main difference lies in the fact that the asymptotic variance for the $\alpha$--posterior is then $\int f_0 \tilde{\psi}_{f_0}^2/\alpha$, which is larger than the optimal variance $\int f_0 \tilde{\psi}_{f_0}^2$ obtained for the posterior.
			
			\item Case $\alpha_n \rightarrow 0$: to fix ideas consider $\alpha_n= n^{-y}$ with $0<y<1$. Let us further choose $K_n= \lfloor n^x \rfloor$ with $x \in (0, 1-y)$, so that the first condition on $K_n$ holds. As before, one chooses weights such that $n^{-b} \leq \delta_{i,n} \leq 1 $ for  $b>0$ and $\sum_{i=1}^{K_n}\delta_{i,n} = o(n^{(1-y)/2})$. 
			 Corollary \ref{cor::BvM_RHP} with $\alpha_n=1$ implies that a BvM with optimal variance holds under the condition 
			 \[ \gamma + \beta > \frac{1}{2x}.\] 
			 On the other hand, applying Corollary \ref{cor::BvM_RHP} with $\alpha_n=n^{-y}$ gives the condition 
			 \[ \gamma + \beta > \frac{1-y}{2x},\] 
for the BvM with rescaling $\sqrt{n\al_n}$ to hold. 			 
			 Thus we obtain a slower rate with the $n^{-y}$--posterior, but we have a weaker condition on the regularities of the functions $a$ and $f_0$.
		\end{enumerate}	
  
Since the above are only sufficient conditions, we next explicitly construct an example where for the same prior, the semiparametric BvM holds for the $\al_n$--posterior but fails for the standard posterior.

{\em Semiparametric bias and possible lack of BvM.}  For the linear functional $\psi(f)= \int a f$, \citet{CR2015} give a specific counterexample in which the BvM theorem is ruled out because of a nonnegligeable bias appearing in the centering of the posterior distribution of $\int a f$. We now investigate the behaviour of the $\alpha_n$--posterior distribution of $\int a f$ in the same context.
	
	In their counterexample, \citet{CR2015} consider a random histogram prior with a {\em random} number of bins. Here we adapt the counterexample of \citet{CR2015} to our setting of a random histogram prior with a {\em deterministic} number of bins and  derive a result regarding the $\alpha_n$--posterior.
	In order to be able to explicitly compute the bias term, we consider a functional with representer of the form 
    \begin{align}\label{assum::a_counterex}
		a(x) = \sum_{l=-1}^{\infty} \sum_{k=0}^{2^l-1} 2^{-l(\frac{1}{2} +\gamma)} \psi_{lk}(x)
	\end{align}
	for $x$ in $[0,1]$, $\gamma>0$ and $(\psi_{lk})$ the Haar wavelet basis.

	\begin{proposition} \label{prop::counterex} Let $f_0$ be a continuously differentiable function with derivative $f_0'>\rho>0$ bounded away from zero, and let $a$ be as in \eqref{assum::a_counterex} with $0<\gamma\leq 1/2$. Consider the random histogram prior \eqref{def::RHP} with $K_n=2^{p_n}$ and $p_n=\lfloor \log(n^{1/3})/\log(2)\rfloor$ and $\delta_{i,n} = n^{-b}$ for all $i$ and some $b>1/6$. Then:
		\begin{enumerate}
			\item The posterior distribution of  $\sqrt{n}(\psi(f) - \hat{\psi}_{[K_n]} )$ converges weakly in $P_0$-probability to the $\mathcal{N}(0, V_0)$ distribution. Moreover the centering  $\hat{\psi}_{[K_n]}$ satisfies $\hat{\psi}_{[K_n]} - \hat{\psi}   = F_0(\tilde{\psi}_{[K_n]}) + o_P(\frac{1}{\sqrt{n}})$ with $|\sqrt{n}F_0(\tilde{\psi}_{[K_n]}) |\geq c > 0$ and even $|\sqrt{n}F_0(\tilde{\psi}_{[K_n]}) | \rightarrow \infty$ if $\gamma < 1/2$ . In particular, the posterior distribution is biased and the BvM theorem does not hold.
			\item Consider a sequence $\alpha_n=n^{-x}$ with $(1-2\gamma)/3<x<2/3$. Then the $\alpha_n$--posterior distribution of $\sqrt{n\alpha_n}(\psi(f) - \hat{\psi} ) $ converges weakly in $P_0$-probability to the $\mathcal{N}(0, V_0)$ distribution.
		\end{enumerate}
	\end{proposition}

This provides an example in which a non--negligible bias appears in the centering of the posterior distribution of $\int af$ (rescaled by $\sqrt{n}$), whereas the $\alpha_n$--posterior distribution of $\int af$ (rescaled by $\sqrt{n\alpha_n}$) is not biased. This has consequences for uncertainty quantification:  $1-\delta$--quantile credible sets from the posterior have less than $1-\delta$ coverage asymptotically, or even $0$ coverage, whereas those for the $\al_n$-posterior have coverage greater than $1-\delta$ asymptotically. Uncertainty quantification using the $\alpha_n$-posterior is thus reliable, if conservative, whereas that using the standard posterior is not. Of course, note that the $\al_n$--posterior has a spread of order $1/\sqrt{n\al_n}$ (instead of the smaller $1/\sqrt{n}$ for the posterior), which makes it easier to verify confidence statements. We refer to Section \ref{sect:corrected_credible_sets} for more details on the coverage and size of credible sets for the $\al_n$--posterior.

\begin{remark}[Approximately linear functionals]\label{rem:functionals}
The above results extend to certain well-behaved non-linear functionals, such as the square-root $\int \sqrt{f}$, power $\int f^q$, $q\geq 2$, and entropy $\int  f\log f$ functionals.
This is proved in Examples 4.2-4.4 of \citet{CR2015} by controlling the remainder of the functional expansion in Assumption \ref{ass:expansion_assumption}, and the extension to the \ap is similar.	
\end{remark}

%%%%%%%%%%%%%%%%%%%%%%%%%%%%%%%%%%%%%%%%%%%%%%%%%%%%%%%%%%%%%%%

%%%%%%%%%%%%%%%%%%%%%%%%%%%%%%%%%%%%%%%%%%%%%%%%%%%%%%%%%%%%%%%%
\subsection{Gaussian Process Priors}\label{sec:gp_bvm}

In this section, we apply our general semiparametric BvM results to the widely used class of Gaussian process priors. For general definitions and background material on Gaussian processes and their associated reproducing kernel Hilbert spaces (RKHS), the reader is referred to Chapter 11 of \citet{vdv_FNBI} or the monograph by \citet{RW06}. We will establish general results for Gaussian priors in density estimation and Gaussian white noise, and then apply these to specific examples commonly used in practice, such as the Mat\'ern and squared exponential covariance kernels.

Let $W = (W(x):x\in[0,1])$ be a mean-zero Gaussian process with covariance function $K(x,y) = \e [W(x)W(y)]$. One can view $W$ as a Borel-measurable map in some Banach space $(\mathbb{B}, \|\cdot\|)$ (e.g. $(C[0,1],\|\cdot\|_\infty)$) with associated RKHS $(\mathbb{H}, \|\cdot\|_\mathbb{H})$. It is known that nonparametric estimation properties of Gaussian process priors depend on their sample smoothness, as measured through their small-ball probability \citep{vdv,vdv_2007,vdVvZ11}. This can be quantified via the \textit{concentration function} $\varphi_{\eta_0}$ at a point $\eta_0\in\mathbb{B}$, defined as
\begin{equation}
   \varphi_{\eta_0}(\eps) = -\log\Pi(\|W\| \leq \eps) + \frac{1}{2}\inf_{h\in\mathbb{H}: \| h - \eta_0\| < \eps}\|h\|_\mathbb{H}^2, \label{def:concentration_function}
\end{equation}
where $\|\cdot\|$ refers to the norm on $\mathbb{B}$. For the full posterior and standard statistical models, the contraction rate for Gaussian processes is then connected to the solution to the equation $\varphi_{\eta _0}(\eps_n) \sim n\eps_n^2$, see \citet{vdv}. As the next theorem shows, a similar result holds for the fractional posterior by instead considering the inequality
\begin{equation}\label{eq:conc_eqn}
\varphi_{\eta_0}(\eps_n) \leq n\alpha_n \eps_n^2,
\end{equation}
i.e. using the effective sample size $n' = n\alpha_n$ on the right-hand side, see Section \ref{sec:contraction} below for details.

\begin{theorem}[Gaussian white noise]\label{thm:general_theorem_gp_gwn}
Consider the Gaussian white noise model and assign to $f$ a mean-zero Gaussian prior $\Pi$ in $L^2[0,1]$ with associated RKHS $\mathbb{H}$. Suppose that $\eps_n \to 0$ satisfies \eqref{eq:conc_eqn} with $\eta_0 = f_0 \in L^2[0,1]$, and that 
Assumption \ref{ass:expansion_assumption} holds for $\psi(f) = \psi(f_0) + \langle \psi_0, f - f_0 \rangle_2 + r(f,f_0)$ and $A_n \subset \{f: \|f-f_0\|_2 \leq \eps_n\}$.
Further assume that there exist sequences $\psi_n \in \mathbb{H}$ and $\zeta_n\rightarrow0$ such that
\begin{equation}\label{eq:RKHS_conditions}
\| \psi_n - \psi_0 \|_2 \leq \zeta_n, \hspace{5mm}
\|\psi_n\|_\mathbb{H} \leq \sqrt{n\alpha_n}\zeta_n, \hspace{5mm}
\sqrt{n\an}\eps_n \zeta_n \rightarrow 0.
\end{equation}
Then the $\alpha_n-$posterior distribution of $\sqrt{n\alpha_n}(\psi(f) - \hat{\psi})$ converges weakly in $P_0-$probability to a Gaussian distribution with mean 0 and variance $\|\psi_0\|^2_2.$
\end{theorem}

The sequence $\psi_n$ allows one to approximate the Riesz representer, $\psi_0$, of the functional by elements of the RKHS $\mathbb{H}$. This is helpful since for elements of the RKHS, one can directly deal with the change of measure condition \eqref{GTGWN_3} using the Cameron-Martin Theorem. Note that if $\psi_0 \in \mathbb{H}$, one may immediately take $\psi_n = \psi_0$ and $\zeta_n = 0$. The main message from Theorem \ref{thm:general_theorem_gp_gwn} is that for semiparametric inference, the fractional posterior mirrors the main heuristic properties of parametric models (e.g. Equation \ref{eq:conjugate}). In particular, all conditions are driven by the usual conditions for semiparametric BvMs for Gaussian priors but with effective sample size $n' = n\alpha_n$ reflecting the downweighting of the data (see Section \ref{sec:contraction} for specific discussion on this regarding contraction rates). The resulting marginal posterior for the functional $\psi(f)$ is again centered at an efficient estimator $\hat{\psi}$, but has variance inflated by a $1/\alpha_n$-factor.

Turning now to density estimation, we use the standard approach of using the exponential link function (\citealp{vdv_FNBI}, Section 2.3.1) to ensure the Gaussian process induces a prior on the set of probability densities:
\begin{equation}
f(x) = f_W(x) = \frac{e^{W(x)}}{\int_0^1 e^{W(y)}dy}. \label{def:gp_inducement}
\end{equation}
% When we do this we have the following result, which gives conditions under which the \ap distribution of a functional of $f$ will converge to a Gaussian.

\begin{theorem}[Density estimation]\label{thm:general_theorem_gp_density_estimation}
Consider density estimation on $[0,1]$ and suppose $f_0\in C[0,1]$ is bounded away from zero. Let $W$ be a mean-zero Gaussian process in $(C[0,1],\|\cdot\|_\infty)$ with RKHS $\mathbb{H}$, and consider the induced prior on densities $f$ via \eqref{def:gp_inducement}. Suppose that $\eps_n \to 0$ satisfies \eqref{eq:conc_eqn} with $\eta_0 = \log f_0$.
Let $\psi(f)$ be a functional with expansion \eqref{expansion_psi} having continuous representer $\effinf$ satisfying
$$
\sup_{f \in A_n}\tilde{r}(f,f_0) = o_P\left(\frac{1}{\sqrt{n\an}}\right)
$$
for some $A_n \subset \{f: \|f-f_0\|_1 \leq \epsilon_n\}$ with $\Pi_\an(A_n |Y^n) = 1+o_P(1)$. Further assume that there exist sequences $\psi_n \in \mathbb{H}$ and $\zeta_n\rightarrow0$ such that
$$
\| \psi_n - \effinf \|_\infty \leq \zeta_n, \hspace{5mm}
\|\psi_n\|_\mathbb{H} \leq \sqrt{n\alpha_n}\zeta_n, \hspace{5mm}
\sqrt{n\an}\epsilon_n \zeta_n \rightarrow 0.	
$$
Then the $\alpha_n-$posterior distribution of $\sqrt{n\alpha_n}(\psi(\eta) - \hat{\psi})$ converges weakly in $P_0-$probability to a Gaussian distribution with mean 0 and variance $\|\effinf\|_L^2 = \int_0^1 \effinf^2 f_0.$
\end{theorem}

The implications of Theorem \ref{thm:general_theorem_gp_density_estimation} are similar to those of Theorem \ref{thm:general_theorem_gp_gwn}. The use of the slightly stronger $\|\cdot\|_\infty$-norm compared to the $\|\cdot\|_2$-norm in Theorem \ref{thm:general_theorem_gp_gwn} is required to deal with the nonlinear link function \eqref{def:gp_inducement} and has little effect on our main results.

We consider the following specific examples of Gaussian priors.

\begin{example}[Infinite series]\label{ex:series}
Let $(\phi_k)_{k \geq 1}$ be an orthonormal basis of $L^2[0,1]$. For $\gamma>0$, consider the random function
\begin{equation}
    W(x) = \sum_{k=1}^\infty k^{-\gamma-1/2} Z_k \phi_k(x) , \qquad \qquad Z_k \sim^{iid} \mathcal{N}(0,1).\label{def:infinite_series_prior}
\end{equation}
\end{example}

Define the Sobolev scales in terms of the $(\phi_k)$ basis:
\begin{equation}\label{eq:sobolev}
\mathcal{H}^\beta(R) :=  \left\{ f \in L^2[0,1]: \sum_{k= 1}^\infty k^{2\beta} |\langle f,\phi_k\rangle_2|^2 \leq R^2 \right\}.
\end{equation}
If $(\phi_k)$ is the Fourier basis, then $\mathcal{H}^\beta$ coincides with the usual notion of Sobolev smoothness of periodic functions on $(0,1]$. The infinite series prior \eqref{def:infinite_series_prior} models an almost $\gamma$-smooth function in the sense that it assigns probability one to $\mathcal{H}^s$ for any $s<\gamma$.

\begin{example}[Mat\'ern]\label{ex:Matern}
The Mat\'ern process on $\R$ with parameter $\gamma>0$ is the mean-zero stationary Gaussian process with covariance kernel (Example 11.8 in \citealp{vdv_FNBI})
$$K(s,t) = K(s-t) = \int_\R e^{-i(s - t)\lambda} (1+|\lambda|^2)^{-\gamma -1/2} d\lambda.$$
The covariance function can alternatively be represented in terms of special functions, see e.g. p.84 of \citet{RW06}.
\end{example}

\begin{example}[Squared exponential]\label{ex:SE}
The rescaled squared exponential process on $\R$ with parameter $\gamma>0$ is the mean-zero stationary Gaussian process with covariance kernel
$$K(s,t) = K(s-t) = \exp\left(-\frac{1}{k_n^2}(s-t)^2\right),$$
where $k_n = \left(\frac{n\an}{\log^2(n\an)}\right)^{-\frac{1}{1+2\gamma}}$ is the length scale.
\end{example}

The Mat\'ern and squared exponential are two of the most widely used covariance kernels in statistics and machine learning \citep{RW06}. The sample paths of the squared exponential process are analytic, and so are typically too smooth to effectively model a function of finite smoothness in the sense that they yield suboptimal contraction rates. Rescaling the covariance kernel using the decaying lengthscale $k_n$ as in Example \ref{ex:SE} allows one to overcome this and model a $\gamma$-smooth function \citep{vdv_2007}.

\begin{example}[Riemann-Liouville]\label{ex:RL}
The Riemann-Liouville process released at zero of regularity $\gamma>0$ is defined as
\begin{equation}
    W^\gamma(x) = \sum_{k=0}^{\lfloor \gamma \rfloor + 1} Z_kx^k + \int_0^x (x-s)^{\gamma - 1/2} dB_s, \label{def:RL_process}
\end{equation}
where $Z_k \sim^{iid} \mathcal{N}(0,1)$ and $B$ is an independent Brownian motion. 
\end{example}

For $\gamma = 1/2$, the Riemann-Liouville process reduces to Brownian motion released at zero. Each of the above Gaussian processes is suitable for modelling a $\gamma$-smooth function in a suitable sense, which can differ between the processes. For simplicity, we state the following result for linear functionals, but it can be extended to certain non-linear functionals following Remark \ref{rem:functionals}.

\begin{corollary}\label{cor:gp_examples}
Let $W$ be a mean-zero Gaussian process. In Gaussian white noise, take as prior $f = W$ and set $\eta_0 = f_0$, while in density estimation take the prior on densities $f$ induced by \eqref{def:gp_inducement} and set $\eta_0 = \log f_0$.
Let $\psi(f) = \int_0^1 f a$ be a linear functional and consider the two cases:
\begin{itemize}
\item[(i)] $W$ is an infinite Gaussian series (Example \ref{ex:series}) with parameter $\gamma$, $\eta_0 \in \mathcal{H}^\beta$ and $a \in \mathcal{H}^\mu$, where $\mathcal{H}^s$ is defined in \eqref{eq:sobolev};
\item[(ii)] $W$ is a Mat\'ern, rescaled squared exponential or Riemann-Liouville process (Examples \ref{ex:Matern}-\ref{ex:RL}) with parameter $\gamma$, $\eta_0 \in C^\beta$ and $a \in C^\mu$.
\end{itemize}
%In Gaussian white noise take as prior $f=W$, while in density estimation take the prior on densities $f$ induced by \eqref{def:gp_inducement}. 
If
$$ \gamma \wedge \beta > \frac{1}{2} +(\gamma - \mu) \vee 0,$$
then the $\alpha_n-$posterior distribution of $\sqrt{n\alpha_n}(\psi(\eta) - \hat{\psi})$ converges weakly in $P_0-$probability to a Gaussian distribution with
\begin{itemize}
\item[(a)] Gaussian white noise: mean 0 and variance  $\|a\|_2^2$ in both cases (i) and (ii); 
\item[(b)] density estimation: mean 0 and variance $\|\effinf\|_L^2 = \int \effinf^2f_0$, with $\effinf = a - \int a f_0$ in case (ii).
\end{itemize}
\end{corollary}

% OLD VERSION OF COROLLARY BELOW \begin{proposition}\label{prop:gp_examples_density_estimation}
% Suppose that $\psi(f) = \int af$ is a linear functional, with $a \in C^\mu$ for some $\mu > 0$.  
%   Suppose that either:
  
%   \underline{Case (i):} We are in the Gaussian white noise context with $\eta_0 = f_0 \in C^\beta$ for some $\beta > 0$, with the prior on $f$ one of
%   \begin{enumerate}
%     \item the infinite series prior with $\sigma_k = k^{-\frac{1}{2} - \gamma}$,
%     \item the Matérn process with parameter $\gamma$,
%     \item the rescaled square exponential process with lengthscale sequence $k_n = n^{-\frac{1}{1+2\gamma}}$,
%     \item the Riemann-Liouville process with parameter $\gamma$.
%     \end{enumerate}
  
%   \underline{Case (ii):} We are in the density estimation context with $\eta_0 = \log f_0 \in C^\beta$ for some $\beta > 0$ and we have a prior, $\Pi_\gamma$, induced by one of the Gaussian processes 2-4 in the above list.
  
% If,
% $$ \gamma \wedge \beta > \frac{1}{2} +(\gamma - \mu) \vee 0,$$
% then the $\alpha_n-$posterior distribution of $\sqrt{n\alpha_n}(\psi(\eta) - \hat{\psi})$ converges weakly in $P_0-$probability to a Gaussian distribution with mean 0 and variance  $\|a\|_2^2$ (Case (i)) or $\|\effinf\|_L^2 = \int \effinf^2 f_0$, with $\effinf = a - \int a f_0$ (Case (ii)).
% \end{proposition}

Corollary \ref{cor:gp_examples} shows that for widely used Gaussian priors, parametric BvM results and conclusions (e.g. \citealp{miller21,MORV22}) extend to semiparametric problems. In particular, for regular enough functionals, the heuristic ideas and intuition extend from low-dimensional frameworks to our more complex setting involving an infinite-dimensional nuisance parameter. 

For the infinite series prior (Example \ref{ex:series}) in Gaussian white noise, one can also directly derive the last conclusion using the explicit form of the posterior coming from conjugacy. In particular, this allows one to consider low regularity functionals where $\sqrt{n}$-estimation is not possible, which falls outside the usual BvM setting. The following extends the computations of Theorem 5.1 of \citet{BIP2011} to the $\alpha_n$-posterior. 
\begin{lemma}\label{prop:contraction_rate_an_posterior_Lmu}
Consider Gaussian white noise, let $f$ have the infinite series prior (Example \ref{ex:series}) of regularity $\gamma>0$ and consider the linear function $\psi(f) = \int_0^1 af$. If $f_0 \in \mathcal{H}^\beta$, $a \in \mathcal{H}^\mu$, $0< \alpha_n\leq 1$ and $\mu \geq -\beta$, then
$$
E_{f_0} \Pi_{\alpha_n}(f : |\psi(f) - \psi(f_0)| \geq M_n \max \{ (n\an)^{-\frac{\beta\wedge(\frac{1}{2}+\gamma) + \mu}{1 + 2\gamma}} , (n\an)^{-1/2} \} | Y) \rightarrow 0,
$$
for every sequence $M_n \rightarrow \infty$ as $n\to\infty$.
%$$
%\eps_n = (n\alpha_n )^{-\left(\frac{\beta+\mu}{1 + 2\gamma} \wedge 1\right)} + (n\alpha_n)^{-\left(\frac{1/2 + \gamma + \mu}{1+2\gamma} \wedge \frac{1}{2}\right)},
%$$
%with the rate uniform over $f_0$ in balls in $\mathcal{H}^\beta$. 
\end{lemma}

Thus in the low regularity regime, the $\alpha_n$-posterior may inflate the posterior variance of $\psi(f)$ by a factor slower than $1/\alpha_n$. This corresponds to a more `nonparametric' regime and the conclusions here are similar to those obtained for contraction rates for the full parameter, see Section \ref{sec:contraction} for more discussion.

\textit{Empirical verification of the BvM for the rescaled squared exponential process.} Consider density estimation with $n=10,000$ observations drawn from the density on $[0,1]$ given by $f \propto e^g$, with $g$ having coefficients $g_k = k^{-\frac{1}{2} - \beta}$ in the Fourier basis of $[0,1]$. Consider estimating the linear functional given by $\psi(f) = \int_0^1 a(t) f(t) dt $ with $a$ defined by coefficients $a_k = k^{-\frac{1}{2} - \mu}$ in the same basis. The estimator is $\hat{\psi}=\frac{1}{n}\sum_{i=1}^na(Y_i)$, the efficient influence function is $\effinf(t) = a(t) - \psi(f_0)$, and the information bound is $\|\effinf\|_L^2 = \int_0^1 a(t)^2 f_0(t) dt- \psi(f_0)^2$. We take as prior the exponentiated Gaussian process prior \eqref{def:gp_inducement} with $W$ a rescaled squared exponential process (Example \ref{ex:SE}) with length scale $k_n = n^{-\frac{1}{1+2\gamma}}$. Figure \ref{fig:gp_bvm_condition_illustration} displays histograms of \ap draws of $\sqrt{n\an}(\psi(f) - \hat{\psi})/\|\effinf\|_L$ with $\alpha_n = 1/4$ for combinations of $\beta,\gamma$; the blue distributions represent cases for which the condition $\gamma \wedge \beta > \frac{1}{2} + (\gamma - \mu)$ in Corollary \ref{cor:gp_examples} is satisfied, while the red distributions represent cases when the condition is violated. Posterior draws were generated by MCMC using the \texttt{sbde} package \citep{Tokdar_sbde}.

One can see that when the condition is verified, the marginal posterior appears to be Gaussian with the correct variance,  but when the condition is violated this does not seem to be the case. This illustrates that the asymptotic results and conditions are applicable in finite sample sizes. 

\begin{figure}[h]
  \centering
  \includegraphics[scale=0.45]{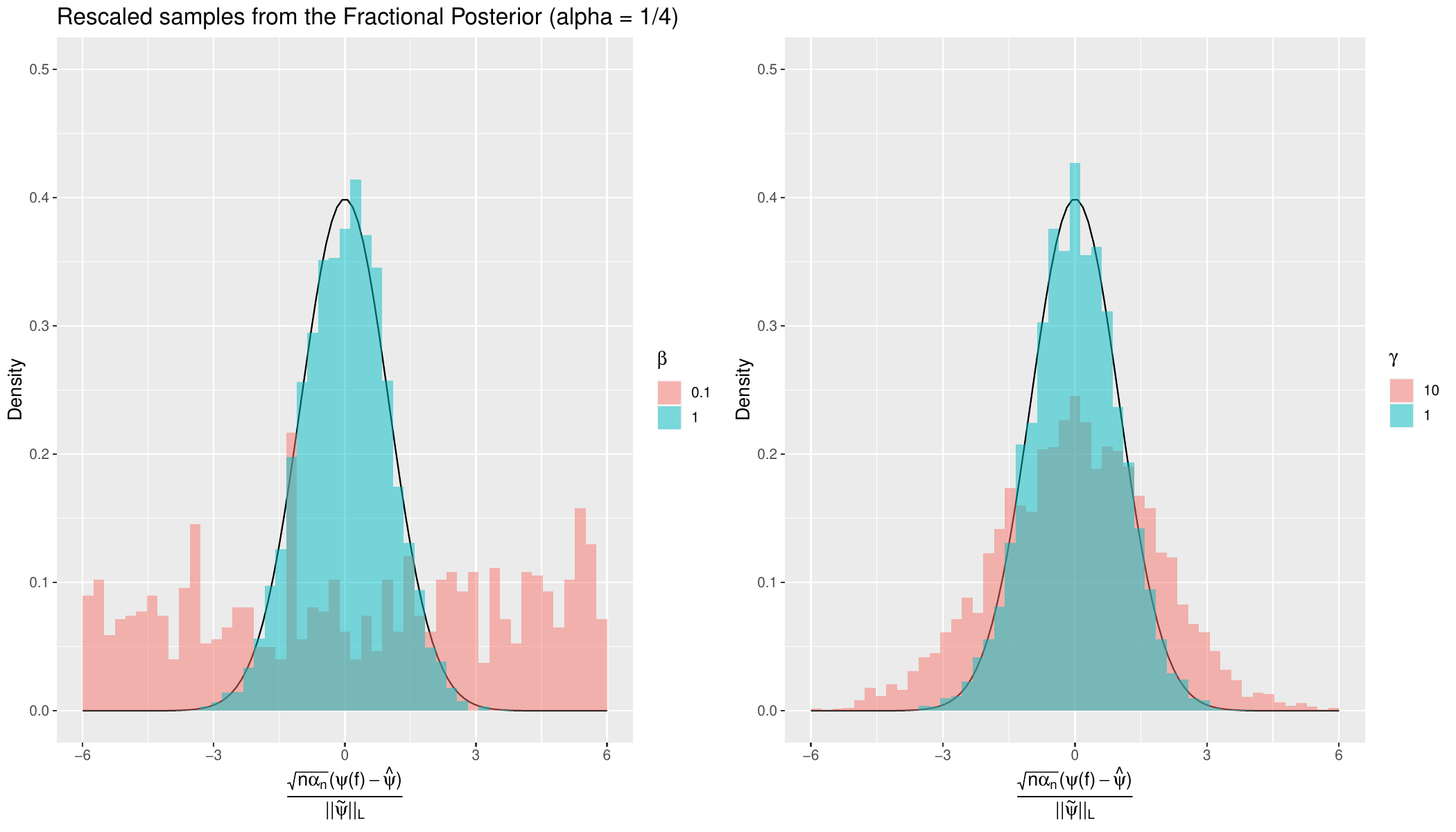}
  \caption{Draws from the fractional posterior distribution of $\sqrt{n\an}(\psi(\eta) - \psi(f_0))/\|\effinf\|_L$ with $\alpha_n =1/4$ for different combinations of $\beta$ and $\gamma$. In all cases $\mu = 1$, and on the left $\gamma = 1$ for different values of $\beta$, while on the right $\beta = 1$ for different values of $\gamma$. The red distributions correspond to cases where $\gamma \wedge \beta < \frac{1}{2} + (\gamma - \mu)$ (the condition in Corollary \ref{cor:gp_examples} is violated), while the blue distributions correspond to cases where $\gamma \wedge \beta > \frac{1}{2} + (\gamma - \mu)$ (the condition is verified). The black line is the density of a  $\cN(0,1)$ random variable.}
  \label{fig:gp_bvm_condition_illustration}
\end{figure}

%%%%%%%%%%%%%%%%%%%%%%%%%%%%%%%%%%%%%%%%%%%%%
\section{Construction of Efficient Confidence Intervals from $\alpha_n$--Posteriors}\label{sect:corrected_credible_sets}
{\color{black} In Section \ref{sect:sp-bvm}, we derived semiparametric BvM theorems for fractional posteriors. When $\al=1$, it is well--known that the BvM theorem implies that certain credible sets (typically built from posterior quantiles) are optimal--sized confidence sets. For $0<\alpha_n<1$, this is no longer true for $\al_n$--posteriors in that the length of the resulting credible sets will overshoot the optimal length given by the semiparametric efficiency bound. We now investigate how this can be remedied.

For simplicity, we focus on the case where $\psi(\eta)$ is one dimensional.  
Suppose one has obtained a BvM theorem for $\psi(\eta)$, for instance using the results from  Section \ref{sect:sp-bvm}, that is,% {\color{red}[do we want to formally defined $\leadsto$? Here it doesn't matter, but it is used in the proofs]}
\begin{align}\label{conv_BvM_alpha}
		\Pi_{\alpha_n}[\cdot | Y^n] \circ \tau_n^{-1} \leadsto \mathcal{N}(0,V),
	\end{align}
	where  $\tau_n : \eta \rightarrow \sqrt{n\alpha_n}(\psi(\eta) - \hat{\psi})$, the centering $\hat{\psi}$ is linear efficient %\ora{note: linear efficient means something different from \eqref{conv_hat_psi}, but it implies it --needed?-- } 
	and  $V$ is the efficiency bound for estimating $\psi(\eta)$. In particular,
	\begin{align}\label{conv_hat_psi}
		\sqrt{n}(\hat{\psi} - \psi(\eta_0)) \xrightarrow{\mathcal{L}} \mathcal{N}(0, V).
	\end{align}  
For  $0<\delta <1$, let $a_{n, \delta}^Y$ denote the $\delta$--quantile of the $\alpha_n$--posterior distribution of $\psi(\eta)$ and consider the quantile region
\[  \mathcal{I}_{\al_n}= \mathcal{I}(\delta,\al_n,Y):=(a_{n, \frac{\delta}{2}}^Y\, ,\, a_{n, 1-\frac{\delta}{2}}^Y].\]
 By definition, $\Pi_{\al_n}\left[ \psi(\eta)\in  \mathcal{I}_{\al_n} \given Y\right] =1-\delta$, 
that is,  $\mathcal{I}_{\al_n}$ is a $(1-\delta)$--credible set (assuming the $\al_n$--posterior CDF is continuous, otherwise one takes generalised quantiles).  
For $\delta \in (0,1)$, denote by $q_\delta$ the quantile of the $\mathcal{N}(0,1)$ distribution. From \eqref{conv_BvM_alpha} and standard results recalled in Lemma \ref{lemma_conv_quantiles}, one deduces that $\mathcal{I}_{\al_n}$ admits the following expansion:
\begin{align}\label{expansion_credible_set} 
	\mathcal{I}_{\al_n} = \Big( \hat{\psi} + \frac{\sqrt{V} q_{\frac{\delta}{2}}}{\sqrt{n\alpha_n}} +o_P\left(\frac{1}{\sqrt{n\alpha_n}} \right) ,  \hat{\psi} + \frac{\sqrt{V} q_{1-\frac{\delta}{2}}}{\sqrt{n\alpha_n}}, +o_P\left(\frac{1}{\sqrt{n\alpha_n}} \right)\Big].
\end{align}
When $\alpha_n=1$ or $\alpha_n \rightarrow 1$, it follows from \eqref{expansion_credible_set} and the fact that $\hat{\psi}$ is linear efficient that $\mathcal{I}_{\al_n}$  is asymptotically an efficient confidence interval of level $1-\delta$ for the parameter $\psi(\eta_0)$.

When $\alpha_n\rightarrow \alpha \in [0,1)$, $\mathcal{I}_{\al_n}$ has a diameter blown-up by a factor $1/\sqrt{\al_n}$ compared to $\mathcal{I}_{1}$ for $\al_n=1$, and its  confidence level thus exceeds $1-\delta$. Denoting by $\Phi$ the cumulative distribution function of the $\mathcal{N}(0,1)$ distribution, it follows from \eqref{expansion_credible_set} that, 
\begin{enumerate}
	\item if $\alpha_n \rightarrow \alpha \in (0,1)$, then $P_0\left[ \psi(\eta_0) \in \mathcal{I}_{\al_n} \right] \rightarrow  2\Phi(q_{1-\delta/2} / \sqrt\alpha) - 1 > 1-\delta$;
	\item if $\alpha_n \rightarrow 0$, then $P_0\left[ \psi(\eta_0) \in \mathcal{I}_{\al_n} \right] \rightarrow1$.
\end{enumerate}
An implication is that while $\mathcal{I}_{\al_n}$ is a valid confidence set, it is {\em conservative}, in that its coverage is larger than the target $1-\delta$.

In order to construct an efficient confidence interval from the $\alpha_n$--posterior of $\psi(\eta)$ when $\alpha_n\rightarrow \alpha \in [0,1)$, we consider a modified quantile region.} Let $\bar{\psi}$ be an estimator of $\psi(\eta_0)$ built from the $\alpha_n$--posterior distribution of $\psi(\eta)$ (e.g. posterior median or mean) and set 
	\begin{align}\label{def::eff_region}
		\mathcal{J}_{\al_n}:=\left(\sqrt\alpha_n(a_{n, \frac{\delta}{2}}^Y - \bar{\psi}) +\bar{\psi} \,,\, \sqrt\alpha_n(a_{n, 1-\frac{\delta}{2}}^Y - \bar{\psi}) +\bar{\psi}\right].
	\end{align}
We call this a {\em shift--and--rescale} version of the quantile set (or sometimes corrected set): this new interval is obtained by recentering $\mathcal{I}_{\al_n}$ at $\bar\psi$ and applying a {\em shrinking} factor $\sqrt{\al_n}$. 
We now provide a condition under which the shift-and-rescale set presented in \eqref{def::eff_region} has the correct coverage.
	
	\begin{theorem}\label{prop_corrected_region}
		Suppose \eqref{conv_BvM_alpha}--\eqref{conv_hat_psi} hold for some $0<\al_n<1$, and suppose the estimator $\bar{\psi}$ satisfies
		\begin{align}\label{assum::estimator_bar_psi}
		\bar{\psi} = \hat{\psi} +o_P(1/\sqrt{n}).
		\end{align}
	Then $\mathcal{J}_{\al_n}$ in \eqref{def::eff_region}  is an asymptotically efficient confidence interval of level $1-\delta$ for the parameter $\psi(\eta_0)$, i.e.
	\[ P_0\left[\psi(\eta_0) \in \mathcal{J}_{\al_n}\right] \to 1-\delta \]
as $n\to\infty$.	
%	\end{theorem}
%
%\begin{theorem}\label{thm::eff_region_alpha_ctt}
If $\alpha_n = \alpha \in (0,1]$ is {\em fixed} and $\bar{\psi}$ is the $\alpha$--posterior median, then \eqref{assum::estimator_bar_psi} holds. In particular,  the region \eqref{def::eff_region} is an asymptotically efficient confidence interval of level $1-\delta$ for $\psi(\eta_0)$.     
\end{theorem}

Theorem \ref{prop_corrected_region} states that if the re--centering is close enough to the efficient estimator $\hat\psi$, then the shift--and--rescale modification leads to a confidence set of optimal size (in terms of efficiency) from an information-theoretic perspective, and this is always possible for fixed $\al$ if one centers at the posterior median. When $\al_n$ can possibly go to zero, the situation is more delicate. Indeed, although by definition \eqref{conv_BvM_alpha} is centered around an efficient estimator at the scale $1/\sqrt{n\al_n}$,  it is not clear in general how to deduce from this a similar result at the smaller scale $1/\sqrt{n}$. We do not provide a general answer here, but to gain some insight we consider two specific examples: the conjugate parametric setting \eqref{eq:conjugate}, and  the nonparametric Gaussian white noise model with a conjugate prior, and investigate whether the $\al_n$--posterior median $a_{n, \frac{1}{2}}^Y$ satisfies \eqref{assum::estimator_bar_psi} when $\alpha_n \rightarrow 0$. 

Theorem \ref{prop_corrected_region} applies to semiparametric models, but also to parametric models as a special case. In particular, in the conjugate example \eqref{eq:conjugate}, it is easy to check that \eqref{assum::estimator_bar_psi} holds if and only if $\sqrt{n}\al_n\to\infty$, which is a fairly mild condition. We now turn to a more complex setting.

\textit{Modified credible sets in Gaussian white noise.}
Consider Model \ref{def:GWN} and write $f_0(t) = \sum_{k=1}^\infty f_{0,k} \phi_k(t)$ for $(\phi_k)_{k > 0}$ an orthonormal basis of $L^2[0,1]$. We assign a prior to $f$ by placing independent priors on the basis coefficients $f_k=\langle f,\phi_k\rangle \sim \cN(0,\lambda_k)$, and consider the problem of estimating the linear functional $\psi(f) = \int_0^1 a(t) f(t) dt = \sum_{k=1}^\infty a_k f_{k}$.
By conjugacy arguments, the \ap distribution of $\psi(f) | Y^{(n)}$ is Gaussian (so its median and mean coincide) $\cN(a_{n,1/2}^Y,\bar\sigma^2)$, with 
\[ a_{n,1/2}^Y=\sum_{k=1}^\infty \frac{n\an \lambda_k}{1+n\an\lambda_k}a_k Y_k,\qquad \bar\sigma^2=\sum_{k=1}^\infty \frac{\lambda_k}{1+n\an\lambda_k}a_k^2.\] Suppose the smoothness of the true function $f_0$, the representer $a$ and the prior are specified through the magnitude of their basis coefficients as follows, for $\beta, \mu, \gamma > 0$,
\begin{align} \label{params}
	f_{0,k} = k^{-\frac{1}{2} - \beta}, \hspace{5mm} a_k = k^{-\frac{1}{2} - \mu}, \hspace{5mm}
	\lambda_k = k^{-1-2\gamma}.
\end{align}
Setting $\bar{\psi}=a_{n,1/2}^Y$ the posterior mean/median,  the shift--and--rescale set is, with $z_\delta$ the standard Gaussian quantiles,
$$
\mathcal{J}_{\al_n}=\left(\bar{\psi} +\sqrt{\an} z_{\delta/2}\bar{\sigma} \,,\, \bar{\psi} +\sqrt{\an} z_{1-\delta/2}\bar{\sigma}\right].
$$
By Theorem \ref{prop_corrected_region}, for the set $\mathcal{J}_{\al_n}$ to have asymptotic coverage $1- \delta$ it suffices that 
$
\bar{\psi} - \hat{\psi} = o_P(1/\sqrt{n}).
$
The following result describes the behaviour of the shift--and--rescale sets.

\begin{proposition}\label{prop:corrected_credible_sets_behaviour}
Consider the Gaussian white noise model with Gaussian prior $f = \sum_{k=1}^\infty f_k \phi_k$, where $f_k \sim^{ind} N(0,\lambda_k)$, and suppose that \eqref{params} holds. {\color{black}Let $\mathcal{J}_{\alpha_n}$ denote the set \eqref{def::eff_region} with $\bar{\psi}$ equal to the posterior mean/median.} Then
	\begin{enumerate}
		\item If $\beta + \mu > 1+ 2\gamma$, the sets $\mathcal{J}_{\al_n}$ are efficient confidence intervals of level $1-\delta$ if and only if $\sqrt{n} \an \rightarrow \infty$.
		\item If $\beta + \mu = 1+ 2\gamma$, then $\mathcal{J}_{\al_n}$  are efficient confidence intervals of level $1-\delta$  if and only if  $\frac{\sqrt{n} }{\log(n)} \an \rightarrow \infty$.
		\item If $\frac{1}{2} + \gamma < \beta + \mu < 1+ 2\gamma$, then the sets $\mathcal{J}_{\al_n}$ are efficient confidence intervals of level $1-\delta$ if and only if  $n^{1-\frac{1+2\gamma}{2(\beta + \mu)}}\an \rightarrow \infty.$
	\end{enumerate}
\end{proposition}
This result assumes $\gamma + 1/2 < \beta + \mu$, which corresponds to the case where 
% Theorem 5.1 of \cite{BIP2011} gives one a contraction rate (in $L_1$ norm) for this example (in the full posterior) of $\epsilon_n = n^{-\left(\frac{\beta \wedge (\frac{1}{2} + \mu)}{1+2\delta} \wedge \frac{1}{2}\right)}$, and this can be generalized to the \ap (Proposiiton \ref{prop:contraction_rate_an_posterior_Lmu}) to give a contraction rate of $\epsilon_n = (n\an)^{-\left(\frac{\beta \wedge (\frac{1}{2} + \mu)}{1+2\delta} \wedge \frac{1}{2}\right)}$. Considering the cases, we see that if $\mu \geq 0$ then $\beta + \mu > \frac{1}{2} + \delta$ implies a contraction rate of $(n\an)^{-\frac{1}{2}}$, while $\beta + \mu < \frac{1}{2} + \delta$ implies a contraction rate slower than $(n\an)^{-\frac{1}{2}}$.
 a Bernstein-von Mises result for the standard posterior ($\an \equiv 1$) holds, see  Theorem 5.4 in \citet{BIP2011}, cases (ii) and (iii).  In agreement with these results, we see by setting $\al_n=1$ in Proposition \ref{prop:corrected_credible_sets_behaviour} that in all three cases standard credible sets $\mathcal{J}_1$ are efficient confidence sets. 
 The point of Proposition \ref{prop:corrected_credible_sets_behaviour} is to investigate to what extent shift--and--rescale sets $\mathcal{J}_{\al_n}$ centered at the posterior median remain efficient confidence sets when $\al_n$ goes to $0$.  In Cases 1 and 2, the condition is very mild and any sequence $(\al_n)$ essentially slower than $1/\sqrt{n}$ works (recall as noted above that in the basic parametric example \eqref{eq:conjugate}, the shift--and--rescale sets are efficient under the same condition $\sqrt{n}\alpha_n\to\infty$). When $\beta+\mu$ approaches $1/2+\gamma$ (Case 3),  $\alpha_n$ is only allowed to decrease quite slowly to $0$ to preserve efficiency. An interpretation is that the problem becomes more `nonparametric' and the $\al_n$--posterior median does not necessarily concentrate fast enough in order for  \eqref{assum::estimator_bar_psi} to be satisfied.
 
%{\it Empirical Results for Corrected Credible Sets in the Gaussian White Noise Model}\\
{\em Simulation study.} We now illustrate the applicability of the asymptotic result presented in Proposition \ref{prop:corrected_credible_sets_behaviour} to the finite sample setting. We simulated 10,000 observations  of $Y^n$ from the Gaussian white noise model ($n=10,000$) with 3 different parameter combinations of $(\beta, \mu, \gamma)$ corresponding to the three different cases presented in Proposition \ref{prop:corrected_credible_sets_behaviour}. With each of these observations, we produced credible sets from the full posterior, the $\an-$posterior, and the shift--and--rescale sets from the $\an-$posterior, and computed their empirical coverage (the proportion of the sets which contained the true parameter $\psi(f_0)$), their length, and the mean bias of their centering. This data is presented in Table \ref{Tab:credible_sets_coverage}.
For the $\an-$posterior and the corrected credible sets, we study two regimes in each case: one where $\an$ breaches the condition described in Proposition \ref{prop:corrected_credible_sets_behaviour} by a $\sqrt{\log{n}}$ factor, and one where $\an$ verifies the condition by a $\sqrt{\log{n}}$ factor. This results in a large difference in the empirical coverage of the shift--and--rescale sets; when the lower bound is breached, the corrected sets have little or no coverage, but when the lower bound is respected they have approximately the target coverage. 
{\color{black}   
In this example, the conditions provided by Proposition \ref{prop:corrected_credible_sets_behaviour} seem to be accurate (note that due here to the moderate sample size of $n=10,000$, the $\sqrt{\log(n)}$ factor is still not completely negligible in comparison to the polynomial factor specified by Proposition \ref{prop:corrected_credible_sets_behaviour}, which explains why the empirical behaviours clearly feature either coverage or non-coverage). }

\begin{table}[h]
\centering
\begin{tabular}{l|ccc}
  {\bf Gaussian White Noise} \\ \hline
 {\bf Case 1}\hspace{1mm} $\beta + \mu > 1+ 2\gamma$ \\ \hspace{14mm}$\beta = 2, \mu =2, \gamma = 0.5$ & Cov. & Len. & Bias (SD)\\
  \hline
Full Posterior & 0.95 & 0.02 & -0.00008 (0.005)  \\ 
  $\an-$Posterior ($\sqrt{n} \an = 1/\sqrt{\log(n)} \rightarrow 0$) & 1.00 & 0.86 & -0.05331 (0.004)  \\ 
$\an-$Posterior ($\sqrt{n} \an = \sqrt{\log(n)} \rightarrow \infty$) & 1.00 & 0.08 & -0.00059 (0.005)  \\ 
 Shift--and--rescale Sets ($\sqrt{n} \an = 1/\sqrt{\log(n)} \rightarrow 0$) & 0.00 & 0.02 & -0.05331 (0.004) \\ 
  Shift--and--rescale  Sets ($\sqrt{n} \an = \sqrt{\log(n)} \rightarrow \infty$) & 0.95 & 0.02 & -0.00059 (0.005)  \\\hline
  {\bf Case 2} \hspace{1mm} $\beta + \mu = 1+ 2\gamma$ \\ \hspace{14mm} $\beta = 1, \mu =1, \gamma = 0.5$\\\hline 
  Full Posterior & 0.95 & 0.02 & -0.00006 (0.005)  \\ 
  $\an-$Posterior ($\frac{\sqrt{n} }{\log(n)} \an = 1/\sqrt{\log(n)} \rightarrow 0$) & 1.00 & 0.29 & -0.01462 (0.005)  \\ 
$\an-$Posterior ($\frac{\sqrt{n} }{\log(n)} \an= \sqrt{\log(n)} \rightarrow \infty$) & 0.99 & 0.03 & -0.00021 (0.005) \\ 
Shift--and--rescale  Sets ($\frac{\sqrt{n} }{\log(n)} \an= 1/\sqrt{\log(n)} \rightarrow 0$) & 0.15 & 0.02 & -0.01462 (0.005) \\ 
Shift--and--rescale  Sets ($\frac{\sqrt{n} }{\log(n)} \an = \sqrt{\log(n)}\rightarrow \infty$) & 0.95 & 0.02 & -0.00021 (0.005) \\ \hline
  {\bf Case 3}\hspace{2mm} $\frac{1}{2} + \gamma < \beta + \mu < 1+ 2\gamma$ \\ \hspace{14mm} $\beta = 0.75, \mu = 0.75, \gamma = 0.5$ \\\hline 
  Full Posterior & 0.95 & 0.02 & -0.00061 (0.005) \\ 
  $\an-$Posterior ( $n^{1-\frac{1+2\gamma}{2(\beta + \mu)}}\an= 1/\sqrt{\log(n)} \rightarrow 0$) & 1.00 & 0.40 & -0.04759 (0.005) \\ 
$\an-$Posterior ( $n^{1-\frac{1+2\gamma}{2(\beta + \mu)}}\an= \sqrt{\log(n)} \rightarrow \infty$) & 1.00 & 0.04 & -0.00145 (0.005) \\ 
Shift--and--rescale  Sets ( $n^{1-\frac{1+2\gamma}{2(\beta + \mu)}}\an= 1/\sqrt{\log(n)} \rightarrow 0$) & 0.00 & 0.02 & -0.04759 (0.005) \\ 
Shift--and--rescale  Sets ( $n^{1-\frac{1+2\gamma}{2(\beta + \mu)}}\an = \sqrt{\log(n)} \rightarrow \infty$) & 0.94 & 0.02 & -0.00145 (0.005) \\ 
   \hline & \\ 
   {\bf \color{black} Density Estimation} \\ \hline
    {\color{black} {\bf Case 4}\hspace{2mm} $\frac{1}{2} + \gamma < \beta + \mu < 1+ 2\gamma$ }\\ \hspace{14mm} $\beta = 1, \mu = 1, \gamma = 1$ \\\hline 
  Full Posterior                                                                                        & 0.95 & 0.01 & -0.00089 (0.004)  \\ 
  $\an-$Posterior ( $n^{1-\frac{1+2\gamma}{2(\beta + \mu)}}\an= 1/\sqrt{\log(n)} \rightarrow 0$)               & 0.93 & 0.06 & -0.01267 (0.005)   \\ 
$\an-$Posterior ( $n^{1-\frac{1+2\gamma}{2(\beta + \mu)}}\an= \sqrt{\log(n)} \rightarrow \infty$)              & 0.94 & 0.02 & -0.00233 (0.005)   \\ 
Shift--and--rescale  Sets ( $n^{1-\frac{1+2\gamma}{2(\beta + \mu)}}\an= 1/\sqrt{\log(n)} \rightarrow 0$)       & 0.31 & 0.01 &  -0.01267 (0.005) \\ 
Shift--and--rescale  Sets ( $n^{1-\frac{1+2\gamma}{2(\beta + \mu)}}\an = \sqrt{\log(n)} \rightarrow \infty$)   & 0.92 & 0.01 &  -0.00233 (0.005) \\ 
   \hline
\end{tabular}
\caption{Data pertaining to the credible sets obtained in the three different cases presented in Proposition \ref{prop:corrected_credible_sets_behaviour}. {\color{black}Case 4 represents a similar experiment in density estimation.}}
\label{Tab:credible_sets_coverage}
\end{table}

We first comment on how the lengths of the corrected  sets in each of the cases roughly match the lengths of the credible sets from the full posterior, but that the bias of the centering of the corrected  sets is always larger than the bias of the centering of the full posterior (even in the regimes where $\an$ does not breach the lower bound). It is easy to see why this is the case in this particular model; the bias is $-\sum_{k=1}^\infty\frac{1}{1+n\an \lambda_k}$, which is obviously larger in magnitude for smaller $\an$. The fact that the corrected sets have a larger bias but the same length as those from the full posterior results in a strictly lower coverage, which can be seen in the empirical results. 

Secondly, we observe that the lengths of the shift--and--rescale sets are roughly the same for different choices of $\alpha_n$, so it is purely the bias of the centering which affects the coverage for $\alpha_n$ breaching the lower bound versus $\alpha_n$ respecting the lower bound. This makes sense on inspection of the assumptions of Proposition \ref{prop_corrected_region}, which relies on the posterior mean being within a factor $o_P(1/\sqrt{n})$ of the efficient centering; when $\an$ breaches the lower bound implied by Proposition \ref{prop:corrected_credible_sets_behaviour}, the bias is orders of magnitude larger than when $\an$ respects the lower bound.

Finally, note that the credible sets from the $\an-$posterior always have coverage close to 1, but at the price of being considerably larger than those from the full posterior or the corrected credible sets.

{\it Density estimation.} We empirically illustrate the behaviour of shift--and--rescale sets in density estimation, where exact computations are not possible. We use the same prior, true density and linear functional as the empirical study in Section \ref{sec:gp_bvm}, with $\beta = \gamma = \mu = 1$. We take $n=10,000$ observations and again generate posterior samples by MCMC using the \texttt{sbde} R-package \citep{Tokdar_sbde} with $\al_n=n^{-1/4}/\sqrt{\log(n)}$, $n^{-1/4} \sqrt{\log(n)}$ and $1$.  
{\color{black} We consider the (empirical) 95\% credible intervals and the corresponding shift-and-rescale credible intervals. Figure \ref{fig:corrected_credible_sets_rescaled_se} shows the roughly Gaussian shape of each of the posterior distributions; the comparatively large credible intervals from the \ap (dashed vertical lines); and the fact that the shift--and--rescale intervals (solid vertical lines) and credible interval from the full posterior have approximately the same length, which shows the correction also appears to work well in this more complex setting. For estimates of the coverage of these shift-and-rescale credible sets, see Case 4 in Table \ref{Tab:credible_sets_coverage}. The condition $n^{1-\frac{1+2\gamma}{2(\beta + \mu)}}\alpha_n \rightarrow \infty$ derived for Gaussian white noise in Proposition \ref{prop:corrected_credible_sets_behaviour} seems to be a good guide in this setting as well, with the shift-and-rescale credible sets achieving very small coverage when this condition is breached, but approximately the right coverage when the condition is verified.}

% 	\begin{figure}[h]
%   \centering
%   \includegraphics[scale=0.45]{corrected_alpha_full_credible_sets.pdf}
%   \caption{Top: Histogram of \ap distribution samples of $\psi(f)$ in density estimation using a rescaled squared exponential prior, with empirical credible interval (red) and corrected empirical credible interval (blue). Bottom: A sample from the full posterior distribution in the same setting with empirical credible interval (green). }
%   \label{fig:corrected_credible_sets_rescaled_se}
% \end{figure}

\begin{figure}[h]
  \centering
  \includegraphics[scale=0.55]{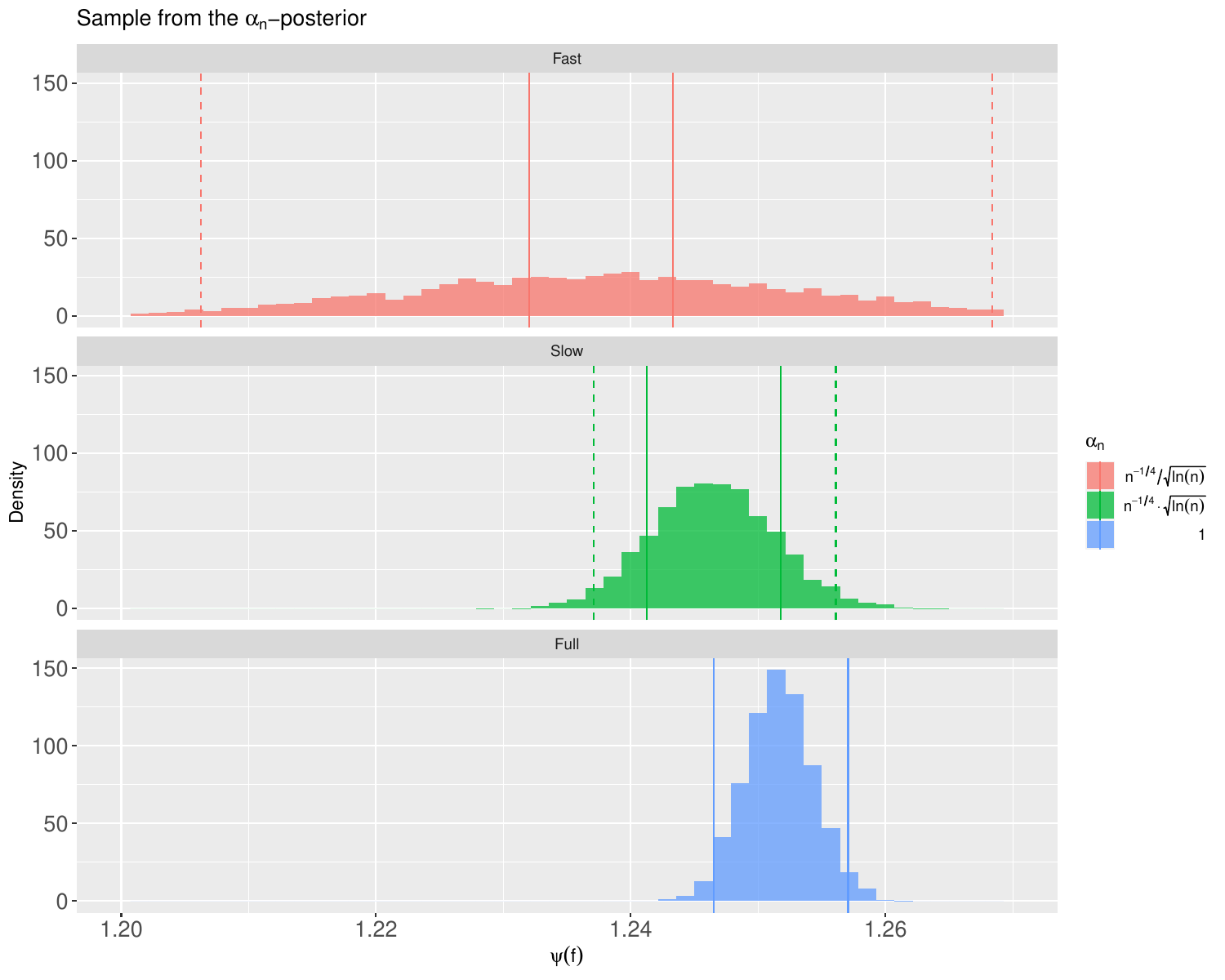}
  \caption{Histograms of \ap distribution samples of $\psi(f)$ in density estimation using a rescaled squared exponential prior, with empirical credible intervals (dashed lines) and shift-and-rescale credible intervals (solid lines). }
  \label{fig:corrected_credible_sets_rescaled_se}
\end{figure}

{\color{black}
\begin{remark}[Multi-dimensional functionals]
Though we do not formally present any multi-dimensional semiparametric BvM results in this paper, we briefly sketch the analogous construction of a multidimensional shift--and--rescale set given a BvM theorem. 
Recall that for a one-dimensional functional, one uses the $\al_n$--posterior quantiles to define the boundary of the credible interval.
%: the $\al_n$--posterior quantiles indeed estimate the appropriate `boundary' of the set. %In higher-dimension, the boundary can be estimated through sampling from the $\al_n$--posterior; for instance, one can fit the ellipse of smallest volume, centered e.g. at the posterior mean, that still gets $95\%$ of the $\al_n$--posterior mass. Another 
In higher-dimensions, a simple possibility is to use a sample from the $\al_n$--posterior to compute its empirical covariance $V_Y$, and use this as a `shape' for the boundary of the credible set. More precisely, for a $d-$dimensional functional, a $(1-\delta)-$credible set from an approximately Gaussian $\mathcal{N}_d(\bar{\psi}, V)$ random variable is approximately
$$
\{\psi : (\psi - \bar{\psi})^T V^{-1}(\psi - \bar{\psi}) \leq \chi_d^2(1-\delta)\},
$$
where $\chi_d^2(1-\delta)$ is the $(1-\delta)-$quantile of the $\chi_d^2$ distribution. The corresponding empirical shift-and-rescale set from the fractional posterior would then  be
$$
\{\psi : (\psi - \bar{\psi})^T V_Y^{-1}(\psi - \bar{\psi}) \leq \alpha_n\chi_d^2(1-\delta)\},
$$
where $V_Y$ is the empirical $\al_n$--posterior covariance. This provides an analogue in dimension $d\ge 1$ of the shift-and-rescale set presented in (26) when $d=1$.
\end{remark}
}

%%%%%%%%%%%%%%%%%%%%%%%%%%%%%%%%%%%%%%%%%%

\section{Contraction Rates for the Fractional Posterior}\label{sec:contraction}

A first step in proving semiparametric BvM results in Section \ref{sect:sp-bvm} is to localize the posterior near the true parameter by establishing a contraction rate. We therefore study nonparametric contraction rates for the $\alpha_n$-posterior distribution with a focus on obtaining the precise dependence on both $n$ and $\alpha_n$, results which are also of independent interest for full nonparametric Bayesian estimation. Given our primary focus is semiparametrics, we will consider common statistical norms which are relevant to this topic, such as $L^p$-distances.

Recall that unlike for the full Bayesian posterior, testing or metric entropy conditions are not needed to obtain contraction rates in the R\'enyi-divergence for the fractional posterior when $\alpha_n<1$ {\color{black} (as derived by \citealp{walkerhjort01} for consistency and \citealp{tongz06} for rates)}, see also \citet{kruijervdv13,BPY,GM20}. Given this result is more flexible than the classic test-based approach for full posteriors, we first examine its implications for some common statistical norms. For $0<\alpha<1$,  the Rényi divergence of order $\alpha$ between two densities $f$ and $g$ on a measurable space $(E, \mathcal{A}, \mu)$ is given by
	\begin{align*}
		D_\alpha(f, g) = -\frac{1}{1-\alpha}\log \left( \int_{E} f^\alpha g^{1-\alpha}d\mu \right).
	\end{align*}
Further define the usual Kullback-Leibler divergence $K(f, g)=\int f \log(f/g) d\mu$ and its $2^{nd}$-variation $V(f,g)= \int f \left(\log(f/g) - K(f,g)\right)^2 d\mu$. It is well-known that posterior contraction rates are related to the prior mass assigned to a Kullback-Leibler type neighbourhood about the true density $p_0^n = p_{\eta_0}^n$:
		\begin{align*}
		B_n(p_{\eta_0}^n, \eps) = B_n(\eta_0,\eps) &=\{\eta \in S: \: K(p_{\eta_0}^n, p_{\eta}^n) \leq n \eps^2,\: V(p_{\eta_0}^n,  p_{\eta}^n) \leq n \eps^2  \},
	\end{align*}	 
see Chapter 8 of \citet{vdv_FNBI}. We first modify Theorem 3.1 of  \citet{BPY} by introducing an explicit dependence on $\alpha_n$ in the `small-ball' probability.
\begin{theorem}\label{bat_result} For any nonnegative sequence $\eps_n$ and $0<\alpha_n<1$ such that $n\alpha_n\eps_n^2 \rightarrow \infty$ and 
		\begin{align}\label{equation_thm_1}
			\Pi(B_n(\eta_0, \eps_n)) \geq e^{-n\alpha_n\eps_n^2},
		\end{align}
		there exists $C>0$ such that as $n\to\infty$,
		\begin{align*}
			\Pi_{\alpha_n}\left( \eta: \: \frac{1}{n}D_{\alpha_n}(p_{\eta}^n, p_{\eta_0}^n) \geq C \frac{\alpha_n\eps_n^2}{1-\alpha_n}   | Y^n\right) = o_P(1).
		\end{align*}	
	\end{theorem}
The last result differs from Theorem 3.1 in \citet{BPY} on two points: first, the required lower bound for the small-ball probability in \eqref{equation_thm_1} takes the form 	$e^{-n\alpha_n\eps_n^2}$ rather than $e^{-n\eps_n^2}$, which is a natural modification in view of the interpretation that the $\alpha_n$-posterior uses effective sample size $n' = n\alpha_n$; second, the obtained rate in terms of $D_{\al_n}(p_{\eta}^n, p_{\eta_0}^n)/n$ is $C\al_n\eps_n^2/(1-\al_n)$ instead of $C\eps_n^2/(1-\al_n)$ (importantly, note that the sequences $\eps_n$ in both rates may be different since the small-ball probability condition is different, see below for more details). We illustrate the difference between these approaches in the next examples.
%
%We emphasize that the smallest $\eps_n$ satisfying our condition \eqref{equation_thm_1} depends implicitly on $\alpha_n$, and thus is not directly comparable to the `$\eps_n$' in Theorem 3.1 of \citet{BPY}, which is independent of $\alpha_n$. 
%The difference in these approaches will be illustrated in the next example.}
%
Note that in interpreting the rate in Theorem \ref{bat_result},   one needs to take care of the dependence of $D_{\alpha_n}$ on the exponent $\alpha_n$. In typical examples for iid models, this scales as $n\al_n$ times squared individual distances between densities. In the Gaussian white noise model for instance, one can directly compute $D_{\alpha_n}(f,f_0) = \frac{n\alpha_n}{2}\|f-f_0\|_2^2$, so that the conclusion of the last statement becomes
		\begin{align*}
			\Pi_{\alpha_n}\left( f: \: \|f-f_0\|_2 \geq C \frac{\eps_n}{\sqrt{1-\alpha_n}}   | Y^n\right) = o_P(1).
		\end{align*}	
Consider for simplicity the case of a $\beta$-smooth Gaussian process with $\beta$-smooth truth $f_0$, in which case condition \eqref{equation_thm_1} above yields the choice $\eps_n = \eps_{n,\alpha_n} = (n\alpha_n)^{-\frac{\beta}{2\beta+1}}$ (see Section \ref{sec:contract_GP} below for precise statements). In this case, Theorem \ref{bat_result} gives $L^2$-rate $\eps_{n,\alpha_n} (1-\alpha_n)^{-1/2} = (n\alpha_n)^{-\frac{\beta}{2\beta+1}}(1-\alpha_n)^{-1/2}$, while Theorem 3.1 of \citet{BPY} implies rate $\eps_{n,1} \alpha_n^{-1/2} (1-\alpha_n)^{-1/2} = n^{-\frac{\beta}{2\beta+1}} \alpha_n^{-1/2} (1-\alpha_n)^{-1/2}$. In particular, for all $\beta>0$ and $0<\alpha_n<1$, the former gives a better dependence on $\alpha_n$, particularly in the small $\alpha_n$ regime. A similar conclusion holds in density estimation with $L^1$-loss, where one has $D_{\alpha_n}(f^n,f_0^n) \geq n\alpha_n \|f-f_0\|_1^2/2$ (\citealp{van_Erven_2014}, Theorem 31) for $f^n(x) = \prod_{i=1}^n f(x_i)$ the $n$-fold product density of $f$, thereby giving the same rates as for $L^2$-loss in Gaussian white noise as just above. Thinking of $\al_n$'s that go to zero polynomially in $n$ (e.g. $\al_n=n^{-1/4}$), one sees that the improvement is polynomial in $n$ in these examples.

\begin{remark}
	One can also more generally compare the rates obtained by the two approaches. Denote $f(\eps):=f_n(\eps)=\Pi(B_n(\eta_0, \eps))$ and $g(\eps)=e^{-n\eps^2}$ and suppose to fix ideas that the equations $f(\eps_n) = e^{-n\alpha_n\eps_n^2}$ and $f(\bar{\eps}_n) = e^{-n\bar{\eps}_n^2}$ have unique solutions $\eps_n, \bar{\eps}_n$. By definition $(f-g)(\bar{\eps}_n)=0$ while $f(\eps_n)-g(\eps_n)=e^{-n\alpha_n\eps_n^2}-e^{-n\eps_n^2}>0$, so that $\eps_n\ge \bar\eps_n$ using that $f-g$ is non-decreasing. In particular, $f(\eps_n)\ge f(\bar\eps_n)$ which leads to $\an\eps_n^2 \leq \bar{\eps}_n^2$, implying that the rate provided by Theorem \ref{bat_result} is in that case, up to constants, at least as fast as that of Theorem 3.1 of \citet{BPY} (and, as the examples above show, sometimes the improvement is polynomial).
\end{remark}
Note that the above rates deteriorate as $\alpha_n \to 1$, i.e. convergence to the full posterior. This is not surprising since contraction rates for the full posterior typically require additional conditions, such as testing or bounded entropy conditions. Indeed, \citet{BSW99} provide a counterexample of a prior which satisfies the small ball condition \eqref{equation_thm_1} with $\alpha_n=1$ but not a related entropy condition. They show the full posterior is inconsistent (\citealp{BSW99}, Section 3.5), whereas the fractional posterior converges to the truth at rate at least $(1-\alpha)^{-1} n^{-1/3}$ when $\alpha \in (0,1)$ is fixed \citep{BPY}. This counterexample shows that one must exploit additional regularity properties of a prior beyond the prior mass condition \eqref{equation_thm_1} to ensure good behaviour as $\alpha_n \to 1$. Note that taking a sequence $\alpha_n \to 1$ is also relevant to certain practical Bayesian computational algorithms, for instance fractionally weighting (tempering) parallel distributions can improve sampling convergence and yield faster mixing times \citep{GT95} or in some empirical Bayes methods \citep{MT20}.

We therefore present a second $\alpha_n$-posterior convergence result  following the testing approach of \citet{GV_2007}, which removes the necessity that $\alpha_n < 1$ at the expense of an extra testing condition needed to control the complexity of the prior support.
Theorem 1 of \citet{GV_2007} extends to the $\alpha_n$-posterior using the same proof technique as for the full posterior.

%	{\bf Testing Condition (T)}\\
%	We say that condition (T) is satisfied if there exist constants $K > 0$ and $a \in (0, 1)$ such that for all $\eps > 0$, if $\eta_0, \eta_1 \in S$ satisfy $d(\eta_0, \eta_1) > \eps$ then there exists tests $\varphi_n$ with:
%	\begin{align*}
%		E_{\eta_0} \varphi_n \leq e^{-Kn\eps^2}, \hspace{5mm}
%		\sup_{\eta: d(\eta, \eta_1) < a\eps} E_\eta(1-\varphi_n) \leq e^{-Kn\eps^2}
%	\end{align*}

	\begin{theorem}\label{thm_GGV} 
Let $d$ be a metric on the parameter space $S$ and $\eta_0 \in S$. Suppose that there exist universal constants $K,a > 0$ such that for all $\eps > 0$ and all $\eta_1 \in S$ satisfying $d(\eta_0, \eta_1) > \eps$, there exist tests $\varphi_n$ satisfying
	\begin{align}\label{eq:tests}
		E_{\eta_0} \varphi_n \leq e^{-Kn\eps^2}, \hspace{2cm}
		\sup_{\eta\in S: d(\eta, \eta_1) < a\eps} E_\eta(1-\varphi_n) \leq e^{-Kn\eps^2}.
	\end{align}	
	Let $\Pi = \Pi_n$ be a prior on $S$, and $\eps_n,\tilde{\eps}_n$ and $0<\alpha_n \leq 1$ be nonnegative sequences such that $n\alpha_n\tilde{\eps}_n^2 \rightarrow \infty$. Suppose further that there exist constants $C,D>0$ and subsets $S_n\subset S$ satisfying
		\begin{enumerate}
			\item $N(\eps_n, S_n, d) \leq e^{Dn\eps_n^2}$, \label{assum_1_GGV}
			\item $\Pi(S_n^c) \leq e^{-(C+3)n\alpha_n\tilde{\eps}_n^2}$, \label{assum_2_GGV}
			\item $\Pi(B_n(\eta_0, \tilde{\eps}_n)) \geq e^{-Cn\alpha_n\tilde{\eps}_n^2}$. \label{assum_3_GGV}
		\end{enumerate}
		Then there exists $M> 0$ such that as $n\to\infty$,
		\begin{align*}
		    \Pi_{\alpha_n}(\eta: d(\eta,\eta_0)\geq M(\eps_n \vee \tilde{\eps}_n)|Y^n) \rightarrow^{P_0} 0 .
		\end{align*}
	\end{theorem}

In the i.i.d. density estimation model, the testing condition \eqref{eq:tests} is satisfied for instance by the Hellinger metric, $L^1$-distance or, for a bounded set of densities, by the $L^2$-distance (\citealp{vdv_FNBI}, Proposition D.8). It similarly extends to Gaussian white noise with the $L^2$-distance (\citealp{vdv_FNBI}, Lemma D.16) and various other non-i.i.d. models such as nonparametric regression, Markov chains and times series, see Chapter 8.3 in \citet{vdv_FNBI}. Having two sequences $\eps_n$ and $\tilde{\eps}_n$ adds flexibility to the approach, which can prove useful in certain non-i.i.d. models.

Returning to the $\beta$-smooth Gaussian process example and assuming for simplicity that $\eps_n \simeq \tilde{\eps}_n \simeq n^{-\frac{\beta}{2\beta+1}}$, Theorem \ref{thm_GGV} yields rate $(n\alpha_n)^{-\frac{\beta}{2\beta+1}}$ compared with the slower rate $(n\alpha_n)^{-\frac{\beta}{2\beta+1}}(1-\alpha_n)^{-1/2}$ from Theorem \ref{bat_result}. In particular, the former rate gains significantly when $\alpha_n \to 1$ and fully matches the original parametric intuition that the fractional posterior uses effective sample size $n' = n\alpha_n$.

We now apply these general results to the concrete examples of histograms and Gaussian process priors. In all cases we use the sharper rate from Theorem \ref{thm_GGV} since these priors satisfy the required entropy conditions.
	
	\begin{proposition}[Histogram prior]\label{rate_denti_PRH}
		Consider density estimation on [0,1] with true density $f_0 \in \mathcal{C}^\beta([0,1])$ for some $\beta \in (0,1]$, bounded away from 0. Let $\Pi = \Pi_n$ denote the histogram prior \eqref{def::RHP} satisfying $K_n=o\left(n\alpha_n/ \log(n\alpha_n)\right)$ and $\frac{1}{(n\alpha_n)^b}\leq \delta_{i,n} \leq 1$ for $i = 1,\dots, K_n$ for some $b>0$. Then there exists $C>0$ such that as $n\to\infty$,
		\begin{align*}
			\Pi_{\alpha_n} \left( f:  \|f-f_0\|_1 \geq C \left( \frac{K_n\log(n\alpha_n K_n)}{n\alpha_n} +\frac{1}{K_n^{2\beta}}\right)^{\frac{1}{2}}  \bigg| Y^n \right) \xrightarrow{P_0} 0.
		\end{align*}
	\end{proposition}
	
As expected, the rate in the last proposition matches that for the full posterior but with the role of the sample size $n$ replaced by the effective sample size $n' = n\alpha_n$ (cf. Equation 4.8 in \citealp{CR2015}). Note that the optimal choice $K_n^* \simeq (\log (n\alpha_n)/(n\alpha_n))^{\frac{1}{2\beta+1}}$ that balances the two terms in the rate also depends on $\alpha_n$ and hence will not match the optimal truncation for the true posterior. This follows since the fractional posterior inflates the variance without significantly affecting the bias in the well-specified setting considered here. We further remark that the prior conditions required in Proposition \ref{rate_denti_PRH} become more stringent as $\alpha_n \to 0$, though one may always take $K_n \to \infty$ since $n\alpha_n \to \infty$ by assumption.

\subsection{Contraction rates for Gaussian process priors}\label{sec:contract_GP}

As mentioned in Section \ref{sec:gp_bvm} above, for a mean-zero Gaussian process $W$ viewed as a Borel-measurable map in a Banach space $(\mathbb{B},\|\cdot\|)$ with corresponding RKHS $(\mathbb{H},\|\cdot\|_{\mathbb{H}})$, the corresponding contraction rates are related to the behaviour of the concentration function, $\varphi_w$. This connection is made explicit in Theorem 2.1 of \citet{vdv}, which characterizes rates such that a Gaussian prior places sufficient mass about a given truth and concentrates on sets of bounded complexity. These conclusions are in terms of the Banach-space norm $\|\cdot\|$, which must then be related to concrete distances in standard statistical settings.

The following result extends Theorem 2.1 of \citet{vdv} to the fractional posterior by considering the solution to the equation $\varphi_{\eta_0}(\eps_n) \sim n\alpha_n \eps_n^2$, i.e. using the effective sample size on the right-hand side, see \eqref{eq:conc_eqn}. Since it is well-established that the support of a Gaussian process $W$ equals the closure of its RKHS $\mathbb{H}$ under the underlying Banach space norm $\|\cdot\|$, we require the true parameter $\eta_0$ to lie in this space.

\begin{lemma}\label{thm:GGV_verification_GP}
Let $W$ be a mean-zero Gaussian random element in a separable Banach space $(\mathbb{B},\|\cdot\|)$ with associated RKHS $(\mathbb{H},\|\cdot\|_{\mathbb{H}})$, and suppose $\eta_0$ lies in $\bar{\mathbb{H}}$, the closure of $\mathbb{H}$ in $\mathbb{B}$. If $\eps_n > 0$ and $\alpha_n > 0$ satisfy
$\varphi_{\eta_0}(\eps_n) \leq n\an \eps_n^2,$
then for any $C > 1$ with $Cn\an\eps_n^2 > \log 2$, there exist measurable sets $B_n \subset \mathbb{B}$ such that
\begin{align*}
	\log N(3\eps_n, B_n, \|\cdot\|) &\leq 6Cn\an \eps_n^2, \\
	P(W \notin B_n) &\leq e^{-Cn\an\eps_n^2}, \\
	P(\|W-\eta_0\| < 2\eps_n) &\geq e^{-n\an\eps_n^2}.
\end{align*}
\end{lemma}

Lemma \ref{thm:GGV_verification_GP} involves the Banach space norm $\|\cdot\|$, which is related to statistically relevant norms and divergences in both Gaussian white noise and density estimation in \citep{vdv}. We will shortly make this correspondence explicit in Propositions \ref{thm:GGV_verification_GP_gwn} and \ref{thm:GGV_verification_GP_density_estimation} below. However, given our interest in the precise role of the fractional parameter $\alpha_n$,  we first study corresponding lower bounds for the contraction rate. For Gaussian process priors, this has been studied in \citet{C2008}, where it is established that a lower bound on the concentration function in turn implies a lower bound on the contraction rate.

%{\it Remark} \\
%Theorem \ref{thm:GGV_verification_GP} provides a tool to obtain contraction rates for the \ap distribution when using a Gaussian process prior - giving three inequalities which can be related one-to-one with those in the assumptions of Theorem \ref{thm_GGV}. Theorems \ref{thm:GGV_verification_GP_gwn} and \ref{thm:GGV_verification_GP_density_estimation} below make the correspondence explicit in two problem settings.
%Before obtaining rates in explicit examples, however, we first display the following result which provides a tool for obtaining lower bounds on the contraction rate of the \ap.

\begin{lemma}[Lower bound for contraction rate]\label{thm:lower_bounds}
Let $W$ be a mean-zero Gaussian random element in a separable Banach space $(\mathbb{B},\|\cdot\|)$ with associated RKHS $(\mathbb{H},\|\cdot\|_{\mathbb{H}})$, and suppose $\eta_0$ lies in $\bar{\mathbb{H}}$, the closure of $\mathbb{H}$ in $\mathbb{B}$. Suppose $\eps_n \rightarrow 0$, $0<\alpha_n\leq 1$ such that $n\alpha_n\eps_n^2\to\infty$ satisfy
$
\Pi\left(B_{n}(\eta_0, \eps_n) \right) \geq e^{-cn\an\eps_n^2}
$
for some $c > 0$.
If $\delta_n \rightarrow 0$ satisfies
$
\varphi_{\eta_0}(\delta_n) \geq (2+c)n\an\eps_n^2,
$
%then $\delta_n$ is a lower bound for the contraction rate of the $\alpha_n-$posterior. That is,
then as $n\to\infty$,
$$
\Pi_\an(\eta : \|\eta-\eta_0\| \leq \delta_n | Y^n) \rightarrow^{P_0} 0.
$$
%in $P_0^n-$probability for small enough $m$.
\end{lemma}

Note that Lemma \ref{thm:lower_bounds} yields a lower bound on the posterior contraction rate for the parameter $\eta$ to which the Gaussian process is assigned, and in the underlying Banach space norm $\|\cdot\|$, which need not match the desired statistical distance. We now specialize the above results to our two concrete models.

\begin{proposition}[Contraction rates in Gaussian white noise]\label{thm:GGV_verification_GP_gwn}
Consider the Gaussian white noise model and let the prior on $f$ be a mean-zero Gaussian random element $W$ in $L^2[0,1]$ with associated RKHS $\mathbb{H}$. If the true parameter $f_0$ lies in the support of $W$ and $\eps_n \rightarrow 0$ satisfies
$
\varphi_{f_0}(\eps_n) \leq n\an \eps_n^2,
$
then for some $M > 0$ large enough,
$$
\Pi_\an(f : \|f-f_0\|_2 > M\eps_n | Y^n) \rightarrow^{P_0} 0,
$$
as $n\to\infty$. Moreover, if $\varphi_{f_0}(\delta_n) \geq \frac{9}{4} n\an\eps_n^2$, then for sufficiently small $m>0$ and as $n\to\infty$,
$$\Pi_{\alpha_n} (f:\|f-f_0\|_2 \leq m \delta_n |Y^n) \to^{P_0} 0.$$
\end{proposition}

In the white noise model, one can consider $W$ as a random element of $L^2[0,1]$, so that the norms for the upper and lower bounds in Proposition \ref{thm:GGV_verification_GP_gwn} match. This is no longer the case in density estimation.

\begin{proposition}[Contraction rates in density estimation]\label{thm:GGV_verification_GP_density_estimation}
Consider density estimation on $[0,1]$ and assign to the density $f$ a prior of the form \eqref{def:gp_inducement}, where $W$ is a mean-zero Gaussian random element in $L^\infty[0,1]$ with associated RKHS $\mathbb{H}$.  If the true parameter $\eta_0 = \log f_0$ lies in the support of $W$ and $\eps_n \rightarrow 0$ satisfies
$
\varphi_{\eta_0}(\eps_n) \leq n\an \eps_n^2,
$
then for $M > 0$ large enough, as $n\to\infty$,
$$
\Pi_\an(f : \|f-f_0\|_1 > M\eps_n | Y^n) \rightarrow^{P_0} 0.
$$
Moreover, there exists $C_1 > 0$ a finite constant such that if $\varphi_{\eta_0}(\delta_n) \geq C_1 n\an\eps_n^2$, then for sufficiently small $m>0$ and as $n\to\infty$,
$$\Pi_{\alpha_n} (f:\|f-f_0\|_\infty \leq m \delta_n |Y^n) \to^{P_0} 0.$$
\end{proposition}

One typically expects the rates in $L^1$ and $L^\infty$ to match up to a logarithmic factor in $n$, so $\eps_n$ and $\delta_n$ in the last proposition should heuristically be of the same polynomial order. However, a lower bound in $L^\infty$ does not strictly imply one in the weaker $L^1$-norm and hence there is a genuine mismatch here. We next apply the above results to the concrete examples of Gaussian priors considered above.

\begin{corollary}\label{corrollary:gp_contraction_rate_examples}
Let $W$ be one of the mean-zero Gaussian process described in Examples \ref{ex:series}-\ref{ex:RL} with regularity parameter $\gamma>0$, considered as a random element in $L^p[0,1]$ with associated concentration function $\varphi_{\eta_0}$. Then $\eps_n\to 0$ satisfies $\varphi_{\eta_0}(\eps_n) \leq n\an \eps_n^2$ in the following cases.
\begin{itemize}
\item[(i)] Infinite series prior (Example \ref{ex:series}) with $p = 2$, $\eta_0 \in \mathcal{H}^\beta$ and $\eps_n \asymp (n\an)^{-\frac{\gamma\wedge \beta}{1+2\gamma}}$.
\item[(ii)] Matérn process (Example \ref{ex:Matern}) with $p=\infty$, $\eta_0 \in C^\beta$ and $\eps_n \asymp (n\an)^{-\frac{\gamma \wedge \beta}{1+2\gamma}}$.
\item[(iii)] Rescaled square exponential process (Example \ref{ex:SE}) with $p=\infty$, $\eta_0 \in C^\beta$ and $\eps_n \asymp \left(\frac{n\an}{\log^2(n\an)} \right)^{-\frac{\gamma \wedge \beta}{1+2\gamma}}$.
\item[(iv)]Riemann-Liouville process (Example \ref{ex:RL}) with $p = \infty$, $\eta_0 \in C^\beta$ and
$$
\eps_n \asymp \begin{cases}(n\an)^{-\frac{\gamma \wedge \beta}{1+2\gamma}} & \textrm{ if } \gamma \leq \beta \textrm{ or } \lfloor \gamma \rfloor = \frac{1}{2} \textrm{ or } \gamma \notin \beta + \frac{1}{2} + \mathbb{N} \\
\left(\frac{n\an}{\log(n\an)}\right)^{-\frac{\gamma \wedge \beta}{1+2\gamma}} & \textrm{ otherwise.}\end{cases}
$$
\end{itemize} 
In particular, such $\eps_n$ give a contraction rate for the $\alpha_n$-posterior distribution in $\|\cdot\|_2$-loss in Gaussian white noise (cases (i)-(iv)) or in $\|\cdot\|_1$-loss in density estimation (cases (ii)-(iv)).
\end{corollary}

In all cases, we recover the `usual' contraction rate with the sample size $n$ replaced by the effective sample size $n\alpha_n$, mirroring the parametric situation. A natural question is whether these rates are sharp, which can be investigated via Lemma \ref{thm:lower_bounds} by lower bounding the concentration function $\varphi_{\eta_0}(\eps_n)$. This is a more delicate issue for which less is known, but we consider two representative examples which can be proved as in  \citet{C2008}. The goal is to find $\delta_n$ as large as possible such that
$$
E_{f_0} \Pi_\an(f : \|f-f_0\|_p \leq m\delta_n | Y^n) \rightarrow 0,
$$
and evaluate the gap between $\delta_n$ and the rate $(n\alpha_n)^{-\frac{\gamma \wedge \beta}{1+2\gamma}}$ (possibly up to $\log (n\alpha_n)$-factors) from Corollary \ref{corrollary:gp_contraction_rate_examples}.
\begin{itemize}
\item \textit{Infinite series prior} (Example \ref{ex:series}) with regularity $\gamma>0$ and $p=2$ in Gaussian white noise. If $\gamma \leq \beta$ (undersmoothing case), then for any $f_0 \in \mathcal{H}^\beta$, we may take $\delta_n \gtrsim (n\alpha_n)^{-\frac{\gamma}{1+2\gamma}}$. If $\gamma > \beta$ (oversmoothing case), then there exists $f_0 \in \mathcal{H}^\beta$ such that for $t>1+\beta/2$, we may take $\delta_n \gtrsim (n\alpha_n)^{-\frac{\beta}{2\gamma+1}} (\log (n\alpha_n))^{-t}$.
\item \textit{Brownian motion released at zero} in density estimation with $p=\infty$. Consider $W(x) = Z_0 + B_x$ for $B$ a standard Brownian motion, $Z_0 \sim \cN(0,1)$ independent and the expontiated prior \eqref{def:gp_inducement}. This corresponds to the Riemann-Liouville process (Example \ref{ex:RL}) with $\gamma = 1/2$, but with a slight correction to the polynomial term. If $f_0 \in C^\beta$ for $\beta \geq 1/2$ (undersmoothing case), then we may take $\delta_n \gtrsim (n\alpha_n)^{-1/4}$, which equals $(n\alpha_n)^{-\frac{\gamma}{1+2\gamma}}$ with $\gamma=1/2$.
\end{itemize}

In these two examples, the upper and lower bounds match, possibly up to logarithmic factors, indicating that our results capture the correct dependence on $\alpha_n$ in the nonparametric contraction rate for the fractional posterior. This matches a similar conclusion in the parametric setting \citep{miller21,MORV22}.

\subsection{Supremum norm contraction rates in Gaussian white noise}

The two general approaches to posterior contraction used above are known to yield suboptimal rates in losses such as $L^\infty$, which are incompatible with the intrinsic distance that geometrizes the statistical model (e.g. the Hellinger distance in density estimation), see \citet{HRSH15}. An alternative method is to express such a loss in terms of multiple functionals, usually involving basis coefficients, and then apply tools from semiparametric BvM results \textit{uniformly} over these functionals \citep{cas_14}. We follow the program of \citet{cas_14} and show that this approach extends to the fractional posterior setting in Gaussian white noise.

Let $(\psi_{lk})$ denote a boundary corrected $S$-regular orthonormal wavelet basis of $L^2[0,1]$, see \citet{HKGP98} for full details and definitions. Consider the Besov ball
$$B_{\infty\infty}^\beta(R) = \left\{ f \in L^2[0,1]: \sup_{l \geq 0} \sup_{0 \leq k \leq 2^l-1} |\langle f,\psi_{lk}\rangle_2| \leq R 2^{-l(\beta+1/2)} \right\}.$$
The space $B_{\infty\infty}^\beta$ is equivalent to the usual H\"older space $C^\beta$ for non-integer $\beta$, while for integer $\beta$ it is slightly larger, satisfying the continuous embedding $C^\beta \subset B_{\infty\infty}^\beta$. We consider a wavelet series prior of the form
\begin{equation}\label{eq:wavelet_prior}
f(x) = \sum_{l\geq 0} \sum_{k=0}^{2^l-1} \sigma_l \zeta_{lk} \psi_{lk}(x),
\end{equation}
where $\zeta_{lk} \sim^{iid} \varphi$ from some density $\varphi$ on $\R$ and $\sigma_l>0$ is a scaling factor.
%and $L_n$ is the largest integer such that $2^{L_n} \leq (n\alpha_n/\log(n\alpha_n))^{\frac{1}{2\beta+1}}$, i.e. the usual truncation point for estimating a $\beta$-smooth function in $L^\infty$-loss based on $n' = n\alpha_n$ observations. 

	\begin{proposition}\label{prop:rate_sup_norm_P1}
		Let $f_0 \in B_{\infty\infty}^\beta(R)$ for some $\beta,R>0$, and consider the wavelet series prior \eqref{eq:wavelet_prior} with (i) $\varphi$ equal to the uniform $\text{Unif}[-B,B]$ density for some $B>R$ and $\sigma_l = 2^{-l(\beta+1/2)}$ or (ii) $\varphi$ equal to a density that is positive on $[-1,1]$ and satisfies the tail condition
			\begin{align}\label{prior_2}
		c_1e^{-b_1 |x|^{1+\delta}} \leq \varphi(x) \leq c_2e^{-b_2 |x|^{1+\delta}}  \qquad \text{for all   } |x| \geq 1,
	\end{align}		
for some $b_1,b_2,c_1,c_2,\delta>0$ and $\sigma_l = 2^{-l(\beta+1/2)}(l+1)^{-\frac{1}{1+\delta}}$. Then there exists $M>0$ large enough such that
		\begin{align*}
			E_0 \int \ninf{f-f_0} d\Pi_{\alpha_n}(f|Y^n) \leq M \left(\frac{\log(n\alpha_n)}{n\alpha_n}\right)^{\frac{\beta}{2\beta +1}}.
		\end{align*}
	\end{proposition}
	
	The conclusion of the proposition is in $E_0$-expectation, which is slightly stronger than the usual notion of a posterior contraction rate and readily implies the latter via Markov's inequality.
Proposition \ref{prop:rate_sup_norm_P1} thus shows that contraction rates in stronger norms, such as the $L^\infty$-norm, satisfy the same heuristic messages derived above, namely that nonparametric contraction rates use the effective sample size. Note that for $g$ a $\mathcal{N}(0,1)$ density, which is covered by the last result, the prior \eqref{eq:wavelet_prior} reduces to a mean-zero Gaussian process with covariance kernel $K(x,y) = \sum_{l\leq L_n,k} 2^{-l(2\beta+1)} \psi_{lk}(x) \psi_{lk}(y)$.

The uniform use of the semiparametric tools developed here can also be used to establish full nonparametric BvM results in weaker topologies which permit estimation at rate $\sqrt{n\alpha_n}$ \citep{CN14}. We mention that such results can provide frequentist coverage guarantees for certain Bayesian credible sets for the full infinite-dimensional parameter as well, although we do not pursue such extensions here.

\acks{The authors would like to thank three reviewers for helpful comments, Surya Tokdar for providing early access to the \texttt{sbde} R-package, and the Imperial College London-CNRS PhD Joint Programme for funding to support this collaboration and travel between the Sorbonne Université and Imperial College London.  ALH is funded by a CNRS--Imperial College PhD grant. IC acknowledges funding from the Institut Universitaire de France and ANR grant project BACKUP ANR-23-CE40-0018-01. }

%%%%%%%%%%%%%%%%%%%%%%%%%%%%%%%%%%%%%%%%

\appendix
\section{Proofs of Main Results}

\subsection{Contraction Rates}

\begin{proof}{Proof of Theorem \ref{bat_result}}
		By Lemma \ref{lem:mino_den_posterior}, on a subset $C_n$ of $P_0$-probability at least $1-\frac{1}{n\eps_n^2}$,  for any measurable set $A \subset S$,
		\begin{equation}\label{eq:Bayes_formula}
		\begin{split}
			E_0 \Pi_{\alpha_n}(A|Y^n)  = E_0 \frac{\int_{A} \frac{p_\eta^n(Y^n)^{\alpha_n}}{p_{\eta_0}^n(Y^n)^{\alpha_n}} d\Pi(\eta)}{\int \frac{p_\eta^n(Y^n)^{\alpha_n}}{p_{\eta_0}^n(Y^n)^{\alpha_n}} d\Pi(\eta)}
			& \leq E_0 \frac{\int_{A} \frac{p_\eta^n(Y^n)^{\alpha_n}}{p_{\eta_0}^n(Y^n)^{\alpha_n}} d\Pi(\eta)}{\Pi(B_n(\eta_0, \eps_n)) e^{-2{\alpha_n}n\eps_n^2}}1_{C_n} + P_0(C_n^c) \\
			& =  \frac{\int_{A} \int p_\eta^n(x)^{\alpha_n} p_{\eta_0}^n(x)^{1-\alpha_n} d\mu(x)d\Pi(\eta)}{\Pi(B_n(\eta_0, \eps_n)) e^{-2{\alpha_n}n\eps_n^2}} +o(1),
		\end{split}		
		\end{equation}
		where the last equality follows from Fubini's theorem. Set
		\begin{align*}
			A_n&:= \left\{\eta : \: \int p_\eta^n(x)^{\alpha_n} p_{\eta_0}^n(x)^{1-\alpha_n} d\mu(x) \leq e^{- 4n\alpha_n\eps_n^2} \right\} \\
			&= \left\{\eta : \: -\frac{1}{n(1-\alpha_n)}\log\left(\int p_\eta^n(x)^{\alpha_n} p_{\eta_0}^n(x)^{1-\alpha_n} d\mu(x) \right) \geq 4\frac{\alpha_n\eps_n^2}{1-\alpha_n}\right\}\\
			& = \left\{\eta : \: \frac{1}{n} D_{\alpha_n}(p_{\eta}^n, p_{\eta_0}^n) \geq 4\frac{\alpha_n\eps_n^2}{1-\alpha_n}\right\}.
		\end{align*}		
Substituting $A_n$ into the second-last display and using the small-ball assumption \eqref{equation_thm_1} yields
		\begin{align*}
			E_0 \Pi_{\alpha_n}(A_n|Y^n) 
			& \leq \frac{\int_{A_n} e^{- 4n\alpha_n\eps_n^2} d\Pi(\eta)}{\Pi(B_n(\eta_0, \eps_n)) e^{-2{\alpha_n}n\eps_n^2}} +o(1)  \leq e^{- n\alpha_n\eps_n^2}  + o(1) = o(1),
		\end{align*}
		since $n\alpha_n\eps_n^2 \rightarrow \infty$.
	\end{proof}

	\begin{proof}{Proof of Theorem \ref{thm_GGV}}
		Denote $\bar{\eps}_n = \eps_n\vee\tilde{\eps}_n $ and note that Assumption 1 of the theorem is also satisfied for the sequence $\bar{\eps}_n$. Then this assumption together with the testing condition imply that there exists $M>0$ and tests $\psi_n$ such that $E_{\eta_0}(\psi_n(Y^n)) = o(1)$ and $\sup \limits_{\eta \in S_n, d(\eta,\eta_0) \geq M\bar{\eps}_n}E_{\eta}(1-\psi_n(Y^n)) \leq e^{-(C+3)n\bar{\eps}_n^2}$.
		Assumptions 2 and 3 and Lemma \ref{lemma_sieve} yield that $ \Pi_{\alpha_n}(S_n^c |Y^n) \to^{P_0} 0 $ and consequently, setting $A_n := \{\eta, d(\eta,\eta_0)\geq M\bar{\eps}_n\}$,
		\begin{align*}\label{eq_2}
			\Pi_{\alpha_n}(A_n | Y^n)
			&= \Pi_{\alpha_n}(A_n \cap S_n|Y^n)\psi_n(Y^n) + \Pi_{\alpha_n}(A_n \cap S_n|Y^n)(1-\psi_n(Y^n)) + \Pi_\an(A_n \cap S_n^c | Y^n) \nonumber \\
			&\leq \psi_n(Y^n) + \Pi_{\alpha_n}(A_n \cap S_n|Y^n)(1-\psi_n(Y^n)) + o_P(1) \nonumber \\
			&=  \Pi_{\alpha_n}(A_n \cap S_n|Y^n)(1-\psi_n(Y^n))	+ o_P(1).
		\end{align*}
		By Lemma \ref{lem:mino_den_posterior}, for a subset $C_n$ of $P_0$-probability at least $1-\frac{1}{n\bar{\eps}_n^2}$ and arguing as in the proof of Theorem \ref{bat_result} just above, we have
		\begin{align*}
			E_0 \Pi_{\alpha_n}(A_n \cap S_n|Y^n)(1-\psi_n(Y^n))
%			&= \frac{\int_{A_n \cap S_n} \frac{p_\eta^n(Y^n)^{\alpha_n}}{p_{\eta_0}^n(Y^n)^{\alpha_n}} d\Pi(\eta)}{\int \frac{p_\eta^n(Y^n)^{\alpha_n}}{p_{\eta_0}^n(Y^n)^{\alpha_n}} d\Pi(\eta)}(1-\psi_n(Y^n)) \\
   &\leq E_0 \frac{\int_{A_n \cap S_n} \frac{p_\eta^n(Y^n)^{\alpha_n}}{p_{\eta_0}^n(Y^n)^{\alpha_n}} d\Pi(\eta)}{\Pi(B_n(\eta_0, \bar{\eps}_n)) e^{-2{\alpha_n}n\bar{\eps}_n^2}} (1-\psi_n(Y^n)) 1_{C_n} + P_0(C_n^c).
		\end{align*}
Using Fubini's theorem and Hölder's inequality, the last display is bounded by
		\begin{align*}
%			&\e_{\eta_0}(\Pi_{\alpha_n}(A_n \cap S_n|Y^n)(1-\psi_n(Y^n)))\\
			& \frac{\int_{A_n \cap S_n} \int p_\eta^n(x)^{\alpha_n} p_{\eta_0}^n(x)^{1-\alpha_n} (1- \psi_n(x)) d\mu(x)d\Pi(\eta)}{\Pi(B_n(\eta_0, \bar{\eps}_n)) e^{-{\alpha_n}2n\bar{\eps}_n^2}} + P_{0}(C_n^c)\\ 
			& \quad \leq  \frac{\int_{A_n \cap S_n} \left(\int p_\eta^n(x) (1- \psi_n(x)) d\mu(x)\right)^{\alpha_n} \left(\int p_{\eta_0}^n(x) d\mu(x)\right)^{1-\alpha_n} d\Pi(\eta)}{e^{-Cn\an\bar{\eps}_n^2 }e^{-2n\an\bar{\eps}_n^2}} + o(1),
		\end{align*}
  which is bounded by $ e^{(2+C)n\an\bar{\eps}_n^2}\int_{A_n} e^{-(C+3)n\alpha_n\bar{\eps}_n^2} d\Pi(\eta) + o(1) \leq e^{-n\an\bar{\eps}_n^2} + o(1) = o(1)$.
%		This result combined with \eqref{eq_2} implies that $\Pi_{\alpha_n}(A_n |Y^n) \xrightarrow{P_0} 0$.
	\end{proof}

	\begin{proof}{Proof of Proposition \ref{rate_denti_PRH}}
		The proof is a direct application of Theorem \ref{thm_GGV}. First, the testing condition \eqref{eq:tests} is satisfied in the density estimation model with $d=\|\cdot\|_1$. Then, let us verify the conditions \ref{assum_1_GGV}, \ref{assum_2_GGV} and \ref{assum_3_GGV} for $S_n=H^1_{K_n}$. For Condition \ref{assum_1_GGV}, set $\eps_n=\sqrt{K_n\log(n)/n}$ that satisfies $K_n\log(3K_n^{1/2}/\eps_n) \lesssim n \eps_n^2 $ and thus $\Big(3K_n^{1/2} / \eps_n\Big)^{K_n} \leq e^{Dn\eps_n^2}$ for some $D>0$. By a standard result on the $\eps$-covering number of the unit ball $B_{\|\cdot\|_2}(0_{\R^{K_n}}, 1)$, for $n$ large enough, it follows,
\begin{align*}
	N(\eps_n, H^1_{K_n}, \|\cdot\|_1) &\leq N(\eps_n, S^1_{K_n}, \|\cdot\|_1) \leq N(\eps_n, B_{\|\cdot\|_2}(0_{\R^{K_n}}, 1) , \|\cdot\|_1) \\ 
	& \leq N(\eps_n / K_n^{1/2}, B_{\|\cdot\|_2}(0_{\R^{K_n}}, 1), \|\cdot\|_2) \leq \Big(3K_n^{1/2} / \eps_n\Big)^{K_n} \leq e^{Dn\eps_n^2},
\end{align*} 
and therefore $\eps_n$ satisfies Condition \ref{assum_1_GGV}. For the random histogram prior, we have $\Pi((H^1_{K_n})^C)= 0$ and so Condition \ref{assum_2_GGV} is clearly satisfied. Finally, by Lemma \ref{lemma_prior_mass_condition}, the sequence $\tilde{\eps}_n^2 = K_n\log(n\alpha_nK_n)/ n\alpha_n + K_n^{-2\beta}$ satisfies $\Pi(B_n(f_0, M\tilde{\eps}_n)) \geq e^{-n\alpha_n(M\tilde{\eps}_n)^2}$ for some $M>0$, and thus the result follows from Theorem \ref{thm_GGV}.
	\end{proof}
	
\begin{proof}{Proof of Lemma \ref{thm:GGV_verification_GP}}
	The proof is a straightforward adaptation of the proof of Theorem 11.20 in \citet{vdv_FNBI} to the $\alpha_n-$posterior, and is hence omitted.
\end{proof}

\begin{proof}{Proof of Lemma \ref{thm:lower_bounds}}
By Lemma I.28 of \citet{vdv_FNBI}, the concentration function satisfies
$$\varphi_{\eta_0}(\eps) \leq -\log \Pi( \|W-\eta_0\| \leq \eps) \leq \varphi_{\eta_0} (\eps/2)$$
for any $\eps>0$. In particular, $\Pi(\|W-\eta_0\| \leq \delta_n) \leq e^{-\varphi_{\eta_0}(\delta_n)} \leq e^{-(2+c)n\alpha_n\eps_n^2}$, so that under the lemma hypotheses,
	$$
	\frac{\Pi(\|W-\eta_0\|\leq \delta_n)}{\Pi(B_{KL}(\eta_0, \eps_n))}\leq \frac{e^{-(2+c)n\an\eps_n^2}}{e^{-cn\an\eps_n^2}} \leq e^{-2n\an\eps_n^2} \to 0.
	$$
The result then follows from Lemma \ref{lemma_sieve}.
\end{proof}

\begin{proof}{Proof of Proposition \ref{thm:GGV_verification_GP_gwn}}
In Gaussian white noise, the testing condition \eqref{eq:tests} is satisfied by the likelihood ratio test with the distance $d = \|\cdot\|_2$ (\citealp{vdv_FNBI}, Lemma D.16), and hence it suffices to verify conditions (1)-(3) of Theorem \ref{thm_GGV} in order to apply that theorem. For $\eps_n$ satisfying $\varphi_{f_0}(\eps_n) \leq n\an\eps_n^2$, Lemma \ref{thm:GGV_verification_GP} gives sets $B_n$ satisfying conditions (1)-(2). By Lemma 8.30 of \citet{vdv_FNBI}, the Kullback-Leibler neighbourhoods take the form $B_n(f_0,\eps_n) = \{f: \|f-f_0\|_2 \leq \eps_n\}$ (not to be confused with the $B_n$ from Lemma \ref{thm:GGV_verification_GP}). But then $	\Pi(\|f-f_0\|_2 < 2\eps_n) \geq e^{-n\an\eps_n^2}$ from the third part of Lemma \ref{thm:GGV_verification_GP}, which verifies (3) for $\eps_n$ possibly a multiple of itself. The contraction upper bound thus follows from Theorem \ref{thm_GGV}. For the lower bound, we apply Lemma \ref{thm:lower_bounds} with $c = 1/4$, so that $\delta_n$ satisfying $\varphi_{f_0}(\delta_n) \geq \frac{9}{4}n\an\eps_n^2$ is a lower bound for the contraction rate.
\end{proof}

\begin{proof}{Proof of Proposition \ref{thm:GGV_verification_GP_density_estimation}}
In density estimation, the testing condition \eqref{eq:tests} is satisfied for the Hellinger distance $d_H$ (\citealp{vdv_FNBI}, Proposition D.8), and hence it again suffices to verify conditions (1)-(3) of Theorem \ref{thm_GGV}. By Lemma 3.1 of \citet{vdv}, the squared Hellinger distance, Kullback-Leibler divergence and its $2^{nd}$-variation $V$ between exponentiated densities $f_w$ and $f_v$ of the form \eqref{def:gp_inducement} are each bounded by a multiple of $\|v-w\|_\infty^2$ as soon as $\|v-w\|_\infty \leq D_0$ for some finite constant $D_0<\infty$. Conditions (1)-(2) can thus be verified with $d=\|\cdot\|_\infty$, while for (3) it suffices to show $\Pi(\|W-\log f_0\|_\infty \leq \eps_n) \geq e^{-Cn\alpha_n \eps_n^2}$. These three conditions each follow from Lemma \ref{thm:GGV_verification_GP} for $\eps_n$ satisfying $\varphi_{\log f_0}(\eps_n) \leq n\an\eps_n^2$, so that we have contraction rate $\eps_n$ in Hellinger distance. Since the $L^1$-distance is bounded by a multiple of the Hellinger distance, we get the same contraction rate in $L^1$.
%Turning to the lower bound, while proving (3) just above, we established that 
%$\Pi\left(B_{n}(\eta_0, \eps_n) \right) \geq e^{-Cn\an\eps_n^2}$
%for some $C>0$. Let now $\mathcal{W}_0 = \{w:[0,1] \to \R: w(0) = \log f_0(0)\}$.
%
%
%
%Applying Lemma \ref{thm:lower_bounds} with parameter $\eta = \log f$ equal to the log-density then implies that $\Pi_\an(f : \|\log f-\log f_0\| \leq m\delta_n | Y^n) \rightarrow^{P_0} 0$ for small enough $m>0$.
%For the lower bound, we proceed as in the proof of Theorem 3 of \cite{C2008}. Lemma 2.5 in \cite{vdv_FNBI} allows us to show that $B_{KL}(f_0, \eps_n) \supset \{\|f-f_0\|_\infty \leq \eps_n/\sqrt{c_1}\}$ for some $1 < c_1 \leq 2$, yielding $\Pi(B_{KL(f_0, \eps_n)}) \geq e^{-\frac{1}{c_1}n\an\eps_n^2}$, so we can apply Lemma \ref{lemma_sieve} with $A_n = \{\|f-f_0\|_\infty \leq \delta_n\}$ to show the lower bound once we have shown that $\Pi(\|f-f_0\|_\infty \leq \delta_n) \lesssim \exp(-C_2n\an\eps_n^2)$ for some $C_2 > 2 + \frac{1}{c_1}$. Using the same argument as in the proof of Theorem 3 in \cite{C2008}, though with the adapted assumption $\varphi_{\eta_0}(\delta_n) \geq C_1 n\an \eps_n^2$, we obtain that $\Pi(\|f-f_0\|_\infty) \leq 2C_3e^{-4n\an\eps_n^2}$ with the adapted assumption.
For the lower bound, the proof is similar to the proof of Theorem 3 of \citet{C2008}.
 \end{proof}

 \begin{proof}{Proof of Corollary \ref{corrollary:gp_contraction_rate_examples}}
{\it Case (i): infinite series}. For $\eps>0$ small enough, the centered small ball probability satisfies $\varphi_{0}(\eps) \asymp \eps^{-\frac{1}{\gamma}}$ (Lemma 11.47 in \citealp{vdv_FNBI}), while $\inf_{h\in\mathbb{H}: \| h - \eta_0\|_2 < \eps}\|h\|_\mathbb{H}^2 \lesssim \eps^{-\frac{2\gamma- 2\beta + 1}{\beta}}$ for $\beta \leq \gamma + 1/2$ (the latter quantity is $O(1)$ if $\beta > \gamma +1/2$ since then $\eta_0$ is in the RKHS of $W$ and one may take $h = \eta_0$). We thus have $\varphi_{\eta_0}(\eps_n) \lesssim \eps_n^{-1/\gamma} + \eps_n^{-(2\gamma-2\beta+1)/\beta}$, which can be checked is $O(n\alpha_n\eps_n^2)$ for $\eps_n \asymp (n\an)^{-\frac{\gamma \wedge \beta}{1+2\gamma}}$.

{\it Case (ii): Mat\'ern.} For $\eps>0$ small enough and $\eta_0 \in C^\beta$, we have $\varphi_{\eta_0}(\eps) \lesssim \eps^{-1/\gamma} + \eps^{-(2\gamma-2\beta+1)/\beta}$ by Lemmas 11.36 and 11.37 of \citet{vdv_FNBI}. As in case (i), this is $O(n\alpha_n\eps_n^2)$ for $\eps_n \asymp (n\an)^{-\frac{\gamma \wedge \beta}{1+2\gamma}}$.

{\it Case (iii): squared exponential}. 
Taking the length scale $k_n = \left(\frac{n\an}{\log^2(n\an)}\right)^{-\frac{1}{1+2\gamma}}$, Lemma 2.2 and Theorem 2.4 of \citet{vdv_2007} imply that  for $\eta_0 \in C^\beta$,
$$
\varphi_{w_0}(\eps_n) \lesssim \frac{1}{k_n} \left(\log \frac{1}{k_n \eps_n^2} \right)^2 + \frac{1}{k_n}
$$
if $k_n^\beta \lesssim \eps_n$.
Then $\varphi_{\eta_0}(\eps_n) \lesssim n\an \eps_n^2$ is satisfied for $\eps_n \gtrsim k_n^\beta \vee \frac{\log (n\an)}{\sqrt{n k_n}}$, which has minimal solution
$
\eps_n \asymp \left(\frac{n\an}{\log^2(n\an)}\right)^{-\frac{\gamma \wedge \beta}{1+2\gamma}}.
$

{\it Case (iv) Riemann-Liouville.} For $\eta_0 \in C^{\beta}$, the concentration function satisfies (Theorem 4 of \citealp{C2008})
\begin{equation}\label{eq:RL_concentration}
    \varphi_{\eta_0}(\eps) \lesssim \begin{cases}
	\eps^{-\frac{1}{\gamma}} & 0 < \gamma \leq \beta, \\
	\eps^{-\frac{2\gamma - 2\beta+1}{\beta}} & \gamma > \beta \textrm{ and } (\lfloor \gamma \rfloor=1/2 \textrm{ or }\gamma \notin \beta + 1/2 + \mathbb{N}), \\
	\eps^{-\frac{2\gamma - 2\beta+1}{\beta}}\log(1/\eps) & \textrm{otherwise}.
\end{cases}
\end{equation}
In the first two cases, $\varphi_{\eta_0}(\eps_n) \lesssim n\an \eps_n^2$ is satisfied by
$
\eps_n = (n\an)^{-\frac{\gamma \wedge \beta}{1+2\gamma}},
$
while in the third case, $\varphi_{\eta_0}(\eps_n) \lesssim n\an \eps_n^2$ for $\eps_n = \left(\frac{n\an}{\log(n\an)} \right)^{-\frac{\gamma \wedge \beta}{1+2\gamma}}$.
 \end{proof}

\subsection{Bernstein--von Mises Results}

\begin{proof}{Proof of Theorem \ref{thm:general_bvm_ap}} In this proof, to avoid any possible confusion, we use the explicit notation $o_{P_0}(1)$ for a term going to $0$ in $P_0$--probability (instead of the shorthand $o_P(1)$). To show that $\sqrt{n}(\psi(\eta)- \hat{\psi})$ converges in distribution (in $P_0$--probability) to a  $\cN(0,V_0)$ law, it suffices to do so for $\sqrt{n}(\psi(\eta)- \hat{\psi})1_{A_n}(\eta)$. Indeed, $\sqrt{n}(\psi(\eta)- \hat{\psi})=\sqrt{n}(\psi(\eta)- \hat{\psi})1_{A_n}(\eta)+\sqrt{n}(\psi(\eta)- \hat{\psi})1_{A_n^c}(\eta)$, and since by assumption $\Pi_{\al_n}[A_n^c\given Y^n]=o_{P_0}(1)$, for $\eta\sim \Pi_{\al_n}[\cdot\given Y^n]$ the variable $1_{A_n^c}(\eta)$  goes to $0$ in probability, and so does $\sqrt{n}(\psi(\eta)- \hat{\psi})1_{A_n^c}(\eta)$ (the probability that it is non--zero is $\Pi_{\al_n}[A_n^c\given Y^n]$).  

Since convergence in distribution is implied by convergence of Laplace transforms (this is also true for convergence in distribution in $P_0$--probability, see Lemma 1 of the supplement of \citealp{CR2015} for details on this), it is enough to show, for any real $t$, that $E_{\al_n}[e^{\sqrt{n}(\psi(\eta)- \hat{\psi})1_{A_n}}\given Y^n]$ goes to $e^{t^2V_0/2}$ in $P_0$--probability. 
Since $e^{\sqrt{n}(\psi(\eta)- \hat{\psi})1_{A_n}}=e^{\sqrt{n}(\psi(\eta)- \hat{\psi})}1_{A_n}+1_{A_n^c}$, using again that $\Pi_{\al_n}[A_n^c\given Y^n]=o_{P_0}(1)$, it is enough to show that 
%Following the approach of Castillo and Rousseau \cite{CR2015}, consider the quantity 
\begin{align*}
 E_{{\alpha_n}}(e^{ t\sqrt{n\an} (\psi(\eta) - \hat{\psi})} |& Y^n, A_n)  := \frac{\int_{A_n} e^{ t\sqrt{n\an} (\psi(\eta) - \hat{\psi})} e^{{\alpha_n}\ell_n(\eta) - \alpha_n \ell_n(\eta_t)} e^{\alpha_n \ell_n(\eta_t)} d\Pi(\eta)}{\int_{A_n} e^{{\alpha_n}\ell_n(\eta)}d\Pi(\eta)} \\
  = &\frac{\int_{A_n} e^{ t\sqrt{n\an} (\psi(\eta) - \hat{\psi})} e^{{\alpha_n}\ell_n(\eta) - \alpha_n \ell_n(\eta_t)} e^{\alpha_n \ell_n(\eta_t)} d\Pi(\eta)}{\int  e^{{\alpha_n}\ell_n(\eta)}d\Pi(\eta)} \Pi_{\alpha_n}(A_n\given Y^n)^{-1} 
\end{align*} 
goes to $e^{t^2V_0/2}$ in $P_0$--probability, where $\eta_t = \eta - t\psi_0/\sqrt{n\an}$ the path as in \eqref{def:eta_perturbed}. 

%{\color{red}Since $\Pi_{\al_n}(A_n|Y^n)= 1+ o_P(1)$, we also have $E_{\alpha_n}(e^{ t\sqrt{n\an} (\psi(\eta) - \hat{\psi})} | Y^n) = e^{t^2V_0/2} (1+o_P(1))$ so that,
%$$\sqrt{n{\alpha_n}}(\psi(\eta) - \hat{\psi})\xrightarrow{\mathcal{L}} \mathcal{N}(0, V_0),$$
%in $P_0-$probability due to convergence of Laplace transforms (Lemma 1 of the supplement of \cite{CR2015}). [K: I think we are missing a step here. We probably need to use the triangle inequality for the BL metric to move to the posterior with prior restricted to $A_n$]}

% {\color{red}[What does conditioning on $A_n$ mean? Should the denominator have an $A_n$?]}
 Using the LAN expansion in Assumption \ref{ass:expansion_assumption} and the linearity of $W_n$,
\begin{align*}
\ell_n (\eta) - \ell_n(\eta_t) &= -\frac{n}{2}\|\eta - \eta_0\|_L^2 + \frac{n}{2} \|\eta_t-\eta_0\|_L^2 + \sqrt{n}W_n(\eta - \eta_t) + R_n(\eta,\eta_0) - R_n(\eta_t,\eta_0) \\
&= - \frac{t\sqrt{n}}{\sqrt{\alpha_n}} \langle \psi_0 , \eta - \eta_0 \rangle_L + \frac{t^2}{2\alpha_n} \|\psi_0\|_L^2 + \frac{t}{\sqrt{\alpha_n}} W_n(\psi_0) + R_n(\eta,\eta_0) - R_n(\eta_t,\eta_0),
\end{align*}
recalling that $\|\cdot\|_L$ is a norm induced by a Hilbert space. Using the definition \eqref{def:psi_hat} of $\hat{\psi}$ and the functional expansion in Assumption \ref{ass:expansion_assumption},
\begin{align*}
t\sqrt{n\an} (\psi(\eta) - \hat{\psi}) = t \sqrt{n\alpha_n} \langle \psi_0 , \eta - \eta_0 \rangle_L - t\sqrt{\alpha_n} W_n(\psi_0) + t \sqrt{n\alpha_n} r(\eta,\eta_0).
\end{align*}
Combining the last two displays thus gives
\begin{align*}
	t\sqrt{n\an} (\psi(\eta) - \hat{\psi}) &+\alpha_n\ell_n(\eta) - \alpha_n \ell_n(\eta_t) \\
	&=  \alpha_n\ell_n(\eta_t) + \frac{t^2 \|\psi_0\|_L^2 }{2} + \underbrace{t\sqrt{n\an}r(\eta, \eta_0) + \alpha_n(R_n(\eta ,\eta_0) - R_n(\eta_t ,\eta_0))}_{\textrm{Rem}(\eta,\eta_0)},
\end{align*}
%Expanding the functional $\psi$ and the log-likelihood as in Assumption \ref{ass:expansion_assumption},
%
%\begin{align*}
%	&t\sqrt{n\an} (\psi(\eta) - \hat{\psi}) +\alpha_n\ell_n(\eta) \\
%	& \quad  = t\sqrt{n\an} (\psi(\eta) - \psi(\eta_0)) +\alpha_n\ell_n(\eta) - t\sqrt{\an}W_n(\psi_0)\\
%	&= t\sqrt{n\an} (\langle \psi_0, \eta - \eta_0 \rangle_L + r(\eta, \eta_0)) +\alpha_n(\ell_n(\eta_0) -\frac{n}{2}\|\eta - \eta_0\|_L^2  + \sqrt{n}W_n(\eta - \eta_0) +R_n(\eta ,\eta_0)) - t\sqrt{\an}W_n(\psi_0)\\
%	% &=  \alpha_n\big(\ell_n(\eta_0) -\frac{n}{2}(\|\eta - \eta_0\|_L^2 - 2 \langle \frac{t\psi_0}{\sqrt{n\alpha_n}}, \eta - \eta_0 \rangle_L ) + \sqrt{n}W_n(\eta - \eta_0) \big) + t\sqrt{n\an}r(\eta, \eta_0) + \alpha_nR_n(\eta ,\eta_0)\\
%	&=  \alpha_n\Big(\ell_n(\eta_0) - \frac{n}{2}(\|\eta - \frac{t\psi_0}{\sqrt{n\alpha_n}} - \eta_0\|_L^2 - \frac{t^2 \|\psi_0\|_L^2 }{n\alpha_n}) + \sqrt{n}\big(W_n(\eta - \frac{t\psi_0}{\sqrt{n\alpha_n}} - \eta_0) + \frac{tW_n(\psi_0)}{\sqrt{n\alpha_n}}\big) \Big) - t\sqrt{\an}W_n(\psi_0) \\
%	&+ t\sqrt{n\an}r(\eta, \eta_0) + \alpha_nR_n(\eta ,\eta_0)\\
%	&=  \alpha_n\ell_n(\eta_t) + \frac{t^2 \|\psi_0\|_L^2 }{2} + \underbrace{t\sqrt{n\an}r(\eta, \eta_0) + \alpha_n(R_n(\eta ,\eta_0) - R_n(\eta_t ,\eta_0))}_{\textrm{Rem}(\eta,\eta_0)}
%\end{align*}
%
where $\sup_{\eta \in A_n}|\textrm{Rem}(\eta, \eta_0)| = o_{P_0}(1)$ by assumption.
Substituting this into the first display of the proof gives
$$E_{{\alpha_n}}(e^{ t\sqrt{n\an} (\psi(\eta) - \hat{\psi})} | Y^n, A_n) = e^{o_{P_0}(1) + t^2\left|\left| \psi_0 \right| \right|^2_L/2} \cdot \frac{\int_{A_n} e^{{\alpha_n}\ell_n(\eta_t) }d\Pi(\eta)}{\int e^{{\alpha_n}\ell_n(\eta)}d\Pi(\eta)}.$$
Since the last ratio equals $1 + o_{P_0}(1)$ by assumption,
the last display goes to  $e^{t^2 V_0/2}$ in $P_0$--probability, which concludes the proof. 
\end{proof}

%%%%%%%%%%%%%%%%%%%%%%%%%%%%%%%%%%%%%%%%%%%%%%%%%%%%%%%%%%%%%%%%% 

% \begin{proof}{Proof of Infinite Series Prior in Gaussian white noise example}
%         Recall that, for $\sigma_k = k^{-\frac{1}{2} - \delta}$ and $f_0 \in C^\beta$, the \ap concentrates on sets $A_n = \{f: \|f-f_0\|_2 \leq \epsilon_n\}$ at rate $\epsilon_n = (n\an)^{-\frac{\delta \wedge \beta}{1+2\delta}}$. With the linear functional $\psi(f) = \int a(t) f(t) dt$, we have $\psi_0 = a$ and $r(f,f_0) = 0$ for all $f$. For $a \in C^\mu$, one can choose $\zeta_n = (n\an)^{-\frac{\mu}{1+2\delta}}$ in order to verify $\|\psi_n - \psi_0\|_2 \leq \zeta_n$ and $\|\psi_n\|_\mathbb{H} \leq \sqrt{n\an}\zeta_n$, and the final condition is verified if $\gamma \wedge \beta > \frac{1}{2} + (\gamma - \mu)$, in which case the \ap distribution of $\sqrt{n\an}(\psi(f) - \hat{\psi})$ converges weakly in $P_0-$probability to a Gaussian with mean 0 and variance $\|a\|_2^2$.
% \end{proof}

\begin{proof}{Proof of Theorem \ref{thm:bvm_density_ap}}
We proceed by verifying the assumptions of Theorem \ref{thm:general_bvm_ap} for the parameter $\eta = \log f$. We first need to verify Assumption \ref{ass:expansion_assumption}. As in the discussion preceding the statement of Theorem 2.4, we have the LAN and functional expansions given by:
\begin{align*}
\ell_n(\eta) - \ell_n(\eta_0) &= -\frac{n}{2}\|\eta - \eta_0\|_L^2 + \sqrt{n}W_n(\eta - \eta_0) + R_n(\eta, \eta_0)\\
% \end{align*}
% where, for $g \in L^2(f_0)$, 
% $$
% 	\|g\|_L^2 = \int (g - F_0(g))^2f_0, \hspace{5mm} W_n(g) = \frac{1}{\sqrt{n}}\sum_{i=1}^n [g(Y_i) - F_0(g)],
% $$
% and,
% $$
% R_n(\eta, \eta_0) = \sqrt{n} P_{f_0}h + \frac{1}{2}\|h\|_L^2,
% $$
% for $h = \sqrt{n}(\eta - \eta_0)$.
% Now for the functional expansion. By assumption we have,
% \begin{align*}
\psi(f) - \psi(f_0) &=\langle \eta - \eta_0, \effinf \rangle_L + \mathcal{B}(f,f_0) + \tilde{r}(f,f_0),
\end{align*}
where
$
\mathcal{B}(f,f_0) = -\int \left[\eta - \eta_0 - \frac{f-f_0}{f_0}\right]\effinf f_0,
$
so that $r(f,f_0) = \mathcal{B}(f,f_0) + \tilde{r}(f,f_0)$.
With $f_t$ as in the statement of Theorem \ref{thm:bvm_density_ap}
and $\eta_t = \log f_t$,
$$
R_n(\eta, \eta_0) - R_n(\eta_t, \eta_0) = \frac{t\sqrt{n}}{\sqrt{\an}} \langle \eta - \eta_0, \effinf \rangle_L - \frac{t^2}{2\an}\|\effinf\|_L^2 + n\log F(e^{-t\effinf/\sqrt{n\an}}).
$$
Expanding the last term, we have for $f \in A_n \subset \{\|f-f_0\|_1 \leq \epsilon_n \}$,
\begin{align*}
    n\log F (e^{-t\effinf/\sqrt{n\an}}) &=n\log\left(1 - \frac{t}{\sqrt{n\an}}\int f \effinf + \frac{t^2}{2n\an}\int f \effinf^2 + o\left(\int f\left(\frac{t^2 \effinf^2}{n\an} \right) \right) \right) \\
    &= n\log\big(1 - \frac{t}{\sqrt{n\an}}\langle \eta - \eta_0, \effinf \rangle_L - \frac{t}{\sqrt{n\an}} \mathcal{B}(f, f_0) + \\
    & \quad + \frac{t^2}{2n\an}\|\effinf\|_L^2 +\frac{t^2}{2n\an}(F - F_0)(\effinf^2) + O((n\an)^{-3/2}) \big) \\
    &= -t\frac{\sqrt{n}}{\sqrt{\an}} \langle \eta - \eta_0, \effinf \rangle_L - t\frac{\sqrt{n}}{\sqrt{\an}}\mathcal{B}(f, f_0) + \frac{t^2}{2\an}\|\effinf\|_L^2 + o(1),
\end{align*}
since $(F - F_0)(\effinf^2) \leq \|\effinf\|^2_\infty \|f - f_0\|_1 \lesssim \eps_n$ on $A_n$. Hence we have
$$
R_n(\eta, \eta_0) - R_n(\eta_t, \eta_0) = -t\frac{\sqrt{n}}{\sqrt{\an}}\mathcal{B}(f, f_0) + o(1),
$$
and the condition on remainder terms in Assumption \ref{ass:expansion_assumption} reduces to
\begin{align*}
\sup_{f \in A_n}|\sqrt{n\an}r(f,f_0)| = o_P(1),
\end{align*}
which is satisfied by assumption. The result then follows from Theorem \ref{thm:general_bvm_ap}.
\end{proof}

\begin{proof}{Proof of Proposition \ref{thm_random_histo_prior}}
		To prove Proposition \ref{thm_random_histo_prior}, we use Lemma \ref{thm} and Lemma \ref{thm_conv_int} stated below.  Lemma \ref{thm} is proved in Section \ref{sec:ancillary} and the proof is very similar to the one of Theorem \ref{thm:bvm_density_ap}. The main differences with Theorem \ref{thm:bvm_density_ap} are that the change of variables condition is stated in term of the projection of $\tilde{\psi}$ and the posterior concentration is around the projection of $f_0$. For a random histogram prior, these two changes turn out to be useful when one wants to give sufficient conditions for the change of variables condition to be satisfied. Indeed, this is is done in Lemma \ref{thm_conv_int} which is also proved in Section \ref{sec:ancillary}.
		\begin{lemma}\label{thm}
			Recall that $\hat{\psi}_{[K_n]}=\psi(f_0)+\frac{1}{n}\sum_{i=1}^{n} \tilde{\psi}_{[K_n]}(Y_i)$. Suppose $f_0$ is bounded and
   \begin{align}\label{assum_RHP_1}
   \Pi_{\alpha_n}(A_n|Y^n):=\Pi_{\alpha_n}(\{f \in H^1_{K_n}, \: \|f-f_{0,K_n}\|_1 \leq \eps_n\}|Y^n)= 1+o_P(1),
   \end{align}
   for a sequence $\eps_n =o(1)$. Set $f_t=fe^{-\frac{t \tilde{\psi}_{[K_n]}}{\sqrt{n\alpha_n}}} / F(e^{-\frac{t \tilde{\psi}_{[K_n]}}{\sqrt{n\alpha_n}}})$ and suppose 
			\begin{align}\label{assum_RHP_2}
				\frac{\int_{A_n} e^{\alpha_nl_n(f_t)} d\Pi(f)}{\int e^{\alpha_nl_n(f)} d\Pi(f)} =1 +o_P(1).
			\end{align}
	Then the $\alpha_n$-posterior distribution of $\sqrt{n\alpha_n}(\psi(f)-\hat{\psi}_{[K_n]})$ converges weakly to a Gaussian distribution with mean 0 and variance $V=\int f_0 \tilde{\psi}_{f_0}^2$.
		\end{lemma}
	\begin{lemma}\label{thm_conv_int}
		Let $\Pi$ be the random histogram prior \eqref{def::RHP} with $k=K_n$ and weights satisfying \eqref{condition_weights}. Suppose 
  \begin{align}\label{assum_consistency}
  \Pi_{\alpha_n}(\tilde{A}_n|Y^n):=\Pi_{\alpha_n}(\{f \in H^1_{K_n}, \|f-f_{0,[K_n]}\|_1 \leq \tilde{\eps}_n \}|Y^n)= 1+o_P(1),
  \end{align}
  for a sequence $\tilde{\eps}_n = o(1)$. Then there exists $\eps_n=o(1)$ a positive sequence (possibly bigger than $\tilde{\eps}_n$), such that
  \begin{align} \label{contraction}
  \Pi_{\alpha_n}(A_n|Y^n):=\Pi_{\alpha_n}(\{f \in H^1_{K_n}, \|f-f_{0,[K_n]}\|_1 \leq \eps_n \}|Y^n)= 1+o_P(1),
  \end{align}
  and
		\begin{align}\label{change_var}
			\frac{\int_{A_n} e^{\alpha_nl_n(f_{t})} d\Pi(f)}{\int e^{\alpha_nl_n(f)} d\Pi(f)} =1 +o_P(1).
		\end{align} 
	\end{lemma}
	We can combine these two results to prove Proposition \ref{thm_random_histo_prior}.
	Indeed, from assumptions \eqref{condition_weights} and \eqref{assum::prop_RHP_1}, using Lemma \ref{thm_conv_int}, we know that there exists a positive sequence $\tilde{\eps}_n$ decreasing to $0$ satisfying \eqref{contraction} and \eqref{change_var}.
	Then we deduce from Lemma \ref{thm} that the posterior distribution of $\sqrt{n\alpha_n}(\psi(f)-\hat{\psi}_{[K_n]})$
	converges weakly to a Gaussian distribution with mean 0 and
	variance $V=\int f_0 \tilde{\psi}_{f_0}^2$.
	Finally, assumption \eqref{assum::prop_RHP_2} implies that $\sqrt{n\alpha_n}(\psi(f)-\hat{\psi})$
	converges weakly to a Gaussian distribution with mean 0 and
	variance $V=\int f_0 \tilde{\psi}_{f_0}^2$.
	\end{proof}

	\begin{proof}{Proof of Corollary \ref{cor::BvM_RHP}}
		This is a direct application of Proposition \ref{thm_random_histo_prior}. Using the assumptions made on $K_n$, the weights and $f_0$, and Proposition \ref{rate_denti_PRH}, we deduce that there exists $\eps_n\to 0$ satisfying \eqref{assum::prop_RHP_1}. Let us now consider the bias term $\sqrt{n\alpha_n}(\hat{\psi}_{[K_n]} - \hat{\psi})$. Recall that from Lemma \ref{lem:bias}, 
		\begin{align*}
			\sqrt{n\alpha_n}(\hat{\psi} - \hat{\psi}_{[K_n]}) = \sqrt{n\alpha_n}(-F_0(\tilde{\psi}_{[K_n]}) + o_P(1/\sqrt{n})) = -\sqrt{n\alpha_n}F_0(\tilde{\psi}_{[K_n]}) + o_P(1).
		\end{align*}
		Using the definition of $\tilde{\psi}_{[K_n]}$,
\begin{align*}
F_0(\tilde{\psi}_{[K_n]})& = \int_{0}^1 f_0 \tilde{\psi}_{[K_n]} = \int_{0}^1 f_0 (a_{[K_n]} - a) =  \int_{0}^1 (f_{0,[K_n]} -f_0 )(a_{[K_n]} - a),
\end{align*}
so that by the H\"older regularity of $f_0$ and $a$,
\begin{align*}
\sqrt{n\alpha_n}|F_0(\tilde{\psi}_{[K_n]})|
&\leq \sqrt{n\alpha_n} \|a_{[K_n]} - a\|_{\infty}\|f_{0,[K_n]} -f_0\|_{\infty} \leq \sqrt{n\alpha_n} K_n^{-\gamma -\beta} =o(1)
\end{align*}
by assumption \eqref{assum::cor_RHP_1}. Hence the assumptions of Proposition \ref{thm_random_histo_prior} are satisfied, which yields the result. 
	\end{proof}
	
\begin{proof}{Proof of Proposition \ref{prop::counterex}}
Recall from \citet{CR2015} p.2371, that the assumptions made on $a$, $K_n$, and $f_0$ allow us to bound the bias term $ F_0(\tilde{\psi}_{[K_n]})$ as follows
	\begin{align}\label{calc::bias}
		n^{-\frac{(\gamma+1)}{3} } \lesssim K_n^{-(\gamma+1) } \lesssim - F_0(\tilde{\psi}_{[K_n]}) \lesssim  K_n^{-(\gamma+1) } \lesssim n^{-\frac{(\gamma+1)}{3} }.
	\end{align}
We first prove the result regarding the full posterior. Since, $f_0 \in \mathcal{C}^1([0,1])$ and is bounded away from 0, $\frac{n^{1/3} }{2}\leq K_n \leq n^{1/3}$ and hence $K_n=o(n/\log(n))$, $ n^{-b}\leq \delta_{i,n}=n^{-b} \leq 1$, we deduce from Proposition \ref{rate_denti_PRH} that there exists $\eps_n\to 0$ satisfying \eqref{assum::prop_RHP_1}. Combining this latter result with the fact that $\sum_{i=1}^{K_n}\delta_{i,n}= K_n n^{-b} \leq n^{1/3 - b} = o(\sqrt{n})$, we can use Proposition \ref{thm_random_histo_prior} to deduce that the posterior distribution of $\sqrt{n}(\psi(f) - \hat{\psi}_{[K_n]} )$ converges weakly to the $\mathcal{N}(0, V_0)$ distribution in $P_0$-probability.
Moreover, \eqref{calc::bias} implies  $|\sqrt{n}F_0(\tilde{\psi}_{[K_n]})| \geq c > 0$ since $\gamma\leq1/2$ and even $|\sqrt{n}F_0(\tilde{\psi}_{[K_n]})| \rightarrow \infty$ if $\gamma<1/2$ .
    
For the result regarding the $\alpha_n$-posterior, the proof is similar. Since, $f_0 \in \mathcal{C}^1([0,1])$ bounded away from 0, $K_n\leq n^{1/3} = o(n^{1-x} /\log(n^{1-x}))$ since $x < 2/3$ hence $K_n=o(n\alpha_n/\log(n\alpha_n))$, $(n\alpha_n)^{-b'}\leq \delta_{i,n}= n^{-b} \leq 1$ for some $b'>0$, we deduce from Proposition \ref{rate_denti_PRH} that there exists $\eps_n\to 0$ satisfying \eqref{assum::prop_RHP_1}.  Combining this latter result with the fact that $\sum_{i=1}^{K_n}\delta_{i,n}= K_n n^{-b} \leq n^{1/3 - b} = o(\sqrt{n\alpha_n})$ since $b>1/6$, we can use Proposition \ref{thm_random_histo_prior} to deduce that the posterior distribution of $\sqrt{n\alpha_n}(\psi(f) - \hat{\psi}_{[K_n]} )$ converges weakly to the $\mathcal{N}(0, V_0)$ distribution in $P_0$-probability.
Finally, by \eqref{calc::bias}, $\sqrt{n\alpha_n}F_0(\tilde{\psi}_{[K_n]}) = o(1)$ since $x>(1-2\gamma)/3$.
Thus, by Lemma \ref{lem:bias}, it follows that $\sqrt{n\alpha_n}(\hat{\psi} - \hat{\psi}_{[K_n]}) = o_P(1)$. Therefore, we deduce that the $\alpha_n$-posterior distribution of $\sqrt{n\alpha_n}(\psi(f) - \hat{\psi} )$ converges weakly to the $\mathcal{N}(0, V_0)$ distribution in $P_{0}$-probability.	
\end{proof}

\begin{proof}{Proof of Theorem \ref{thm:general_theorem_gp_gwn}}
We will verify the conditions of Theorem \ref{thm:bvm_gwn_ap}, for which we need to construct suitable sets $A_n$ satisfying Assumption \ref{ass:expansion_assumption} and the `change of measure' condition \eqref{GTGWN_3}. Under the theorem hypothesis that $\varphi_{f_0}(\eps_n) \leq n\an \eps_n^2$, Proposition \ref{thm:GGV_verification_GP_gwn} implies that the posterior contracts about $f_0$ at rate $\eps_n$ in $\|\cdot\|_2$, i.e. $\Pi_{\alpha_n}(B_n|Y) \to^{P_0} 1$ for $B_n = \{f:\|f-f_0\|_2 \leq M \eps_n\}$ with $M>0$ large enough.

Turning to condition \eqref{GTGWN_3}, we follow \citet{C2012} and first approximate the perturbation $f_t = f- \frac{t\psi_0}{\sqrt{n\an}}$ by an element of the RKHS and then apply the Cameron-Martin Theorem. To this end, let $\psi_n\in\mathbb{H}$ satisfy \eqref{eq:RKHS_conditions}. Define the following isometry associated to the Gaussian process $W$:
\begin{align*}
	U_W : \textrm{Vect}\langle \{ t \rightarrow K(\cdot,t): t\in\mathbb{R} \} \rangle &\rightarrow L^2(\Omega) \\
	\eta := \sum_{i=1}^p a_iK(\cdot, t_i) &\mapsto \sum_{i=1}^p a_i W_{t_i} =: U_W(\eta),
\end{align*}
and since any $h \in \mathbb{H}$ is the limit of a sequence $\sum_{i=1}^{p_n} a_{i,n}K(\cdot, t_{i,n})$, $U_W$ can be extended to an isometry $U_W: \mathbb{H} \rightarrow L^2(\Omega)$. Then $U_W(h)$ is the $L^2$-limit of the sequence $\sum_{i=1}^{p_n}a_{i,n}W_{t_i}$, so that it is a Gaussian random variable with mean 0 and variance $\|h\|_\mathbb{H}^2$.
Recalling that $f=W$ is a Gaussian process under the prior, the usual Gaussian tail bound implies
\begin{equation}\label{eq:tail_bound}
\Pi(W: |U_W(\psi_n)| \geq M_0 \sqrt{n\alpha_n} \eps_n \|\psi_n\|_{\mathbb{H}}) \leq 2e^{-M_0^2n\alpha_n\eps_n^2/2},
\end{equation}
so that the posterior probability of the set in the last display tends to zero in $P_0$-probability by Lemma \ref{lemma_sieve} for $M_0>0$ large enough. Together with the contraction result, this shows that the sets
$$ A_n = \{ w: |U_w(\psi_n)| \leq M_0 \sqrt{n\an} \eps_n \|\psi_n\|_\mathbb{H}\} \cap B_n$$
satisfy $\Pi(A_n|Y) \to^{P_0} 1$ as $n\to\infty$. Since $A_n \subset B_n$ and using the assumptions of the present theorem, the sets $A_n$ satisfy Assumption \ref{ass:expansion_assumption}.

It thus remains to establish the condition \eqref{GTGWN_3}. Define the approximate perturbation $f_n = f-\frac{t\psi_n}{\sqrt{n\alpha_n}}$, which we will now show satisfies
\begin{equation}\label{eq:likelihood_change}
\sup_{f \in B_n}|\an(\ell_n(f_n) - \ell_n(f_t))| = o_P(1).
\end{equation}
Indeed, using the LAN expansion for the Gaussian white noise model, under $P_0$,
$$\an(\ell_n(f_n) - \ell_n(f_t)) = t\sqrt{n\an}\int_0^1(f-f_0)(\psi_n - \psi_0) - \frac{t^2}{2} (\|\psi_n\|_2^2 - \|\psi_0\|_2^2) + \frac{t}{\sqrt{n\alpha_n}} W_n(\psi_0-\psi_n),$$
where we recall $W_n(g) \sim \cN(0,\|g\|_2^2)$ for any $g \in L^2$. By Cauchy-Schwarz, the first term is bounded by $t\sqrt{n\an}\|f-f_0\|_2 \|\psi_n - \psi_0 \|_2 \leq t\sqrt{n\an}\zeta_n \epsilon_n = o(1)$ by assumption \eqref{eq:RKHS_conditions} for $f\in B_n$. The absolute value of the second term equals
  \begin{align*} 
  \tfrac{t^2}{2}\left|\langle\psi_n - \psi_0, \psi_n + \psi_0 \rangle_2  \right|  \leq \tfrac{t^2}{2}\|\psi_n-\psi_0\|_2\|\psi_n+\psi_0\|_2 \leq \tfrac{t^2}{2} \zeta_n (2\|\psi_0\|_2 + \zeta_n) = o(1),
  \end{align*}  
again by assumption \eqref{eq:RKHS_conditions}. The third term has distribution $N\left(0, t^2 \an\|\psi_n - \psi_0\|_2^2/n \right)$, which is $o_P(1)$ since its variance tends to zero as $n\to\infty$. Together, these three bounds establish \eqref{eq:likelihood_change}.

A version of the Cameron-Martin theorem (\citealp{c12}, Lemma 17) states that for all $\Phi : \mathbb{B} \rightarrow \mathbb{R}$ measurable and for any $g, h \in \mathbb{H}$ and $\rho > 0$,
$$ E(1_{\{ |U_W(g)| \leq \rho \}} \Phi(W-h)) = E(1_{ \{ |U_W(g) + \langle g,h\rangle_\mathbb{H}| \leq \rho \} } \Phi(W)e^{U_W(-h) - \|h\|^2_\mathbb{H}/2}).$$
Using \eqref{eq:likelihood_change} and the last display with $h_t = t\psi_n/\sqrt{n\alpha_n}$ and
$\rho_t = M_0 \sqrt{n\alpha_n}\eps_n \|\psi_n\|_{\mathbb{H}}$, the quantity in \eqref{GTGWN_3} equals
$$\frac{\int_{A_n} e^{\alpha_n \ell_n(f_t)} d\Pi(f)}{\int e^{\alpha_n \ell_n(f)} d\Pi(f)}
= \frac{\int_{B_{n,t}} 1_{\{ |U_w(\psi_n) + \langle\psi_n,h_t \rangle_{\mathbb{H}}| \leq \rho_t\}} e^{\alpha_n \ell_n(w)} e^{U_w(-h_t) - \|h_t\|_{\mathbb{H}}^2/2} d\Pi(w)}{\int e^{\alpha_n \ell_n(f)} d\Pi(f)} e^{o_P(1)},
$$
where $B_{n,t} = B_n - h_t =  \{w: \|w+ t\psi_n/\sqrt{n\alpha_n} - f_0\|_2 \leq M \eps_n\}$. For $w$ in the domain of the top integral, using also \eqref{eq:RKHS_conditions},
\begin{align*}
|U_w(-h_t) - \|h_t\|_{\mathbb{H}}^2/2|  &= \frac{t}{\sqrt{n\alpha_n}} \left| U_w(\psi_n) + \tfrac{1}{2} \langle \psi_n ,h_t \rangle_\mathbb{H} \right| 
\\& \leq \frac{t}{\sqrt{n\alpha_n}} \rho_t + \frac{t^2}{2n\alpha_n}\|\psi_n\|_\mathbb{H}^2  \leq tM_0 \sqrt{n\alpha_n} \eps_n \zeta_n + \frac{t^2}{2} \zeta_n^2 \to 0.
\end{align*}
Setting
$A_{n,t} = \{ w: |U_w(\psi_n) + \langle \psi_n,h_t \rangle_\mathbb{H}| \leq \rho_t \} \cap B_{n,t},$
the ratio of integrals thus equals
$$\frac{\int_{A_{n,t}} e^{\alpha_n \ell_n(w)}  d\Pi(w)}{\int e^{\alpha_n \ell_n(f)} d\Pi(f)} e^{o_P(1)} = \Pi_{\alpha_n} (A_{n,t}|Y) e^{o_P(1)}.$$
It thus remains to show $\Pi_{\alpha_n}(A_{n,t}|Y) = 1+o_P(1)$. Since
$$A_{n,t}^c = \{ w:  |U_w(\psi_n) + \langle \psi_n,h_t \rangle_\mathbb{H}| > M_0 \sqrt{n\alpha_n}\eps_n \|\psi_n\|_{\mathbb{H}} \} \cup \{w: \|w+t\psi_n/\sqrt{n\alpha_n} - f_0\|_2 >M \eps_n \},$$
it suffices to consider the posterior probability of each of the last sets. Since $\|w+t\psi_n/\sqrt{n\alpha_n} - w\|_2 \lesssim \|\psi_n\|_2 /\sqrt{n\alpha_n} \lesssim (1 + \zeta_n)/\sqrt{n\alpha_n}$, the second set is contained in $\{w:\|w-f_0\|_2 > M\eps_n - C/\sqrt{n\alpha_n}\}$, which has posterior probability $o_P(1) $ by Proposition \ref{thm:GGV_verification_GP_gwn}, possibly after replacing $\eps_n$ by a multiple of itself.
For the first set, note that
$|\langle \psi_n,h_t \rangle_\mathbb{H}| = t \|\psi_n\|_{\mathbb{H}}^2 /\sqrt{n\alpha_n}$ is of strictly smaller order than$\sqrt{n\alpha_n} \eps_n \|\psi_n\|_\mathbb{H}$
if and only if $\|\psi_n\|_{\mathbb{H}} = o(n\alpha_n\eps_n)$. By \eqref{eq:RKHS_conditions}, it suffices that $\zeta_n = o(\sqrt{n\alpha_n}\eps_n)$, which holds since $\zeta_n \to 0$ while $n\alpha_n\eps_n^2 \to \infty$. Thus the first set is contained in $\{ w:|U_w(\psi_n)| >(M_0/2) \sqrt{n\alpha_n}\eps_n \|\psi_n\|_\mathbb{H} \}$ for $n$ large enough. Arguing as in \eqref{eq:tail_bound} and using Lemma \ref{lemma_sieve}, the posterior probability of this set is thus $o_P(1)$. This shows that $\Pi(A_{n,t}^c|Y) = o_P(1)$ as required.
\end{proof}

\begin{proof}{Proof of Theorem \ref{thm:general_theorem_gp_density_estimation}}
The proof is similar to the proof of Theorem \ref{thm:general_theorem_gp_gwn}, but with a few minor differences. For the LAN expansion, we have under $P_0$,
$$
\an(\ell_n(\eta_n) - \ell_n(\eta_t)) = t\mathbb{G}_n(\effinf - \psi_n) + t\sqrt{n} \int (f_0 - f_\eta)(\effinf  - \psi_n) + o(1),
$$
where $\mathbb{G}_n(g) = \frac{1}{\sqrt{n}} \sum_{i=1}^n (g(Y_i) - E_0(g(Y_i)))$, so that $\mathbb{G}(\effinf - \psi_n) = o_P(1)$. We consider sets $B_n$ defined in terms of the $\|\cdot\|_1$-norm rather than $\|\cdot\|_2$-norm, so that we may use our contraction results for density estimation.  In the last display, we thus use H\"older's inequality $t\sqrt{n}\|f_0 - f_\eta\|_1 \|\effinf - \psi_n\|_\infty$ instead of Cauchy-Schwarz, which requires the slightly stronger assumption involving the $L^\infty$-norm $\|\effinf - \psi_n\|_\infty$ to show that this tends to 0.
\end{proof}

\begin{proof}{Proof of Corollary \ref{cor:gp_examples}}
%For all of the cases where $a \in \mathbb{H}$, the sequence $\psi_n$ can be chosen to be identically $\effinf = a-\int af_0$. The form of the concentration function is well known for each of these processes and a choice of $\epsilon_n$ always exists to satisfy the inequality.	
We apply Theorem \ref{thm:general_theorem_gp_gwn} in Gaussian white noise and Theorem \ref{thm:general_theorem_gp_density_estimation} in density estimation. In both cases, the required functional expansion holds by the linearity of $\psi(f)$, so that it remains to verify \eqref{eq:conc_eqn} and that one can suitably approximate the representers $\psi_0=a$ or $\tilde{\psi}_{f_0} = a - \int_0^1 f_0 a$ by elements of the RKHS.

By Corollary \ref{corrollary:gp_contraction_rate_examples}, in each case $\eps_n = (n\an)^{-\frac{\gamma \wedge \beta}{2\gamma + 1}}$ satisfies the condition \eqref{eq:conc_eqn} on the concentration function, possibly up to a $\log (n\an)$-factor that does not affect our results here.
		Next, one can show that in Examples \ref{ex:Matern}-\ref{ex:RL} (see the proof of Theorem 4 in \citealp{C2008} for the Riemann-Liouville process, the proof of Lemma 11.37 in \citealp{vdv_FNBI} for the Matérn process, and the proof of Lemma 2.2 in \citealp{vdv_2007} for the rescaled square exponential process), for an appropriate kernel smoother $\phi$ and sequence $\sigma_n$,  
		$$
		\psi_n(x) = \left[\frac{1}{\sigma_n}\phi\left(\frac{\cdot}{\sigma_n}\right)*a(\cdot)\right](x)
		$$
		 satisfies $\psi_n \in \mathbb{H}$, 
		 $\|\psi_n - a\|_\infty \leq \sigma_n^\mu,$
		  and
		  $
		  \|\psi_n\|_\mathbb{H}^2 \lesssim \sigma_n^{-2\gamma -1 + 2\mu}.
		  $
		Setting $\sigma_n = \zeta_n^{1/\mu}$, we obtain $\zeta_n = (n\an)^{-\frac{\mu}{2\gamma + 1 }}$ as a suitable choice to satisfy the bounds on $\psi_n$. It thus remains to show
$\sqrt{n\an} \epsilon_n \zeta_n \rightarrow 0,$
for which a sufficient condition is $\gamma \wedge \beta > \frac{1}{2} + (\gamma - \mu)$. The final part of the condition comes from the fact that we need $\epsilon_n \rightarrow 0,$ which is satisfied if $\gamma \wedge \beta > 0$.

For the infinite series prior in Gaussian white noise, one instead uses the truncated series $\psi_n =\sum_{k=1}^{J_n} \langle a,\phi_k\rangle_2 \phi_k \in \mathbb{H}$, for which $\|\psi_n - a\|_2^2 = \sum_{k>J_n} k^{2\mu - 2\mu} |\langle a,\phi_k\rangle_2|^2 \lesssim J_n^{-2\mu} \|a\|_{\mathcal{H}^\mu}^2$ and $\|\psi_n\|_{\mathbb{H}}^2 \lesssim J_n^{2\gamma+1-2\mu}$. Taking $J_n \asymp (n\alpha_n)^{1/(2\gamma+1)}$ and $\zeta_n \asymp (n\alpha_n)^{-\mu/(2\gamma)}$ as above, we recover the same conditions as in Examples \ref{ex:Matern}-\ref{ex:RL}.
\end{proof}

\begin{proof}{Proof of Lemma \ref{prop:contraction_rate_an_posterior_Lmu}}
By conjugacy of the $\alpha_n$-posterior of $f$, the \ap distribution of $\psi(f)$ is
$$N\left(
\sum_{k=1}^\infty\frac{n\alpha_n \lambda_k}{1+n\alpha_n\lambda_k}\psi_kY_k,
\sum_{k=1}^\infty \frac{\lambda_k}{1+n\alpha_n\lambda_k} \psi_k^2 \right).$$
As in \citet{BIP2011}, it thus suffices to show that
	$$
	\left|\sum_{k=1}^\infty \left( \frac{n\an \lambda_k}{1+n\an \lambda_k} \psi_kf_{0,k} - \psi_kf_{0,k} \right)\right|^2 + \frac{1}{n}\sum_{k=1}^\infty \left(\frac{n\an \lambda_k}{1+n\an \lambda_k}\psi_k \right)^2 + \sum_{k=1}^\infty \frac{\lambda_k}{1+n\an\lambda_k} \psi_k^2
	$$
	is bounded by a multiple of $\eps_n^2$. We have,
\begin{align*}
	\left|\sum_{k=1}^\infty \left( \frac{n\an \lambda_k}{1+n\an \lambda_k} \psi_kf_{0,k} - \psi_kf_{0,k} \right)\right|^2 &=	\left|\sum_{k=1}^\infty \frac{\psi_k f_{0,k}}{1+n\an \lambda_k}\right|^2 \leq   \|f_0\|_\beta^2 \sum_{k=1}^\infty \frac{\psi_k^2k^{-2\beta}}{(1+n\alpha_nk^{-1-2\gamma})^2} \\
	&\lesssim \|f_0\|_\beta^2 \cdot\|(l_i)\|_\mu^2 \cdot(n\alpha_n)^{-\left(\frac{2\beta + 2\mu}{1+2\gamma} \wedge 2\right)} \lesssim\eps_n^2,
\end{align*}
where we have used Lemma 8.2 in \citet{BIP2011} to deduce the second last inequality. For the second term, 
\begin{align*}
	\frac{1}{n}\sum_{k=1}^\infty \left(\frac{n\an \lambda_k}{1+n\an \lambda_k}\psi_k \right)^2 &= \sum_{i=1}^\infty \frac{\psi_k^2n\alpha_n^2\lambda_i^2}{(1+n\alpha_n\lambda_k)^2} =n\alpha_n^2 \sum_{k=1}^\infty \frac{\psi_k^2 k^{-2-4\gamma}}{(1+n\an k^{-1-2\gamma})^2} \\ 
	&\lesssim  \|l\|_\mu^2 \cdot n\an^2 (n\an)^{-\left(\frac{2 + 4\gamma +2\mu}{1+2\gamma}\wedge 2\right)}
	\lesssim \alpha_n (n\an )^{-\left(\frac{1 + 2\gamma +2\mu}{1+2\gamma}\wedge 1\right)} \lesssim \eps_n^2,
\end{align*}
where we used Lemma 8.1 of \citet{BIP2011} for the first inequality. Finally, 
\begin{align*}
	\sum_{k=1}^\infty \frac{\lambda_k}{1+n\an\lambda_k} \psi_k^2
	&= \sum_{k=1}^\infty \frac{\psi_k^2 k^{-1-2\gamma}}{1 + n\an k^{-1-2\gamma}} \lesssim \|l\|_\mu^2 (n\alpha_n)^{-\left(\frac{1+2\gamma + 2\mu}{1+2\gamma} \wedge 1\right)} \lesssim \eps_n^2
\end{align*}
where we again invoke Lemma 8.1 of \citet{BIP2011} for the first inequality.
\end{proof}

\subsection{Credible Regions}
\begin{proof}{Proof of Theorem \ref{prop_corrected_region}}
	Using \eqref{conv_BvM_alpha} and standard results recalled in Lemma \ref{lemma_conv_quantiles}, it follows that for all $\delta \in (0,1)$, $\sqrt{n\alpha_n}(a_{n,\delta}^Y-\hat{\psi}) = \sqrt{V}q_\delta + o_P(1)$ and hence we have the expansion of the quantile
	\begin{align}\label{dvpmt_quantiles_alpha}
%		\sqrt{n\alpha_n}(a_{n,\delta}^Y-\hat{\psi}) = \sqrt{V}q_\delta + o_P(1) \iff
		 a_{n,\delta}^Y = \hat{\psi} + \frac{\sqrt{V}q_\delta}{\sqrt{n\alpha_n}} + o_P(\frac{1}{\sqrt{n\alpha_n}}).
	\end{align} 
	Combining this expansion and assumption \eqref{assum::estimator_bar_psi}, we obtain $\sqrt{\alpha_n}(a_{n, \delta}^Y - \bar{\psi}) =  \sqrt{V}q_{\delta}/\sqrt{n} + o_P(1/\sqrt{n}).$
%	\begin{align*}
%		\sqrt{\alpha_n}(a_{n, \delta}^Y - \bar{\psi})&=\sqrt{\alpha_n}(\hat{\psi} + \frac{\sqrt{V}q_{\delta}}{\sqrt{n\alpha_n}} + o_P(\frac{1}{\sqrt{n\alpha_n}}) -( \hat{\psi}  + o_P(\frac{1}{\sqrt{n}}) ) )\\
%		&=\sqrt{\alpha_n}(\frac{\sqrt{V}q_{\delta}}{\sqrt{n\alpha_n}} + o_P(\frac{1}{\sqrt{n\alpha_n}}) )
%		= \frac{\sqrt{V}q_{\delta}}{\sqrt{n}} + o_P(\frac{1}{\sqrt{n}}).
%	\end{align*}
	%
	Hence we can expand the shift-and-rescale set as
	\begin{align*}
		\mathcal{J}_{\alpha_n} & = \left(\sqrt{\alpha_n}(a_{n, \frac{\delta}{2}}^Y - \bar{\psi}) +\bar{\psi} , \sqrt{\alpha_n}(a_{n, 1-\frac{\delta}{2}}^Y - \bar{\psi}) +\bar{\psi} \right] \\
%		&=(\frac{\sqrt{V}q_{\frac{\delta}{2}}}{\sqrt{n}} + o_P(\frac{1}{\sqrt{n}}) +\hat{\psi} + o_P(\frac{1}{\sqrt{n}})  ,  \frac{\sqrt{V}q_{1-\frac{\delta}{2}}}{\sqrt{n}} + o_P(\frac{1}{\sqrt{n}})  +\hat{\psi} + o_P(\frac{1}{\sqrt{n}}) ]\nonumber\\
		&= \Big(\hat{\psi} + \sqrt{V}q_{\frac{\delta}{2}}/\sqrt{n} + o_P(1/\sqrt{n})  , \hat{\psi} + \sqrt{V}q_{1-\frac{\delta}{2}} /\sqrt{n} + o_P(\frac{1}{\sqrt{n}}) \Big]. 
	\end{align*}
	 This last expansion together with assumption \eqref{conv_hat_psi} yield the first conclusion of Theorem \ref{prop_corrected_region}.
  Let us move to the case $\alpha_n=\alpha\in (0,1]$ is fixed. By \eqref{dvpmt_quantiles_alpha}, the posterior median equals
	 \begin{align*}
	 	a_{n, \frac{1}{2}}^Y = \hat{\psi} + \frac{\sqrt{V}q_\frac{1}{2}}{\sqrt{n\alpha_n}} + o_P(1/\sqrt{n\alpha_n}) = \hat{\psi} + o_P(\frac{1}{\sqrt{n\alpha_n}}) = \hat{\psi} + o_P(\frac{1}{\sqrt{n}})
	 \end{align*}
	 since $q_\frac{1}{2}=0$.
  Hence \eqref{assum::estimator_bar_psi} is satisfied and the result follows. 
\end{proof}

\begin{proof}{Proof of Proposition \ref{prop:corrected_credible_sets_behaviour}}
For $\bar{\psi}$ the posterior mean/median, define $T_n = \sqrt{n}(\bar{\psi} - \hat{\psi})$. We need to show that $|T_n| = o_P(1)$ to satisfy the assumption \eqref{assum::estimator_bar_psi} of Theorem \ref{prop_corrected_region}. We have
\begin{align*}
	T_n = -\sqrt{n}\sum_{k=1}^\infty \frac{1}{1+n\an \lambda_k}a_k f_{0,k} - \sum_{k=1}^\infty \frac{1}{1+n\an \lambda_k}a_k \epsilon_k =: -t_{n,1}  -t_{n,2}.
\end{align*}
The second term is Gaussian with mean 0 and variance $\sum_{k=1}^\infty \frac{1}{(1+n\an \lambda_k)^2}a_k^2 \asymp (n\an)^{-(\frac{2\mu}{1+2\gamma}\wedge 2)}$ by Lemma 8.1 of \citet{BIP2011}. Thus, $|t_{n,2}| = o_P(1)$ since $\mu, \gamma > 0$. Turning to $t_{n,1}$, set ${k^*} = (n\an)^{\frac{1}{1+2\gamma}}$. For $k \leq k^*$, we have $(n\an)k^{-1-2\gamma} < 1+(n\an)k^{-1-2\gamma} \leq 2(n\an)k^{-1-2\gamma}$, and for $k > k^*$, we have $1 < 1+(n\an)k^{-1-2\gamma} < 2$.
Hence, we can write
\begin{equation}\label{eq:tn1}
\begin{split}
	|t_{n,1}| = \sqrt{n}\sum_{k=1}^{\infty} \frac{k^{-1-(\beta + \mu)}}{1+n\an k^{-1-2\gamma}} &\asymp \frac{\sqrt{n}}{n\an}\sum_{k=1}^{k^*} \frac{k^{-1-(\beta + \mu)}}{k^{-1-2\gamma}} + \sqrt{n}\sum_{k={k^*} +1}^\infty k^{-1-(\beta + \mu)}	\\
& \asymp \frac{1}{\sqrt{n} \alpha_n} \sum_{k=1}^{k^*} k^{-(\beta + \mu-2\gamma)} + \sqrt{n} (n\alpha_n)^{-\frac{\beta+\mu}{1+2\gamma}}.
%&= \frac{\sqrt{n}}{n\an}\sum_{k=1}^{k^*} k^{-(\beta + \mu - 2\delta)} + \sqrt{n}\sum_{k={k^*} +1}^\infty k^{-1-(\beta + \mu)}	
\end{split}
\end{equation}
It thus suffices to study when this quantity is $o(1)$. The second term in the last display is $o(1)$ if and only if $\alpha_n \gg n^{\frac{1+2\gamma}{2\beta+2\mu}-1}=: \omega_n$, while the first term has three cases.
%\begin{align*}
%	|t_{n,1}| = \sqrt{n}\sum_{k=1}^{k^*} \frac{k^{-1-(\beta + \mu)}}{1+n\an k^{-1-2\gamma}} + \sqrt{n}\sum_{k={k^*} +1}^\infty\frac{k^{-1-(\beta + \mu)}}{1+n\an k^{-1-2\gamma}}.
%\end{align*}
%For $k \leq k^*$, we have $(n\an)k^{-1-2\gamma} < 1+(n\an)k^{-1-2\gamma} \leq 2(n\an)k^{-1-2\gamma}$, and for $k > k^*$ we have, $1 < 1+(n\an)k^{-1-2\gamma} < 2$. Hence, we can write
%\begin{align*}
%|t_{n,1}| &\asymp \frac{\sqrt{n}}{n\an}\sum_{k=1}^{k^*} \frac{k^{-1-(\beta + \mu)}}{k^{-1-2\delta}} + \sqrt{n}\sum_{k={k^*} +1}^\infty k^{-1-(\beta + \mu)}	 \\
%&= \frac{\sqrt{n}}{n\an}\sum_{k=1}^{k^*} k^{-(\beta + \mu - 2\delta)} + \sqrt{n}\sum_{k={k^*} +1}^\infty k^{-1-(\beta + \mu)}	
%\end{align*}
%For $\beta + \mu > 0$, we have, 
%$$
%\sqrt{n}\sum_{k={k^*} +1}^\infty k^{-1-(\beta + \mu)} = \sqrt{n}C {k^*}^{-(\beta + \mu)}(1+o(1)) = \sqrt{n}C(n\an)^{-\frac{\beta+\mu}{1+2\delta}}(1+o(1)),
%$$
%for a constant $C > 0$.
% For the first term we have three cases to consider.

(1) $\beta + \mu > 1+2\gamma$: the first sum in \eqref{eq:tn1} is summable even for $k^*=\infty$, and hence $|t_{n,1}| \asymp \frac{1}{\sqrt{n}\alpha_n} + + \sqrt{n} (n\alpha_n)^{-\frac{\beta+\mu}{1+2\gamma}}$, which is $o(1)$ if and only if $\alpha_n \gg \max(n^{-1/2},\omega_n) = n^{-1/2}$, i.e. $\sqrt{n}\alpha_n \to \infty$.

(2) $\beta + \mu = 1+2\gamma$: note that $\omega_n = n^{-1/2}$, while the first term in \eqref{eq:tn1} equals $\frac{1}{\sqrt{n}\alpha_n} \sum_{k=1}^{k^*} k^{-1} \asymp \frac{1}{\sqrt{n}\alpha_n} \log k^* \asymp \frac{1}{\sqrt{n}\alpha_n} \log (n\alpha_n)$. This is $o(1)$ if and only if $\alpha_n \gg n^{-1/2} \log n$, i.e. $\frac{\sqrt{n}\an}{\log n} \rightarrow \infty$.

(3) $\beta + \mu < 1 + 2\gamma$: the first term in \eqref{eq:tn1} is of size $\frac{1}{\sqrt{n}\alpha_n} (k^*)^{1+2\gamma-\beta-\mu} \asymp \sqrt{n} (n\alpha_n)^{-\frac{\beta+\mu}{1+2\gamma}}$, which is exactly the same order as the second term in \eqref{eq:tn1}. Thus $|t_{n,1}| = o(1)$ if and only if $\alpha_n \gg\omega_n$. But our results are restricted to the regime $0<\alpha_n\leq 1$ and hence we require that $\omega_n \to 0$ to have a valid choice satisfying $0<\omega_n \ll \alpha_n\leq 1$. One can then check that $\omega_n \to 0$ if and only if $\frac{1}{2} + \gamma < \beta + \mu$, which determines the lower bound in this range. For such a choice, $|t_{n,1}| = o(1)$ if and only if $\omega_n^{-1} \alpha_n = n^{1 - \frac{1+2\delta}{2(\beta + \mu)}}\an \rightarrow \infty$.
\end{proof}

\subsection{Supremum norm contraction rates}
\begin{proof}{Proof of Proposition \ref{prop:rate_sup_norm_P1}}% 
We focus on the case (ii) for brevity, the case (i) being similar (though easier, see also \citealp{cas_14}). Set $L_n:=\lfloor\frac{n\alpha_n}{\log(n\alpha_n)\log(2)(2\beta+1}\rfloor$. For all sequences $(f_{lk})$, denote $f^{L_n}: = \sum_{l=0}^{L_n}\sum_{k} f_{lk}\psi_{lk}$ and $f^{L_n^{C}} := \sum_{l>L_n}\sum_{k} f_{lk}\psi_{lk}$. Also denote $\hat{f}^{L_n}:=\sum_{l=0}^{L_n}\sum_{k} Y_{lk}\psi_{lk}$. We have 
		\begin{align}\label{majo_gen}
		E_0(\int \ninf{f-f_0} d\Pi_{\alpha_n}(f|Y^n)) \leq \underbrace{E_0(\int\ninf{f^{L_n}-\hat{f}^{L_n}}d\Pi_{\alpha_n}(f|Y^n))}_{(a)}
		+\underbrace{ E_0(\ninf{\hat{f}^{L_n}-f_0^{L_n}})}_{(b)}\nonumber\\
		+ \underbrace{E_0(\int \ninf{f^{L_n^c}}d\Pi_{\alpha_n}(f|Y^n))}_{(c)} + \underbrace{\ninf{f_0^{L_n^c}}}_{(d)}.
		\end{align}
\textbf{Term $(d)$} Using the assumptions made on the coefficients $(f_{0,lk})$ and the localisation property of the wavelet basis $(\psi_{lk})$ that $\sum_{k}|\psi_{lk}(x)| \lesssim 2^{l/2} $ for all $x\in [0,1]$,
        \begin{align}\label{majo_iv}
			\ninf{f_0^{L_n^c}} & \leq \sum_{l>L^n} \max_{k}\left|f_{0,lk}\right|\ninf{\sum_{k}|\psi_{lk}|} \lesssim  \sum_{l>L^n} 2^{-l(\frac{1}{2}+\beta )} 2^{l/2} \lesssim 2^{-\beta L_n}.
		\end{align}
\textbf{Term $(b)$} Using the localisation property of the basis $(\psi_{lk})_{lk}$, it follows
		\begin{align*}
		 \ninf{\hat{f}^{L_n}-f_0^{L_n}} = \ninf{\sum_{l=0}^{L_n} \sum_{k}\frac{\eps_{lk}}{\sqrt{n}}\psi_{lk} } \leq \sum_{l=0}^{L_n} \max_{k}\left|\frac{\eps_{lk}}{\sqrt{n}}\right| \ninf{\sum_{k}|\psi_{lk}|}\lesssim \frac{1}{\sqrt{n}} \sum_{l=0}^{L_n} \max_{k}\left|\eps_{lk}\right| 2^{l/2}.
         \end{align*}
Then a standard result about the maximum of $n$ gaussian variables gives that
        \begin{align}\label{majo_ii}
	   E_0(\ninf{\hat{f}^{L_n}-f_0^{L_n}}) & \lesssim \frac{1}{\sqrt{n}} \sum_{l=0}^{L_n} E_0(\max_{2^l- 1 \geq k \geq 0}\left|\eps_{lk}\right|) 2^{l/2} \lesssim \frac{1}{\sqrt{n}} \sum_{l=0}^{L_n} \sqrt{\log(2^{l+1})} 2^{l/2}\lesssim \frac{\sqrt{L_n}}{\sqrt{n}} 2^{\frac{L_n}{2}}.
		\end{align}
\textbf{Term $(a)$}
		Let $t>0$. Using the localisation property of the basis $(\psi_{lk})_{lk}$ and Jensen's inequality, we have
		\begin{align}\label{majo_1}
		\int\ninf{f^{L_n}-\hat{f}^{L_n}}d\Pi_{\alpha_n}(f|Y^n) &= \frac{1}{\sqrt{n\alpha_n}}\sum_{l=0}^{L_n} 2^{l/2} \int \max_{k} |\sqrt{n\alpha_n}(f_{lk}- Y_{lk})|d\Pi_{\alpha_n}(f|Y^n) \nonumber\\
		&\leq  \frac{1}{\sqrt{n\alpha_n}}\sum_{l=0}^{L_n} 2^{l/2} \frac{1}{t} \log \left(\sum_{k=0}^{2^l-1} \int e^{t|\sqrt{n\alpha_n}(f_{lk}- Y_{lk})|}d\Pi_{\alpha_n}(f|Y^n)\right) .
		\end{align}
		%vérifier l'hypothèse que l'on doit satisfaire pour appliquer Jensen, exactement comme dans le lemme pour les gaussienne.
Let $t \in \R$, we want to bound $\int e^{t\sqrt{n\alpha_n}(f_{lk}- Y_{lk})}d\Pi_{\alpha_n}(f|Y^n)$ uniformly over $l\leq L_n$ and $k=0,\dots, 2^l-1$. By definition of the $\alpha_n$-posterior distribution
		\begin{align}\label{def_TL}
		\int e^{t\sqrt{n\alpha_n}(f_{lk}- Y_{lk})}d\Pi_{\alpha_n}(f|Y^n) 
		&= \frac{\int_{\R} e^{t\sqrt{n\alpha_n}(u-Y_{lk})} e^{-\frac{n\alpha_n}{2}(u-Y_{lk})^2} \frac{1}{\sigma_l} \varphi(\frac{u}{\sigma_l}) du}{\int_{\R} e^{-\frac{n\alpha_n}{2}(u-Y_{lk})^2} \frac{1}{\sigma_l} \varphi(\frac{u}{\sigma_l}) du} \nonumber \\
		&= \frac{\int_{\R} e^{t(u-\sqrt{\alpha_n}\eps_{lk})} e^{-\frac{1}{2}(u-\sqrt{\alpha_n}\eps_{lk})^2} \varphi(\frac{1}{\sigma_l}(\frac{u}{\sqrt{n\alpha_n}}+ f_{0, lk})) du}{\int_{\R} e^{-\frac{1}{2}(u-\sqrt{\alpha_n}\eps_{lk})^2}  \varphi(\frac{1}{\sigma_l}(\frac{u}{\sqrt{n\alpha_n}}+ f_{0, lk})) du}.
		\end{align}
Then let us notice that for $B>R$, if $x \in [-B(l+1)^\mu ; B(l+1)^\mu]$, then $\varphi(x) \geq c_1 e^{-b_1B^{(1+\delta)}(l+1)} \geq C e^{-cl} $ and $1_{[-B(l+1)^\mu ; B(l+1)^\mu]}(\frac{1}{\sigma_l}(\frac{u}{\sqrt{n\alpha_n}}+ f_{0, lk})) \geq 1_{[-\sqrt{\log(n\alpha_n)}(B - R) ; \sqrt{\log(n\alpha_n)}(B- R) ]}(u) \geq 1_{[-1; 1]}(u)$. Combining this remark with \eqref{def_TL}, it follows,
		\begin{align}\label{majo_3}
		\int e^{t\sqrt{n\alpha_n}(f_{lk}- Y_{lk})}d\Pi_{\alpha_n}(f|Y^n)  \lesssim \frac{e^{\frac{t^2}{2}}e^{cl}}{\int_{-1}^{1} e^{-\frac{1}{2}(u-\sqrt{\alpha_n}\eps_{lk})^2}  du} \lesssim \frac{e^{\frac{t^2}{2}}e^{cl}}{\int_{-1}^{1} e^{-\frac{1}{2}(u-\eps_{lk})^2}  du}.
		\end{align}
Combining \eqref{majo_1} and \eqref{majo_3}, we obtain
		\begin{align*}
		\int\ninf{f^{L_n}-\hat{f}^{L_n}}d\Pi_{\alpha_n}(f|Y^n) \leq \frac{1}{\sqrt{n\alpha_n}}\sum_{l=0}^{L_n} 2^{l/2} \frac{1}{t} \log \left(\sum_{k=0}^{2^l-1} \frac{2e^{\frac{t^2}{2}}e^{cl}}{\int_{-1}^{1} e^{-\frac{1}{2}(u-\eps_{lk})^2}  du}\right).
		\end{align*}
		Taking the $E_0$-expectation and using Jensen's inequality, we get
		\begin{align*}
		E_0 \int\ninf{f^{L_n}-\hat{f}^{L_n}}d\Pi_{\alpha_n}(f|Y^n) &\leq \frac{1}{\sqrt{n\alpha_n}}\sum_{l=0}^{L_n} 2^{l/2} \frac{1}{t} \log \left(\sum_{k=0}^{2^l-1} 2e^{\frac{t^2}{2}}e^{cl} C\right) \\
		&= \frac{1}{\sqrt{n\alpha_n}}\sum_{l=0}^{L_n} 2^{l/2} (\frac{\log(2^{l+1}e^{cl} C)}{t} + \frac{t}{2} ).
		\end{align*} 
		Setting $t=\sqrt{2\log(2^{l+1}e^{cl} C)}$, we obtain  
		\begin{align}\label{majo_i_1}
		E_0 \int\ninf{f^{L_n}-\hat{f}^{L_n}}d\Pi_{\alpha_n}(f|Y^n) &\leq  \frac{1}{\sqrt{n\alpha_n}}\sum_{l=0}^{L_n} 2^{l/2} \sqrt{2\log(2^{l+1}e^{cl} C)} \lesssim \frac{\sqrt{L_n}}{\sqrt{n\alpha_n}}   2^{\frac{L_n}{2}}.
		\end{align}
		\textbf{Term $(c)$} Using the localisation property of the basis and Jensen's inequality, we have 
		\begin{align}\label{majo_gen_2}
		E_0\int \ninf{f^{L_n^c}}d\Pi_{\alpha_n}(f|Y^n)
		\leq \sum_{l>L_n} 2^{l/2} \frac{1}{t} \log (\sum_{k=0}^{2^l-1} E_0 \int e^{t|f_{lk}|}d\Pi_{\alpha_n}(f|Y^n)).
		\end{align}
		Let $t\in \R$, we have
		\begin{align*}
		\int e^{tf_{lk}}d\Pi_{\alpha_n}(f|Y^n) &= \frac{\int e^{tu} e^{-\frac{n\alpha_n}{2}(u-Y_{lk})^2 } \frac{1}{\sigma_l}\varphi(\frac{u}{\sigma_l})du}{\int e^{-\frac{n\alpha_n}{2}(u-Y_{lk})^2 } \frac{1}{\sigma_l}\varphi(\frac{u}{\sigma_l})du} \\
		&= \frac{\int e^{t(\frac{u}{\sqrt{n\alpha_n}} +f_{0,lk})} e^{-\frac{u^2}{2} + u\sqrt{\alpha_n}\eps_{lk}} \frac{1}{\sqrt{n\alpha_n}\sigma_l}\varphi(\frac{1}{\sigma_l}(\frac{u}{\sqrt{n\alpha_n}} +f_{0,lk}))du}{\int e^{-\frac{u^2}{2} + u\sqrt{\alpha_n}\eps_{lk}} \frac{1}{\sqrt{n\alpha_n}\sigma_l}\varphi(\frac{1}{\sigma_l}(\frac{u}{\sqrt{n\alpha_n}} +f_{0,lk}))du}.
		\end{align*}
		First, we bound from below the denominator. Denote $\mathcal{A}:= \{u : \left|\frac{1}{\sigma_l}(\frac{u}{\sqrt{n\alpha_n}} +f_{0,lk}) \right| \leq 1 \}$ and \newline $\mu(\mathcal{A}) := \int_{\mathcal{A}} \frac{1}{\sqrt{n\alpha_n}\sigma_l}\varphi(\frac{1}{\sigma_l}(\frac{u}{\sqrt{n\alpha_n}} +f_{0,lk}))du = \int_{-1}^{1}\varphi(u)du  $.
		Using Jensen's inequality with the exponential function, we obtain
		\begin{align*}
		D_{lk} &:= \int e^{-\frac{u^2}{2} + u\sqrt{\alpha_n}\eps_{lk}} \frac{1}{\sqrt{n\alpha_n}\sigma_l}\varphi(\frac{1}{\sigma_l}(\frac{u}{\sqrt{n\alpha_n}} +f_{0,lk}))du \\
		&\geq \mu(\mathcal{A}) \int_{\mathcal{A}} e^{-\frac{u^2}{2} + u\sqrt{\alpha_n}\eps_{lk}} \frac{1}{\sqrt{n\alpha_n}\sigma_l \mu(\mathcal{A})}\varphi(\frac{1}{\sigma_l}(\frac{u}{\sqrt{n}} +f_{0,lk}))du \\
		&\geq  \mu(\mathcal{A}) e^{\int_{\mathcal{A}} \left(-\frac{u^2}{2} + u\sqrt{\alpha_n}\eps_{lk}\right) \frac{1}{\sqrt{n\alpha_n}\sigma_l \mu(\mathcal{A})}\varphi(\frac{1}{\sigma_l}(\frac{u}{\sqrt{n\alpha_n}} +f_{0,lk}))du}.
		\end{align*}	
		Denote $\zeta_l= \int_{\mathcal{A}} u \frac{1}{\sqrt{n\alpha_n}\sigma_l \mu(\mathcal{A})}\varphi(\frac{1}{\sigma_l}(\frac{u}{\sqrt{n\alpha_n}} +f_{0,lk}))du $.
		\begin{align}\label{mino_D}	
		D_{lk} &\geq \mu(\mathcal{A}) e^{- \frac{1}{2}\sup_{u \in \mathcal{A}} u^2   + \sqrt{\alpha_n}\eps_{lk} \zeta_l} \geq \mu(\mathcal{A}) e^{- Cn\alpha_n(\sigma_l^2 +f_{0,lk}^2 )  + \sqrt{\alpha_n}\eps_{lk} \zeta_l},
		\end{align}
		for some constant $C>0$.
		%on utilise ici l'expression de l'ensemble mathcal{A}= [-\sqrt{n}(\sigma_l + f_{0,lk}) ; \sqrt{n}(\sigma_l - f_{0,lk}) ] \subset [-\sqrt{n}(\sigma_l + R\sigma_l) ; \sqrt{n}(\sigma_l + R\sigma_l) ]
		%
		Now split the integral of the numerator as follows 
		\begin{align*}
		&\int e^{t(\frac{u}{\sqrt{n\alpha_n}} +f_{0,lk})} e^{-\frac{u^2}{2} + u\sqrt{\alpha_n}\eps_{lk}} \frac{1}{\sqrt{n\alpha_n}\sigma_l}\varphi(\frac{1}{\sigma_l}(\frac{u}{\sqrt{n\alpha_n}} +f_{0,lk}))du\\
		&= \underbrace{\int_{\mathcal{A}} e^{t(\frac{u}{\sqrt{n\alpha_n}} +f_{0,lk})} e^{-\frac{u^2}{2} + u\sqrt{\alpha_n}\eps_{lk}} \frac{1}{\sqrt{n\alpha_n}\sigma_l}\varphi(\frac{1}{\sigma_l}(\frac{u}{\sqrt{n\alpha_n}} +f_{0,lk}))du}_{:=N^1_{lk}(t)} \\
		&+ \underbrace{\int_{\mathcal{A}^C} e^{t(\frac{u}{\sqrt{n\alpha_n}} +f_{0,lk})} e^{-\frac{u^2}{2} + u\sqrt{\alpha_n}\eps_{lk}} \frac{1}{\sqrt{n\alpha_n}\sigma_l}\varphi(\frac{1}{\sigma_l}(\frac{u}{\sqrt{n\alpha_n}} +f_{0,lk}))du}_{:=N^2_{lk}(t)}. 
		\end{align*}
		Using \eqref{mino_D} and Fubini's theorem, it follows
		\begin{align*}
		E_0\frac{N^1_{lk}(t)}{D_{lk}}  &\lesssim  e^{ Cn\alpha_n(\sigma_l^2 +f_{0,lk}^2 ) } \int_{\mathcal{A}} e^{t(\frac{u}{\sqrt{n\alpha_n}} +f_{0,lk})} e^{-\frac{u^2}{2}}	\underbrace{E_0( e^{\sqrt{\alpha_n}\eps_{lk} (u  - \zeta_l)} )}_{=e^{\frac{\alpha_n(u-\zeta_l)^2}{2}}} \frac{1}{\sqrt{n\alpha_n}\sigma_l}\varphi(\frac{1}{\sigma_l}(\frac{u}{\sqrt{n\alpha_n}} +f_{0,lk}))du\\
		&\lesssim e^{ Cn\alpha_n(\sigma_l^2 +f_{0,lk}^2 )  + \frac{\alpha_n\zeta_l^2}{2}} \int_{\mathcal{A}} e^{t(\frac{u}{\sqrt{n\alpha_n}} +f_{0,lk})} e^{-(1-\alpha_n)\frac{u^2}{2}}e^{-u\alpha_n\zeta_l}	 \frac{1}{\sqrt{n\alpha_n}\sigma_l}\varphi(\frac{1}{\sigma_l}(\frac{u}{\sqrt{n\alpha_n}} +f_{0,lk}))du\\
		&\lesssim e^{ Cn\alpha_n(\sigma_l^2 +f_{0,lk}^2 )  + \frac{\zeta_l^2}{2}} \int_{\mathcal{A}} e^{\left|t(\frac{u}{\sqrt{n\alpha_n}} +f_{0,lk})\right|} e^{|u\zeta_l|}	 \frac{1}{\sqrt{n\alpha_n}\sigma_l}\varphi(\frac{1}{\sigma_l}(\frac{u}{\sqrt{n\alpha_n}} +f_{0,lk}))du\\
		&\lesssim  e^{ Cn\alpha_n(\sigma_l^2 +f_{0,lk}^2 )  + \frac{\zeta_l^2}{2}} e^{|t|\sigma_l + |\zeta_l| \sqrt{n\alpha_n}(\sigma_l +|f_{0,lk}|)}.
		\end{align*}
		Since $|\zeta_l| \leq \sup_{u \in \mathcal{A}} |u| \leq \sqrt{n\alpha_n}(\sigma_l +|f_{0,lk}|)  $, we have   
		\begin{align}\label{majo_N_1}
		E_0\frac{N^1_{lk}(t)}{D_{lk}}  \lesssim  e^{ Cn\alpha_n(\sigma_l^2 +f_{0,lk}^2 )  +|t|\sigma_l},
		\end{align}
		for some constant $C>0$. On the other hand, the change of variables $u = \sqrt{n\alpha_n}(\sigma_ly + f_{0,lk})$ and \eqref{mino_D} give
		\begin{align*}
		\frac{N^2_{lk}(t)}{D_{lk}} \lesssim e^{Cn\alpha_n(\sigma_l^2 +f_{0,lk}^2 ) } \int_{[-1;1]^C} e^{tu\sigma_l} e^{-\frac{n\alpha_n}{2}(\sigma_l u - f_{0,lk})^2 + \sqrt{\alpha_n}\eps_{lk}(\sqrt{n\alpha_n}(\sigma_l u - f_{0,lk})-\zeta_l)} \varphi(u)du.
		\end{align*}
		Therefore, using Fubini's theorem, we get
		\begin{align*}
		E_0\frac{N^2_{lk}(t)}{D_{lk}} &\lesssim e^{Cn\alpha_n(\sigma_l^2 +f_{0,lk}^2 )} \int_{[-1;1]^C} e^{tu\sigma_l} e^{-\frac{n\alpha_n}{2}(\sigma_l u - f_{0,lk})^2} \underbrace{E_0(e^{\eps_{lk}\sqrt{\alpha_n}(\sqrt{n\alpha_n}(\sigma_l u - f_{0,lk})-\zeta_l)})}_{e^{\frac{\alpha_n(\sqrt{n\alpha_n}(\sigma_l u - f_{0,lk})-\zeta_l)^2}{2}}} \varphi(u)du \\
		&\lesssim e^{Cn\alpha_n(\sigma_l^2 +f_{0,lk}^2 ) +  \frac{\alpha_n\zeta_l^2}{2}} \int_{[-1;1]^C} e^{tu\sigma_l} e^{-(1-\alpha_n)\left(\frac{n\alpha_n}{2}(\sigma_l u - f_{0,lk})^2\right)} e^{-\alpha_n\sqrt{n\alpha_n}\zeta_l(\sigma_l u - f_{0,lk})} du \\
		&\lesssim e^{Cn\alpha_n(\sigma_l^2 +f_{0,lk}^2 ) + \frac{\zeta_l^2}{2}} \int_{[-1;1]^C} e^{tu\sigma_l} e^{-\alpha_n\sqrt{n\alpha_n}\zeta_l(\sigma_l u - f_{0,lk})} \varphi(u)du\\
		&\lesssim e^{Cn\alpha_n(\sigma_l^2 +f_{0,lk}^2 ) + \frac{\zeta_l^2}{2}+ \alpha_n\sqrt{n\alpha_n}|\zeta_l f_{0,lk}|} \int_{[-1;1]^C} e^{u(t\sigma_l - \alpha_n\sqrt{n\alpha_n} \zeta_l \sigma_l)} \varphi(u)du\\
        &\lesssim e^{Cn\alpha_n(\sigma_l^2 +f_{0,lk}^2 )} \int_{[-1;1]^C} e^{u(t\sigma_l - \alpha_n\sqrt{n\alpha_n} \zeta_l \sigma_l)} \varphi(u)du,
		\end{align*}
		where the constant $C$ may change for line to line. Using the tail behavior of $\varphi$, one can bound its Laplace tranform and we get 
		\begin{align}\label{majo_N_2}
		E_0\frac{N^2_{lk}(t)}{D_{lk}} &\lesssim e^{Cn\alpha_n(\sigma_l^2 +f_{0,lk}^2 )} e^{C(|t|\sigma_l + \alpha_n\sqrt{n\alpha_n} |\zeta_l|\sigma_l)^{\frac{\delta +1}{\delta}}} \lesssim  e^{Cn\alpha_n(\sigma_l^2 +f_{0,lk}^2 ) + C(|t|\sigma_l + \sqrt{n\alpha_n} |\zeta_l|\sigma_l)^{\frac{\delta +1}{\delta}}}.
		\end{align}
		Combining \eqref{majo_N_1} and \eqref{majo_N_2}, we obtain
		\begin{align}\label{majo_trans}
		E_0\int e^{tf_{lk}}d\Pi_{\alpha_n}(f|Y^n) \lesssim e^{C(n\alpha_n(\sigma_l^2 +f_{0,lk}^2 )  + |t|\sigma_l+ (|t|\sigma_l + \sqrt{n\alpha_n} |\zeta_l|\sigma_l)^{\frac{\delta +1}{\delta}})} .
		\end{align}
		%on utilise le fait que e^a +e^b \leq 2e^{a+b} 
		Combining \eqref{majo_gen_2} and \eqref{majo_trans}, and denoting $\phi_l=R2^{-l(\frac{1}{2}+\beta)}$, $\forall t>0$ we have
		\begin{align}\label{mino_2}
		E_0\int \ninf{f^{L_n^c}}d\Pi_{\alpha_n}(f|Y^n) 
		\leq \sum_{l>L_n} 2^{l/2} \frac{1}{t} \log (2^{l+1}e^{C(n\alpha_n(\sigma_l^2 +\phi_l^2)  + |t|\sigma_l+ (|t|\sigma_l + n\alpha_n |\sigma_l + \phi_l|\sigma_l)^{\frac{\delta +1}{\delta}})}    ).
		\end{align}
		Using the fact that for $l>L_n$, $n\alpha_n \phi_l \leq n\alpha_n\phi_{L_n} \leq \sqrt{\log(n\alpha_n} \leq L_n \leq l$ and $n\alpha_n\phi_{L_n}\sigma_{L_n} \leq \log(n\alpha_n)(L_n+1)^\mu \leq \log(n\alpha_n)^{\frac{\delta}{1+\delta}} \lesssim L_n^{\frac{\delta}{1+\delta}} \lesssim l^{\frac{\delta}{1+\delta}}$, we deduce that for any $t>0$,
		\begin{align*}
		E_0\int \ninf{f^{L_n^c}}d\Pi_{\alpha_n}(f|Y^n) &\leq \sum_{l>L_n} 2^{l/2} \frac{1}{t} \log (2^{l+1}e^{C(l + t\sigma_l+ (t\sigma_l + l^{\frac{\delta}{\delta+1}})^{\frac{\delta +1}{\delta}})} )\\
		&\lesssim \sum_{l>L_n} 2^{l/2} \frac{1}{t} (l + t\sigma_l+ (t\sigma_l + l^{\frac{\delta}{1+\delta}})^{\frac{\delta +1}{\delta}}).
		\end{align*}
		Choosing $t=l^{\frac{\delta}{\delta + 1}}\sigma_l^{-1}$, we obtain
		\begin{align}\label{majo_iii_1}
		E_0\int \ninf{f^{L_n^c}}d\Pi_{\alpha_n}(f|Y^n) &\lesssim \sum_{l>L_n} 2^{l/2} \sigma_l l^{-\frac{\delta}{\delta + 1}}(l + l^{\frac{\delta}{\delta + 1}}+ (l^{\frac{\delta}{\delta + 1}} + l^{\frac{\delta}{1+\delta}})^{\frac{\delta +1}{\delta}}) \nonumber\\ 
        &\lesssim \sum_{l>L_n} 2^{l/2} \sigma_l l^{-\frac{\delta}{\delta + 1}} (l + l^{\frac{\delta}{\delta + 1}}) \lesssim \sum_{l>L_n} 2^{l/2} \sigma_l l^{-\frac{\delta}{\delta + 1}} l  \nonumber \\
        &\lesssim \sum_{l>L_n}2^{l/2} \sigma_l l^{\frac{1}{\delta + 1}} \lesssim  \sum_{l>L_n}2^{l/2} 2^{- l(\frac{1}{2} +\beta)} \lesssim 2^{-\beta L_n}.
		\end{align} 
		\textbf{Conclusion.} Combining \eqref{majo_ii}, \eqref{majo_iv}, \eqref{majo_i_1} and \eqref{majo_iii_1}, one gets, as desired,
		\begin{align*}
		E_0(\int \ninf{f-f_0} d\Pi_{\alpha_n}(f|Y^n)) \lesssim \frac{1}{\sqrt{n\alpha_n}}  \sqrt{L_n} 2^{\frac{L_n}{2}} + 2^{-\beta L_n} \lesssim \frac{\log(n\alpha_n)}{n\alpha_n}^{\frac{\beta}{2\beta+1}}.
		\end{align*}
	\end{proof}

\section{Ancillary Results}\label{sec:ancillary}
\subsection{Contraction Rates}
\begin{lemma}\label{lem:mino_den_posterior} For any distribution $\Pi$ on $S$, any $C, \eps>0$ and $0 < \alpha \leq 1$, with $P_{0}$-probability at least $1-\frac{1}{C^2n\eps^2}$, we have
		\begin{align*}
			\int_{S} \frac{p_\eta^n(Y^n)^\alpha}{p_{\eta_0}^n(Y^n)^\alpha}d\Pi(\eta) \geq \Pi(B_n(\eta_0, \eps))e^{-\alpha(C+1)n\eps^2}.
		\end{align*}
	\end{lemma}

	\begin{proof}{Proof}
		Suppose $\Pi(B_n(\eta_0, \eps))>0$ (otherwise the result is immediate), and denote by $\bar{\Pi} = \frac{\Pi(\cdot \cap B_n(\eta_0, \eps))}{\Pi(B_n(\eta_0, \eps))}$ the normalized prior to $B_n(\eta_0, \eps)$.
		Now let us bound from below
		\begin{align}\label{eq:ELBO_lb}
			\int_{S} \frac{p_\eta^n(Y^n)^\alpha}{p_{\eta_0}^n(Y^n)^\alpha}d\Pi(\eta) \geq \int_{B_n(\eta_0, \eps)} \frac{p_\eta^n(Y^n)^\alpha}{p_{\eta_0}^n(Y^n)^\alpha}d\Pi(\eta) = \Pi(B_n(\eta_0, \eps)) \int \frac{p_\eta^n(Y^n)^\alpha}{p_{\eta_0}^n(Y^n)^\alpha}d\bar{\Pi}(\eta). 
		\end{align}
		Since $\bar{\Pi}$ is a probability measure on $S$, Jensen's inequality applied to the logarithm gives,
		\begin{align*}
				\log\left(\int \frac{p_\eta^n(Y^n)^\alpha}{p_{\eta_0}^n(Y^n)^\alpha}d\bar{\Pi}(\eta)\right) 
%				\geq \int \log\left(\frac{p_\eta^n(Y^n)^\alpha}{p_{\eta_0}^n(Y^n)^\alpha}\right)d\bar{\Pi}(\eta) =
				\geq  \alpha  \int \log\left(\frac{p_\eta^n(Y^n)}{p_{\eta_0}^n(Y^n)}\right)d\bar{\Pi}(\eta).
		\end{align*}	
		Consider now the random variable $Z := \int \log\left(\frac{p_\eta^n(Y^n)}{p_{\eta_0}^n(Y^n)}\right)d\bar{\Pi}(\eta)$. Then
		\begin{align*}
			E_0 |Z|  
%			E_0\left\arrowvert\int \log\left(\frac{p_\eta^n(Y^n)}{p_{\eta_0}^n(Y^n)}\right)d\bar{\Pi}(\eta)\right\arrowvert 
			 &\leq   \int_{B_n(\eta_0, \eps)} E_0  \left\arrowvert\log\left(\frac{p_\eta^n(Y^n)}{p_{\eta_0}^n(Y^n)}\right)\right\arrowvert d\bar{\Pi}(\eta) \\
			&= \int_{B_n(\eta_0, \eps)} \int \left\arrowvert\log\left(\frac{p_\eta^n(x)}{p_{\eta_0}^n(x)}\right)\right\arrowvert  p_{\eta_0}^n(x) d\mu^n(x)d\bar{\Pi}(\eta)) \\ & \leq n\eps^2 + 1. 
		\end{align*}
		Thus Z is integrable and using Fubini's theorem,
		\begin{align*}
			E_0 Z
%			 E_0(\int \log\left(\frac{p_\eta^n(Y^n)}{p_{\eta_0}^n(Y^n)}\right)d\bar{\Pi}(\eta))
			= \int_{B_n(\eta_0, \eps)}  \int \log\left(\frac{p_\eta^n(x)}{p_{\eta_0}^n(x)}\right) p_{\eta_0}^n(x) d\mu^n(x) d\bar{\Pi}(\eta) 
		 = \int_{B_n(\eta_0, \eps)} -K(p_{\eta_0}^n, p_{\eta}^n) d\bar{\Pi}(\eta) \geq -n\eps^2.
		\end{align*}
		Turning to the variance,
		\begin{align*}
			\text{Var}_0(Z) = \text{Var}_0(-Z) 
%			&= E_0 \left(\int \log\left(\frac{p_\eta^n(Y^n)}{p_{\eta_0}^n(Y^n)}\right)d\bar{\Pi}(\eta) + \int_{B_n(\eta_0, \eps)} K(p_{\eta_0}^n, p_{\eta}^n) d\bar{\Pi}(\eta)\right)^2\\
			&= E_0 \left(\int \log\left(\frac{p_{\eta_0}^n(Y^n)}{p_\eta^n(Y^n)}\right)d\bar{\Pi}(\eta) - \int_{B_n(\eta_0, \eps)} K(p_{\eta_0}^n, p_{\eta}^n) d\bar{\Pi}(\eta)\right)^2 \\
			&= E_0 \left(\int \log\left(\frac{p_{\eta_0}^n(Y^n)}{p_\eta^n(Y^n)}\right) - K(p_{\eta_0}^n, p_{\eta}^n) d\bar{\Pi}(\eta)\right)^2 \\
			&\leq  \int_{B_n(\eta_0, \eps)} E_0 \left(\log\left(\frac{p_{\eta_0}^n(Y^n)}{p_\eta^n(Y^n)}\right) - K(p_{\eta_0}^n, p_{\eta}^n)\right)^2 d\bar{\Pi}(\eta) 
%			&= \int_{B_n(\eta_0, \eps)} \int \left(\log\left(\frac{p_{\eta_0}^n(x)}{p_\eta^n(x)}\right) - K(p_{\eta_0}^n, p_{\eta_0}^n)\right)^2 p_{\eta_0}^n(x) d\mu^n(x) d\bar{\Pi}(\eta))
			\leq n\eps^2,
		\end{align*}
		using that $\bar{\Pi}$ is supported on $B_n(\eta_0, \eps)$. By Chebychev's inequality, 
		$P_0(|Z-E(Z)| \geq Cn\eps^2) \leq \frac{1}{Cn\eps^2} $.
		Thus, on the event $\{ |Z-E(Z)| \leq Cn\eps^2\}$, which has a probability at least $1 - \frac{1}{Cn\eps^2}$,
		\begin{align*}
			\log\left(\int \frac{p_\eta^n(Y^n)^\alpha}{p_{\eta_0}^n(Y^n)^\alpha}d\bar{\Pi}(\eta)\right) \geq \alpha(Z - EZ + EZ) \geq -\alpha(C+1)n\eps^2.
		\end{align*}
		Substituting this bound into \eqref{eq:ELBO_lb} then gives the result.
%		\begin{align*}
%			\int_{S} \frac{p_\eta^n(Y^n)^\alpha}{p_{\eta_0}^n(Y^n)^\alpha}d\Pi(\eta) \geq \Pi(B_n(\eta_0, \eps)) \int \frac{p_\eta^n(Y^n)^\alpha}{p_{\eta_0}^n(Y^n)^\alpha}d\bar{\Pi}(\eta) \geq \Pi(B_n(\eta_0, \eps)) e^{-\alpha(C+1)n\eps^2}.
%		\end{align*}
	\end{proof}

	\begin{lemma}\label{lemma_sieve}
		Let $A_n$ be measurable sets, $0 < \alpha_n \leq 1$  and $\eps_n$ be a non-negative sequence such that $n\alpha_n \eps_n^2 \rightarrow \infty$. If
		\begin{align*}
			\frac{\Pi(A_n)}{\Pi(B_n(\eta_0, \eps_n)) e^{-2n\alpha_n\eps_n^2}} = o(1),
		\end{align*}
		then $\Pi_{\alpha_n}(A_n |Y^n) \to^{P_{0}} 0 $.
	\end{lemma}
	
\begin{proof}{Proof}
Applying H\"older's inequality to the right-hand side of \eqref{eq:Bayes_formula} implies
\begin{align*}
			E_0 \Pi_{\alpha_n}(A_n|Y^n)   &\leq  \frac{\int_{A_n} \left(\int p_\eta^n(x) d\mu^n(x)\right)^{\alpha_n} \left(\int p_{\eta_0}^n(x) d\mu^n(x)\right)^{1-\alpha_n} d\Pi(\eta)}{\Pi(B_n(\eta_0, \eps_n)) e^{-2{\alpha_n}n\eps_n^2}} + o(1)\\
			&= \frac{\Pi(A_n)}{\Pi(B_n(\eta_0, \eps_n)) e^{-2{\alpha_n}n\eps_n^2}} + o(1) = o(1).
		\end{align*}
\end{proof}	

%	\begin{proof}{Proof} On a subset $C_n$ of probability at least $1-\frac{1}{n\eps_n^2}$, by Lemma \ref{lem:mino_den_posterior} we have
%		\begin{align*}
%			\Pi_{\alpha_n}(A_n|X^n) \leq \frac{\int_{A_n} \frac{p_\eta^n(X^n)^{\alpha_n}}{p_{\eta_0}^n(X^n)^{\alpha_n}} d\Pi(\eta)}{\Pi(B_n(\eta_0, \eps_n)) e^{-{\alpha_n}2n\eps_n^2}}1_{C_n} + 1_{C_n^c}.
%		\end{align*}
%		Taking the $E_0$-expectation, and using Fubini's theorem and Hölder's inequality, one obtains
%		\begin{align*}
%			E_{0}(\Pi_{\alpha_n}(A_n|X^n))
%			&\leq \frac{\int_{A_n} \int p_\eta^n(x)^{\alpha_n} p_{\eta_0}^n(x)^{1-\alpha_n} d\mu^n(x)d\Pi(\eta)}{\Pi(B_n(\eta_0, \eps_n)) e^{-{\alpha_n}2n\eps_n^2}} + \frac{1}{n\eps_n^2}\\ 
%			&\leq  \frac{\int_{A_n} \left(\int p_\eta^n(x) d\mu^n(x)\right)^{\alpha_n} \left(\int p_{\eta_0}^n(x) d\mu^n(x)\right)^{1-\alpha_n} d\Pi(\eta)}{\Pi(B_n(\eta_0, \eps_n)) e^{-{\alpha_n}2n\eps_n^2}} + o(1)\\
%			&\leq \frac{\Pi(A_n)}{\Pi(B_n(\eta_0, \eps_n)) e^{-{\alpha_n}2n\eps_n^2}} + o(1).
%		\end{align*}
%		Combining this bound with assumption \eqref{assum_sieve}, one deduces that  $E_0(\Pi_{\alpha_n}(A_n|X^n)) \rightarrow 0$.
%	\end{proof}
	
	\begin{lemma}\label{lemma_prior_mass_condition} Consider density estimation on [0,1] with true density $f_0 \in \mathcal{C}^\beta([0,1])$ for some $\beta \in (0,1]$, bounded away from 0. Let $\Pi = \Pi_n$ denote the histogram prior \eqref{def::RHP} satisfying $K_n=o\left(n\alpha_n/ \log(n\alpha_n)\right)$ and for all $i \in \{1,\dots, K_n\}$, $\frac{1}{(n\alpha_n)^b}\leq \delta_{i,n} \leq 1$ for some $b>0$. Then the sequence $\eps_n^2 = K_n\log(n\alpha_nK_n) / (n\alpha_n) + K_n^{-2\beta}$ satisfies $\Pi(B_n(f_0, M\eps_n)) \geq e^{-n\alpha_n(M\eps_n)^2}$ for some $M>0$. 
	\end{lemma}
	
		\begin{proof}{Proof}
		Using that the Kullback-Leiber and its $2^{nd}$-variation tensorizes in density estimation, we may write 
		%
		%		\begin{align*}
			%			K(p_{f_0}^{(n)},p_{f}^{(n)})= n K(f_0,f) \quad V(p_{f_0}^{(n)},p_{f}^{(n)}) = n V(f_0,f).
			%		\end{align*}
		%
		%		Thus
		\begin{align}\label{def_KL}
			B_n(f_0,\eps)= B_1(f_0, \eps)= \{f \in \mathcal{F}: \: K(f_0,f)\leq\eps^2, V(f_0,f)\leq \eps^2  \}.
		\end{align}
		Let $\rho_n^2 = \log(n\alpha_n K_n)/(n\alpha_n K_n)$. Since $K_n= o(\frac{n\alpha_n}{\log(n\alpha_n)})$, it holds that $(\rho_n K_n)^2 = \\ K_n\log(n\alpha_n K_n)/(n\alpha_n) = o(1)$ and thus $\rho_n \leq K_n^{-1}$ for $n$ large enough. This bound, the assumption $\delta_{i,n} \leq 1$ together with Lemma \ref{lemma_6_1_GGV} give that there exist positive constants $C$ and $c$ such that for all integer $n$, 
		\begin{align}\label{calc_norme_l_1}
			\Pi(f \in H^1_{K_n}, \|f-f_{0,K_n}\|_1 \leq 2 \rho_n) \geq  Ce^{-cK_n\log(\frac{1}{\rho_n})} \prod_{i=1}^{K_n} \delta_{i,n}.
		\end{align}
		Using basic properties of histograms, we also have
		\begin{align}\label{calc_norme_l_inf}
			\Pi(f \in H^1_{K_n},\|f-f_{0,K_n}\|_\infty \leq 2K_n\rho_n) &\geq \Pi(f \in H^1_{K_n}, \|f-f_{0,K_n}\|_1 \leq 2\rho_n).
		\end{align}
		Since $f_0$ satisfies $m \leq f_0\leq M$ for some $M>m>0$, we also have $m\leq f_{0,K_n}\leq M$ for all $n$.	Now let $f \in H^1_{K_n}$ such that $\|f-f_{0,K_n}\|_\infty \leq 2K_n\rho_n$. Since $K_n\rho_n \rightarrow 0$, for $n$ large enough, $\frac{m}{2}<f<2M$. Since $f$ and $f_0$ are bounded away from zero and infinity, using that $\log(1+x) \leq x$,
		\begin{align*}
			K(f_0,f) 
			%			&= \int_{0}^{1} \log(\frac{f_0}{f})f_0
			&=\int_{0}^{1} \log\left(1+\frac{f_0-f}{f} \right)f_0 \leq  \int_{0}^{1} \frac{f_0-f}{f}(f_0-f+f) \nonumber\
			= \int_{0}^{1} \frac{(f_0-f)^2}{f} \leq \frac{2}{m}\|f-f_0\|_{\infty}^2.
		\end{align*}
		Also, since $x \mapsto \log x$ is $\frac{1}{r}$-Lipschitz on $[r, \infty)$,
		\begin{align*}
			V(f_0,f)
			%			&= \int_{0}^{1} \left(\log(\frac{f_0}{f}) - K(f_0,f)\right)^2f_0
			%			&= \int_{0}^{1} \left( \log\frac{f_0}{f}\right)^2 f_0 - K(f_0,f)^2
			& \leq \int_{0}^{1} \left(\log\frac{f_0}{f}\right)^2 f_0 \nonumber 
			%			&= \int_{0}^{1} (|\log(f_0) - \log(f)|)^2f_0 
			\leq \frac{4}{m^2} \int_{0}^{1} |f_0-f|^2 f_0 \leq \frac{4}{m^2}\|f-f_0\|_{\infty}^2.
		\end{align*}
		Moreover, for $f_0 \in \mathcal{C}^\beta([0,1])$ and any $f \in H^1_{K_n}$ such that $\|f-f_{0,K_n}\|_\infty \leq 2K_n\rho_n$,
		\begin{align*}
			&\|f-f_0\|_{\infty}^2 \leq 2\|f-f_{0,K_n}\|_{\infty}^2 + 2\|f_0-f_{0,K_n}\|_{\infty}^2 \leq 8(K_n\rho_n)^2 + 2K_n^{-2\beta} \lesssim \eps_n^2.
		\end{align*}
		Combining the last three displays thus implies that $K(f_0,f) \leq D \eps_n^2$ and $V(f_0,f) \leq D \eps_n^2$ for some constant $D = D(m) = D(f_0)>0$. Together with \eqref{def_KL}-\eqref{calc_norme_l_inf}, we obtain
		%		\begin{align}\label{calc_KL}
			%			&K(f,f_0) \leq \frac{2}{m}\|f-f_0\|_{\infty}^2 \leq \frac{4}{m}(\|f-f_{0,K_n}\|_{\infty}^2 + \|f_0-f_{0,K_n}\|_{\infty}^2) \leq \frac{16}{m}((K_n\rho_n)^2 + \frac{1}{K_n^{2\beta}})= \frac{16}{m}\eps_n^2 \leq D\eps_n^2 \nonumber \\
			%			&V(f_0,f) \leq \frac{4}{m^2}\|f-f_0\|_{\infty}^2 \leq \frac{32}{m^2}\eps_n^2 \leq D\eps_n^2, 
			%		\end{align}
		%		for some constant $D>0$.
		% ici le D>1 assure que la suite D\eps_n est bien plus grande que $\rho_n$, et permet de faire les bonnes minorations dans ce qui suit
		%
		%		Using \eqref{calc_norme_l_1}, \eqref{calc_norme_l_inf}, \eqref{calc_KL}, and \eqref{def_KL}, we obtain 
		%
		\begin{align}\label{masse_vois_inf}
			\Pi(B_n(f_0, \sqrt{D}\eps_n)) &= \Pi(\{f, ~ K(f_0,f) \leq D \eps_n^2, \; V(f_0,f) \leq D \eps_n^2 \}) 
			%			\geq \Pi(f \in H^1_{K_n},\|f-f_{0,K_n}\|_\infty \leq 2K_n\rho_n) 
			\nonumber\\ 
			&\geq \Pi(f \in H^1_{K_n}, \|f-f_{0,K_n}\|_1 \leq 2\rho_n) \geq  Ce^{-cK_n\log(\frac{1}{\rho_n})} \prod_{i=1}^{K_n} \delta_{i,n}. 
		\end{align}
		Using the assumption on the weights $(\delta_{i,n})$ and the definition of $\rho_n$ yields that 
		\begin{align*}
			\Pi(B_n(f_0, \sqrt{D}\eps_n)) &\geq Ce^{-cK_n\log(\frac{1}{\rho_n})} e^{-bK_n\log(n\alpha_n)} \geq Ce^{-cK_n\log(\frac{1}{\rho_n})},
		\end{align*}
		where the constants $C$ and $c$ may change from line to line. Finally, since the sequence $\rho_n^2$ satisfies $\frac{1}{\rho_n^2}\log(\frac{1}{\rho_n}) \leq n\alpha_nK_n$ and thus $K_n\log(\frac{1}{\rho_n}) \leq n \alpha_n(K_n\rho_n)^2 $, it follows 	
		\begin{align*}
			\Pi(B_n(f_0, \sqrt{D}\eps_n)) &\geq Ce^{-cK_n\log(\frac{1}{\rho_n})}   \geq Ce^{-cn\alpha_n(K_n\rho_n)^2 } \geq Ce^{-cn\alpha_n\eps_n^2 }  = Ce^{-n\alpha_n(\sqrt{c}\eps_n)^2 }.
		\end{align*} 
		Denoting $D':= \max(\sqrt{D}, \sqrt{c}) + 1$, for $n$ large enough we have
		\begin{align*}
			\Pi(B_n(f_0, D'\eps_n)) \geq e^{-n\alpha_n(D'\eps_n)^2}.
		\end{align*}
		
		%Ici, le max permet de faire les bonnes minorations dans l'expo (notament on enlève c), et le +1 permet d'enlever aussi le C devant l'expo. 
	\end{proof}
	
	\begin{lemma}\label{lemma_6_1_GGV}
		Let $X_1, \dots, X_K$ be distributed according to the Dirichlet distribution on the K-simplex with parameters $\delta= (\delta_1, \dots, \delta_K)$, where $0< \delta_i \leq 1$ for all $i$ . Let $x_0=(x_{10}, \dots, x_{K0})$ be any point on the K-simplex. There exist positive constants $c$, $C$, independent of $K$, $\delta$ and $x_0$  such that, for $\eps \leq K^{-1}$
		\begin{align*}
			P\left(\sum_{i=1}^{K} |X_i -x_{i0}| \leq 2\eps \right) \geq C e^{-cK\log(\frac{1}{\eps})}\prod_{i=1}^{K} \delta_i
		\end{align*}
	\end{lemma}
	\begin{proof}{Proof}
		The proof is the same as that of Lemma 6.1 in \citet{GGV}, except one keeps track of the dependence on the Dirichlet parameters.
	\end{proof}

\subsection{Bernstein--von Mises Results}

\begin{proof}{Proof of Lemma \ref{thm} }
		Let $f \in A_n$. First we have
		\begin{align}\label{calculation_1}
			\alpha_n \ell_n(f_t) &= \alpha_n  \sum_{i=1}^{n} \Big(\log f(Y_i) - \frac{t \tilde{\psi}_{[K_n]}(Y_i) }{\sqrt{n\alpha_n}} - \log (F(e^{-\frac{t \tilde{\psi}_{[K_n]}}{\sqrt{n\alpha_n}}}) ) \Big) \nonumber\\
			&= \alpha_n l_n(f) - t \sqrt{n \alpha_n} \frac{1}{n}\sum_{i=1}^{n}  \tilde{\psi}_{[K_n]}(Y_i)-n\alpha_n \log (F(e^{-\frac{t \tilde{\psi}_{[K_n]}}{\sqrt{n\alpha_n}}}) ).
		\end{align}
		Let us expand the term $\log (F(e^{-\frac{t \tilde{\psi}_{[K_n]}}{\sqrt{n\alpha_n}}}))$. Throughout the calculations below, one can keep track of the uniformity of the remainder terms and check that the remainder in the final expansion is uniform over $A_n$. The fact that $\tilde{\psi}$ is bounded (and so is $\tilde{\psi}_{[K_n]}$) ensures this uniformity. By expanding the logarithm around 1, 
		\begin{align*}
			\log (F(e^{-\frac{t \tilde{\psi}_{[K_n]}}{\sqrt{n\alpha_n}}})) &= \log \int fe^{-\frac{t \tilde{\psi}_{[K_n]}}{\sqrt{n\alpha_n}}}= \log \int_{0}^{1} f (1 -\frac{t \tilde{\psi}_{[K_n]}}{\sqrt{n\alpha_n}} + \frac{t^2 \tilde{\psi}_{[K_n]}^2}{2n\alpha_n} + o(\frac{1}{n\alpha_n})) \\
			&= \log\Big( 1 - \frac{t}{\sqrt{n\alpha_n}} \int_{0}^{1} f \tilde{\psi}_{[K_n]} + \frac{t^2}{2n\alpha_n} \int_{0}^{1} f \tilde{\psi}_{[K_n]}^2 +o(\frac{1}{n\alpha_n}))\\
			&= - \frac{t}{\sqrt{n\alpha_n}} \int_{0}^{1} f \tilde{\psi}_{[K_n]} + \frac{t^2}{2n\alpha_n} \left(\int_{0}^{1} f \tilde{\psi}_{[K_n]}^2 - \left(\int_{0}^{1} f \tilde{\psi}_{[K_n]}\right)^2 \right) + o(\frac{1}{n\alpha_n}).
		\end{align*}
		Since over $A_n$, $f$ is an histogram of size $K_n$, we deduce
		\begin{align*}
			\log (F(e^{-\frac{t \tilde{\psi}_{[K_n]}}{\sqrt{n\alpha_n}}}))=- \frac{t}{\sqrt{n\alpha_n}} \int_{0}^{1} f \tilde{\psi}_{f_0} + \frac{t^2}{2n\alpha_n} \left(\int_{0}^{1} f \tilde{\psi}_{[K_n]}^2 - \left(\int_{0}^{1} f \tilde{\psi}_{f_0}\right)^2 \right) + o(\frac{1}{n\alpha_n}).
		\end{align*}
		Then, the facts that $\|f-f_{0,K_n}\|_1\leq \eps_n$ over $A_n$ and $\tilde{\psi}_{[K_n]}$ is bounded imply
		\begin{align}\label{expansion}
			&\log (F(e^{-\frac{t \tilde{\psi}_{f_0}}{\sqrt{n\alpha_n}}})) \nonumber\\
			&= - \frac{t}{\sqrt{n\alpha_n}} \int_{0}^{1} f \tilde{\psi}_{f_0} 
			+ \frac{t^2}{2n\alpha_n} \left(\int_{0}^{1} f_{0,K_n} \tilde{\psi}_{[K_n]}^2 - \Big( \int_{0}^{1} f_{0, K_n}\tilde{\psi}_{f_0}\Big) ^2 \right)+o(\frac{1}{n\alpha_n}) \nonumber\\
			&=- \frac{t}{\sqrt{n\alpha_n}} \int_{0}^{1} f \tilde{\psi}_{f_0} 
			+ \frac{t^2}{2n\alpha_n} V_{K_n}  +o(\frac{1}{n\alpha_n}).
		\end{align}
		Thus, combining \eqref{calculation_1} and \eqref{expansion}, we have 
		\begin{align*}
			\alpha_n l_n(f_t)
			&= \alpha_n l_n(f) - t \sqrt{n \alpha_n} \frac{1}{n}\sum_{i=1}^{n}  \tilde{\psi}_{[K_n]}(Y_i)+ t\sqrt{n\alpha_n} \int_{0}^{1} f \tilde{\psi}_{f_0} - \frac{t^2}{2} V_{K_n} + o(1)\\
			&= \alpha_n l_n(f) + t\sqrt{n\alpha_n}(-\frac{1}{n}\sum_{i=1}^{n}  \tilde{\psi}_{[K_n]}(Y_i) + \psi(f)-\psi(f_0))  - \frac{t^2}{2} V_{K_n} + o(1).
		\end{align*}
		By rearranging and using the definition of $\hat{\psi}_{[K_n]}$,
		\begin{align}\label{dvpmt_final}
			\alpha_n l_n(f)+ t\sqrt{n\alpha_n}(\psi(f)-\hat{\psi}_{[K_n]} )=\alpha_n l_n(f_t) + \frac{t^2}{2}V_{K_n} + o(1).
		\end{align}
		Let us show that $V_{K_n} \rightarrow \int f_0 \tilde{\psi}^2$. Since $K_n \rightarrow \infty$, $\int (\tilde{\psi}_{[K_n]} - \tilde{\psi})^2 =o(1)$ and since $f_0$ is bounded it follows $\int f_0(\tilde{\psi}_{[K_n]} - \tilde{\psi})^2 =o(1) $. Hence $\left|\int f_0\tilde{\psi}_{[K_n]}^2 - \int f_0\tilde{\psi}^2 \right| = o(1)$ and thus $\left|\int f_{0,K_n}\tilde{\psi}_{[K_n]}^2 - \int f_0\tilde{\psi}^2 \right| = o(1)$. Moreover, $|\int f_0\tilde{\psi}_{[K_n]}| = |\int f_0(\tilde{\psi}_{[K_n]} - \tilde{\psi} )| \leq \|f_0\|_{\infty}\|\tilde{\psi}_{[K_n]} - \tilde{\psi}\|_2 = o(1)$. Finally $V_{K_n}= \int f_0 \tilde{\psi} ^2 +o(1)$. Using this result together with \eqref{dvpmt_final} and Assumption \eqref{assum_RHP_2} it follows
		\begin{align*}
			&E_{\alpha_n}(e^{t\sqrt{n\alpha_n}(\psi(f)-\hat{\psi}_{[K_n]})}|Y^n, A_n)=\frac{\int_{A_n} e^{t\sqrt{n\alpha_n}(\psi(f)-\hat{\psi}_{[K_n]})}e^{\alpha_nl_n(f)} d\Pi(f)}{\int e^{\alpha_nl_n(f) } d\Pi(f)} \\
			& = \frac{\int_{A_n} e^{\alpha_nl_n(f_t) + \frac{t^2}{2}V_{K_n}+o(1)} d\Pi(f)}{\int e^{\alpha_nl_n(f)} d\Pi(f)}
			=e^{\frac{t^2}{2}V_{K_n}}(1+o(1))\frac{\int_{A_n} e^{\alpha_nl_n(f_t)} d\Pi(f)}{\int e^{\alpha_nl_n(f)} d\Pi(f)} \\
            &= e^{\frac{t^2}{2}F_0(\tilde{\psi}^2_{f_0})}(1+o_P(1)).
		\end{align*}
		The last estimate is for the restricted distribution $\Pi_{\alpha_n}(\cdot|Y^n, A_n)$ but Assumption \eqref{assum_RHP_1} implies
		that the unrestricted version also follows and this proves Lemma \ref{thm}.
	\end{proof}
	%thm qui utilise le faut que on a un random hisgram prior, et que donc on peut regarder l'estimateur défini à partir de la projection. et aussi, le fait que l'on obtient alors la projection dans la variance asymptotique ne nous embête pas trop car comme K_n tend vers infty, ça va revenir au même à la fin, et idem, on peut supposer que qu'on est sur un voisinage autour de f_{0,Kn}, cela ne nous embête pas trop car on retrouve quand même la bonne variance optimale juste en supposant que K_n tend vers infty. 
	%enfin, ces 2 changements (estimateur défini à partir de la projection et concentration autour de la projection, nous sont utiles pour étudier la condition de changement de variable dans la suite).
	%
	\begin{proof}{Proof of Lemma \ref{thm_conv_int}}
		Set $\eps_n= \tilde{\eps}_n +( e^{2\frac{|t|}{\sqrt{n\alpha_n}} \|\tilde{\psi}\|_{\infty}} - 1)$. First, $\eps_n \rightarrow 0$ and  since $\eps_n \geq \tilde{\eps}_n$, we have $\Pi_{\alpha_n}(A_n|Y^n)=1 +o_P(1)$. Now let us show the convergence \eqref{change_var}.	For $k \geq 1$, let us set
		\begin{align*}
			U_{k}=\left\{(\omega_1, \dots, \omega_{k-1}) \in (0,1)^{k-1}, \quad \sum_{i=1}^{k-1}\omega_i<1 \right\}.
		\end{align*}
		Throughout the proof, we will use the notation $\omega_{k}=1-\sum_{j=1}^{k-1} \omega_j$. Let us denote by $H$ the map
		\begin{align*}  
			\begin{array}{cccc}
				H: & U_{k}& \rightarrow & H^1_k \\
				&  (\omega_1, \dots, \omega_{k-1}) & & x \rightarrow k\sum_{j=1}^{k}\omega_j1_{I_j}(x).\\
			\end{array}
		\end{align*}
		By definition of the prior distribution, we have 
		\begin{align}\label{calc_chan_var_1}
			&\int_{A_n} e^{\alpha_nl_n(f_{t})} d\Pi(f)= \int_{H^1_{K_n}} 1_{f\in A_n} e^{\alpha_nl_n(f_{t})} d\Pi(f) \\
			&= \int_{U_{K_n-1}} 1_{H(\omega_1, \dots, \omega_{K_n-1}) \in A_n} e^{\alpha_nl_n(H(\omega_1, \dots, \omega_{K_n-1}) e^{-\frac{t \bar{\psi}_{[K_n]}}{\sqrt{n\alpha_n}}})/\int H(\omega_1, \dots, \omega_{K_n-1})e^{-\frac{t \bar{\psi}_{[K_n]}}{\sqrt{n\alpha_n}}}))} \\
			&\times \frac{1}{B(\delta)}\prod_{i=1}^{K_n}\omega_i^{\delta_{i,n}-1} d\omega_1\dots d\omega_{K_n-1}. \nonumber
		\end{align}
		For an integer $n$ and $j \in \{1,\dots, K_n\}$, denote $\gamma_j=e^{t\tilde{\psi}_j / \sqrt{n\alpha_n}}$ with $\tilde{\psi}_j=K_n\int_{I_j}\tilde{\psi}$. For $k$ an integer and $x \in ]0,+\infty[^k$, let us denote by $S_x$ the map
		\begin{align*}
			\begin{array}{cccc}
				S_x: & U_{k}& \rightarrow & ]0,+
				\infty[ \\
				&  (\omega_1, \dots, \omega_{k-1}) & & \sum_{j=1}^{k}\omega_jx_j.\\
			\end{array}
		\end{align*}
		For an integer $k$ and vector $x \in ]0,+\infty[^k$, denote
		\begin{align*}
			\begin{array}{cccc}
				\phi_x: & U_{k}& \rightarrow & U_{k} \\
				&  (\omega_1, \dots, \omega_{k-1}) & & (\frac{\omega_1x_1}{S_{x}(\omega)}, \dots, \frac{\omega_{k-1}x_{k-1}}{S_{x}(\omega)}).\\
			\end{array}
		\end{align*}
		This mapping is well defined and note that 
  \begin{align}\label{calc_chan_var_2}
      \frac{H(\omega_1, \dots, \omega_{K_n-1}) e^{-\frac{t \bar{\psi}_{[K_n]}}{\sqrt{n\alpha_n}}}}{\int H(\omega_1, \dots, \omega_{K_n-1})e^{-\frac{t \bar{\psi}_{[K_n]}}{\sqrt{n\alpha_n}}}} = H(\phi_{\gamma^{-1}}(\omega_1, \dots, \omega_{K_n-1})).
  \end{align}
	Moreover, one can show that for all integer $k$ and for all $x \in ]0,+\infty[^k$, $\phi_x$ is bijective and its inverse is $\phi_{x^{-1}}$.
		From Lemma 5 in the supplemental article of \citet{CR2015}, the mappings $\phi_x$ and $\phi_{x^{-1}}$  are $\mathcal{C}^1$, and the determinant of the jacobian matrix of the map $\phi_x$ is given by 			
		\begin{align}
			\det(D_{\phi_x}(\omega_1, \dots, \omega_{k-1}))=\frac{1}{S_{x}(\omega)^k}\prod_{i=1}^{k}x_i.
		\end{align}
		Let us combine \eqref{calc_chan_var_1} and \eqref{calc_chan_var_2} and then make the change of variables $(\xi_1, \dots, \xi_{K_n-1}) \rightarrow \phi_{\gamma}(\xi_1, \dots, \xi_{K_n-1})$ in \eqref{calc_chan_var_1} and using $1-\sum_{i=1}^{K_n-1}\frac{\gamma_i\xi_i}{S_\gamma(\xi)}=\frac{\gamma_{K_n}\xi_{K_n}}{S_\gamma(\xi)}$, it follows
		\begin{align}\label{calc_chan_var_4}
			&\int_{A_n} e^{\alpha_nl_n(f_{t})} d\Pi(f)&\nonumber\\
			&=\int_{U_{K_n-1}} 1_{H(\phi_{\gamma}(\xi_1, \dots, \xi_{K_n-1}))\in A_n} e^{\alpha_nl_n(H(\xi_1, \dots, \xi_{K_n-1}))} \frac{1}{B(\delta)}\prod_{i=1}^{K_n}\left(\frac{\gamma_i\xi_i}{S_{\gamma}(\xi)}\right)^{\delta_{i,n}-1} \frac{1}{S_{\gamma}(\xi)^{K_n}}\prod_{i=1}^{K_n}\gamma_i d\xi_i\nonumber\\
			&=\int_{U_{K_n-1}} 1_{H(\phi_{\gamma}(\xi_1, \dots, \xi_{K_n-1}))\in A_n} e^{\alpha_nl_n(H(\xi_1, \dots, \xi_{K_n-1}))} \underbrace{\prod_{i=1}^{K_n}\gamma_i^{\delta_{i,n}}}_{(*)} \underbrace{\frac{1}{S_{\gamma}(\xi)^{\sum_{i=1}^{K_n} \delta_{i,n}}}}_{(**)} \frac{1}{B(\delta)}\prod_{i=1}^{K_n}\xi_i^{\delta_{i,n}-1} d\xi_i.
		\end{align}
		For the term $(*)$, which does not depend on $\xi$, we have
		\begin{align*}
			&\prod_{j=1}^{K_n}\gamma_j^{\delta_{j,n}}= \prod_{j=1}^{K_n}e^{t\tilde{\psi}_j\delta_{j,n}/\sqrt{n\alpha_n}}=e^{t\sum_{j=1}^{K_n}\tilde{\psi}_j\delta_{j,n}/\sqrt{n\alpha_n} } \\
			&  e^{-\frac{|t|}{\sqrt{n\alpha_n}} \|\tilde{\psi}\|_\infty\sum_{j=1}^{K_n}\delta_{j,n}} \leq \prod_{j=1}^{K_n}\gamma_j^{\delta_{j,n}} \leq e^{\frac{|t|}{\sqrt{n\alpha_n}} \|\tilde{\psi}\|_\infty\sum_{j=1}^{K_n}\delta_{j,n}},
		\end{align*}
		so that $\prod_{j=1}^{K_n}\gamma_j^{\delta_{j,n}} = 1+o(1)$ using the condition \eqref{condition_weights}.
		As for the term $(**)$, for $\xi \in U_{K_n-1}$ we have $ S_{\gamma}(\xi) = \sum_{j=1}^{K_n} \gamma_j \xi_j = \sum_{j=1}^{K_n} e^{\frac{t}{\sqrt{n\alpha_n}}\tilde{\psi}_j}  \xi_j $, and thus
    \begin{align}
        e^{-\frac{t}{\sqrt{n\alpha_n}}\|\tilde{\psi}\|_\infty}  \sum_{j=1}^{K_n} \xi_j \leq S_{\gamma}(\xi) \leq e^{\frac{t}{\sqrt{n\alpha_n}}\|\tilde{\psi}\|_\infty}  \sum_{j=1}^{K_n} \xi_j \label{calc_inter}  \\
        e^{-\frac{|t|}{\sqrt{n\alpha_n}} \|\tilde{\psi}\|_\infty \sum_{i=1}^{K_n} \delta_{i,n}}\leq  S_{\gamma}(\xi)^{-\sum_{i=1}^{K_n} \delta_{i,n}} \leq e^{\frac{|t|}{\sqrt{n\alpha_n}} \|\tilde{\psi}\|_\infty \sum_{i=1}^{K_n} \delta_{i,n}}.
    \end{align}
        Using Assumption \eqref{condition_weights} gives that term $(**)$ is $1+o(1)$ uniformly over $U_{K_n-1}$. By combining \eqref{calc_chan_var_4} and the results on term $(*)$ and term $(**)$, we obtain
		\begin{align}\label{calc_chan_var_5}
			&\int_{A_n} e^{\alpha_nl_n(f_{t})} d\Pi(f) \\
			&=(1+o(1))\int_{U_{K_n-1}} 1_{H(\phi_{\gamma}(\xi_1, \dots, \xi_{K_n-1}))\in A_n} e^{\alpha_nl_n(H(\xi_1, \dots, \xi_{K_n-1}))} \frac{1}{B(\delta)}\prod_{i=1}^{K_n}\xi_i^{\delta_{i,n}-1} d\xi_1\dots d\xi_{K_n-1}\nonumber.
		\end{align}
		Next, let us show that we have the inclusion
		\begin{align}\label{inclusion}
			\{(\xi_1,\dots,\xi_{K_n-1})& \in U_{K_n-1} , \: \|H(\xi_1,\dots,\xi_{K_n-1}) - f_{0,K_n} \|_1 \leq \tilde{\eps}_n\} \nonumber\\
			\subset &\{(\xi_1,\dots,\xi_{K_n-1}) \in U_{K_n-1} , \: \|H( \phi_{\gamma}(\xi_1,\dots,\xi_{K_n-1})) - f_{0,K_n} \|_1 \leq \eps_n\}.
		\end{align} 
		For all integer $n$, denote $(\xi^0_1,\dots,\xi^0_{K_n-1})$ the element of $ U_{K_n-1}$ such that $H(\xi^0_1,\dots,\xi^0_{K_n-1})=f_{0,K_n}$. Let $(\xi_1,\dots,\xi_{K_n-1}) \in U_{K_n-1}$ such that $\|H(\xi_1,\dots,\xi_{K_n-1}) - f_{0,K_n} \|_1 \leq \tilde{\eps}_n \iff \sum_{i=1}^{K_n}|\xi_i- \xi^0_i| \leq \tilde{\eps}_n$. Then we have
		\begin{align*}
			\|H( \phi_{\gamma}(\xi_1,\dots,\xi_{K_n-1})) - f_{0,K_n} \|_1= \sum_{i=1}^{K_n}|\frac{\xi_i\gamma_i}{S_{\gamma}(\xi)} - \xi^0_i| &\leq \sum_{i=1}^{K_n}|\xi_i - \xi^0_i| + \sum_{i=1}^{K_n} |\frac{\xi_i\gamma_i}{S_{\gamma}(\xi)} -\xi_i|\\
			&\leq \tilde{\eps}_n + \sum_{i=1}^{K_n} \xi_i|\frac{\gamma_i}{S_{\gamma}(\xi)} -1|.
		\end{align*}
		By \eqref{calc_inter}, it follows that for all $i \in \{1,\dots, K_n\}$
		\begin{align*}
			\frac{e^{-\frac{|t|}{\sqrt{n\alpha_n}}\|\tilde{\psi}\|_\infty}}{e^{\frac{|t|}{\sqrt{n\alpha_n}}\|\tilde{\psi}\|_\infty}} \leq \frac{\gamma_i}{S_{\gamma}(\xi)} &\leq \frac{e^{\frac{|t|}{\sqrt{n\alpha_n}}\|\tilde{\psi}\|_\infty}}{e^{-\frac{|t|}{\sqrt{n\alpha_n}}\|\tilde{\psi}\|_\infty}} \\
			|\frac{\gamma_i}{S_{\gamma}(\xi)} -1 | &\leq e^{2\frac{|t|}{\sqrt{n\alpha_n}}\|\tilde{\psi}\|_\infty}-1.
		\end{align*}
		Hence $\|H( \phi_{\gamma}(\xi_1,\dots,\xi_{K_n-1})) - f_{0,K_n} \|_1 \leq \tilde{\eps}_n + (e^{2\frac{|t|}{\sqrt{n\alpha_n}}\|\tilde{\psi}\|_\infty}-1) = \eps_n.$ Thus we have the inclusion \eqref{inclusion}. Finally, combining \eqref{calc_chan_var_5} and \eqref{inclusion}, we obtain
		\begin{align*}
			(1+o_P(1))\int_{U_{K_n-1}} 1_{H(\xi_1,\dots,\xi_{K_n-1})\in \tilde{A}_n} e^{\alpha_nl_n(H(\xi_1, \dots, \xi_{K_n-1}))} \frac{1}{B(\delta)}\prod_{i=1}^{K_n}\xi_i^{\delta_i-1} d\xi_1\dots d\xi_{K_n-1} \\
			\leq \int_{A_n} e^{\alpha_nl_n(f_{t})} d\Pi(f) \leq (1+o_P(1))\int e^{\alpha_nl_n(f)}  d\Pi(f),
		\end{align*}
		hence
		\begin{align*}
			(1+o_P(1)) \Pi_{\alpha_n}(\tilde{A}_n|Y^n)
			\leq \frac{\int_{A_n} e^{\alpha_nl_n(f_{t})} d\Pi(f)}{\int e^{\alpha_nl_n(f)} d\Pi(f)} \leq (1+o_P(1)).
		\end{align*}
		By assumption \ref{assum_consistency} $\Pi_{\alpha_n}(\tilde{A}_n|Y^n)=1+o_P(1)$  and the result follows.
	\end{proof}

	\begin{lemma}\label{lem:bias}
	The following expansion holds for $\hat\psi$,
		\begin{align*}
			\hat{\psi} - \hat{\psi}_{[K_n]} = -F_0(\tilde{\psi}_{[K_n]}) +o_P(1/\sqrt{n}).
		\end{align*}
	\end{lemma}
	\begin{proof}{Proof} By definition,
		\begin{align}\label{conv}
			\hat{\psi} - \hat{\psi}_{[K_n]} &= \psi(f_0)+\frac{1}{n}\sum_{i=1}^{n} \tilde{\psi}(Y_i) - \psi(f_0) -\frac{1}{n}\sum_{i=1}^{n} \tilde{\psi}_{[K_n]}(Y_i) = \frac{1}{n}\sum_{i=1}^{n} (\tilde{\psi}(Y_i) - \tilde{\psi}_{[K_n]}(Y_i)  )\nonumber\\
			&= \frac{1}{n}\sum_{i=1}^{n}\Big( \tilde{\psi}(Y_i) - \tilde{\psi}_{[K_n]}(Y_i) + F_0(\tilde{\psi}_{[K_n]}) \Big) - F_0(\tilde{\psi}_{[K_n]}). 
		\end{align}
		Moreover, using that $E_0(\tilde{\psi}(Y_1) - \tilde{\psi}_{[K_n]}(Y_1)) = - F_0(\tilde{\psi}_{[K_n]})$,
		\begin{align}\label{expansion_bias}
			&E_0(\left(\frac{1}{n}\sum_{i=1}^{n} \Big(\tilde{\psi}(Y_i) - \tilde{\psi}_{[K_n]}(Y_i) + F_0(\tilde{\psi}_{[K_n]})\Big)\right)^2) \nonumber \\ 
            & =\frac{1}{n^2} E_0(\left(\sum_{i=1}^{n} \Big(\tilde{\psi}(Y_i) - \tilde{\psi}_{[K_n]}(Y_i) + F_0(\tilde{\psi}_{[K_n]}\Big)\right)^2) \nonumber\\
			&= \frac{1}{n} E_0(\Big(\tilde{\psi}(Y_1) - \tilde{\psi}_{[K_n]}(Y_1) + F_0(\tilde{\psi}_{[K_n]}\Big)^2) \leq \frac{1}{n} E_0(\Big(\tilde{\psi}(Y_1) - \tilde{\psi}_{[K_n]}(Y_1)\Big)^2) \nonumber\\
			&=\frac{1}{n}\int \left(\tilde{\psi} - \tilde{\psi}_{[K_n]}\right)^2 f_0  \leq \frac{\|f_0\|_{\infty}}{n}  \int \left(\tilde{\psi} - \tilde{\psi}_{[K_n]}\right)^2 =\frac{\|f_0\|_{\infty}}{n}  o(1) = o(\frac{1}{n}),
		\end{align}
		where we used the  assumption $K_n \rightarrow \infty$ in the last calculation. Hence, we obtain $\frac{1}{n}\sum_{i=1}^{n}\Big( \tilde{\psi}(Y_i) - \tilde{\psi}_{[K_n]}(Y_i) + F_0(\tilde{\psi}_{[K_n]})\Big) = o_P(1/\sqrt{n})$. Combining this with \eqref{conv}  yields the result.
	\end{proof}
	
%To prove Lemma \ref{prop:contraction_rate_an_posterior_Lmu}, we require the following two Lemmas from \cite{BIP2011}.
%
%\begin{lemma}[Lemma 8.1 in \cite{BIP2011}]\label{lemma:knapik_8.1}
%For any $1 \geq 0$, $t \geq -2q$, $u > 0$ and $v\geq 0$, as $N \rightarrow \infty$,
%$$
%\sup_{\|\zeta\|_q \leq 1}\sum_{i}\frac{\zeta_i^2i^{-t}}{(1+Ni^{-u})^v} \asymp N^{-(\frac{t+2q}{u} \wedge v)}.
%$$
%Moreover, for every fixed $\zeta \in S^q$, as $N \rightarrow \infty$, 
%$$
%N^{\frac{t+2q}{u}\wedge v}\sum_{i}\frac{\zeta_i^2i^{-t}}{(1+Ni^{-u})^v} \rightarrow \begin{cases}
%	0, & (t+2q/u) < v \\
%	\|\zeta\|^2_{(uv - t)/2}, &(t+2q)/u \geq v
%\end{cases}
%$$
%\end{lemma}
%
%\begin{lemma}[Lemma 8.2 in \cite{BIP2011}]\label{lemma:knapik_8.2}
%For any $t, v \geq 0$, $u > 0,$ and $(\zeta_i)$ such that $|\zeta_i| = i^{-q-1/2}$, for $q > -t/2$, as $n \rightarrow \infty$,
%$$
%\sum_{i}\frac{\zeta_i^2i^{-t}}{(1+Ni^{-u})^v} \asymp \begin{cases}
%	N^{-\frac{t+2q}{u}} & (t+2q/u) < v \\
%	N^{-v} & (t+2q/u) > v\\
%\end{cases}
%$$
%\end{lemma}
	
\subsection{Credible regions}
    \begin{lemma}\label{lemma_conv_quantiles}
		Let $(Q^Y_{n})$ be a sequence of random real distributions, $(u_n)$ a positive sequence, $(Y_n)$ a sequence of real random variables and $V$ a positive constant. Let $\delta \in (0,1)$ and denote $a_{n,\delta}^Y$ the random $\delta$-quantile of $Q^Y_{n}$. If
		\begin{align}\label{conv_loi}
			\tilde{Q}^Y_{n}:=u_n(Q^Y_{n} - Y_n)\overset{\cL}{\to} \mathcal{N}(0, V),
		\end{align}		
	then $u_n( a_{n,\delta}^Y - Y_{n}) \xrightarrow[]{P} \sqrt{V}q_{\delta}$.
	\end{lemma}
	
	\begin{proof}{Proof}
		By Lemma 2 in the supplement of \citet{CR2015}, \eqref{conv_loi} implies that 
		\begin{align*}
			\sup_{s \in \R} | \tilde{Q}_n^Y((-\infty , s]) - \mathcal{N}(0,V)((-\infty , s]) | \xrightarrow{P} 0.
		\end{align*}
		Using Lemma \ref{conv_alea_1} and that $u_n( a_{n,\delta}^Y - Y_n)$ is the $\delta$-quantile of $\tilde{Q}_n^Y$, we deduce that $u_n( a_{n,\delta}^Y - Y_n) \xrightarrow[]{P} \sqrt{V}q_{\delta}$.
	\end{proof}

	\begin{lemma}\label{conv_alea_1}
			Let $p\in (0,1)$. Let $(F_n)$ be a sequence of random cumulative distribution functions and $(q^n_p)$ the (random) sequence of its $p$-quantiles. Let $F$ be a fixed continuous increasing cumulative distribution function and $q_p$ be its $p$-quantile. If $\sup_{s \in \R} |F_n(s)-F(s)| \to^P 0$ as $n\to\infty$, then $|q^n_p - q_p| \to^P 0$ as $n\to\infty$.
	\end{lemma}
	\begin{proof}{Proof}
	For $\rho > 0$ arbitrary, we show that $P(|q_p-q_p^n| \leq \rho) \rightarrow 1$. Since $F$ is increasing and continuous, we have $F(q_p-\rho)<F(q_p)=p<F(q_p+\rho)$. Set $\eps =\min(F(q_p+\rho)-p, p - F(q_p-\rho))/2$. On the event $\{\sup_{s \in \R} |F_n(s)-F(s)|\leq \eps\}$ it follows 
			\begin{align*}
				F_n(q_p+\rho) \geq F(q_p+\rho)-\eps \geq F(q_p+\rho)-\frac{F(q_p+\rho)-p}{2} = p + \frac{F(q_p+\rho)-p}{2} > p  .
			\end{align*}
		By definition of the quantile $q_n^p$, this implies $q_n^p \leq q_p+\rho$. Similarly,
			\begin{align*}
			F_n(q_p-\rho) \leq F(q_p-\rho)+\eps \leq F(q_p-\rho)+\frac{p - F(q_p-\rho) }{2} = p - \frac{p - F(q_p-\rho)}{2} < p  .
			\end{align*}
			Hence $ q_p-\rho < q_n^p $, and thus it follows 
         $\{\sup_{s \in \R} |F_n(s)-F(s)|\leq \eps\} \subset \{|q^n_p - q_p|\leq \rho\}.$
		Let $\delta > 0$. Since $\sup_{s \in \R} |F_n(s)-F(s)|=o_P(1)$, there exists $N_0$ such that for all $n \geq N_0$,$P(\sup_{s \in \R} |F_n(s)-F(s)|\leq \eps) \geq 1-\delta$. Hence one deduces that for all $ n \geq N_0$, $1-\delta \leq P(\sup_{s \in \R} |F_n(s)-F(s)|\leq \eps) \leq P(|q^n_p - q_p|\leq \rho)$.
		\end{proof}

%%%%%%%%%%%%%%%%%%%%%%%%%%%%%%%%%%%%%%%%%%%%

\bibliography{23-0089.bib}

\end{document}